\documentclass[a4paper]{article}

\usepackage[english]{babel}
\usepackage[utf8x]{inputenc}
\usepackage[T1]{fontenc}

\usepackage[a4paper,top=3cm,bottom=2cm,left=3cm,right=3cm,marginparwidth=1.75cm]{geometry}

\usepackage{float}
\usepackage{enumitem}   
\usepackage{comment}
\usepackage{hyperref}       
\usepackage{url}            
\usepackage{booktabs}       
\usepackage{amsfonts}       
\usepackage{nicefrac}       
\usepackage{microtype}      
\usepackage{amsmath,amsfonts,amssymb}
\usepackage{bbm}
\usepackage[font=small,labelfont=small]{caption}
\usepackage{subcaption}
\usepackage{graphicx}
\usepackage[ruled,linesnumbered]{algorithm2e}
\usepackage[dvipsnames]{xcolor}
\usepackage{amsmath,amsfonts,amsthm} 
\usepackage{mathtools}
\usepackage{bm}
\usepackage{multirow}
\usepackage{multicol} 
\usepackage{color,soul}
\usepackage{pifont}

\newcommand{\cmark}{\ding{51}}%
\newcommand{\xmark}{\ding{55}}%

\usepackage{tikz}
\usetikzlibrary{decorations.markings,arrows}
\usetikzlibrary{automata, positioning, backgrounds, fit, shapes.misc}
\usetikzlibrary{calc}
\definecolor{left} {HTML}{001528}

  \newcommand{\boundellipse}[3]
{(#1) ellipse (#2 and #3)
}

\newcommand{\tobi}[1]{{\color{blue}tobias says: #1}}  
\newtheorem{definition}{Definition}[section]

\newtheorem{theorem}[definition]{Theorem}
\newtheorem{corollary}[definition]{Corollary}
\newtheorem{example}[definition]{Example} 
\newtheorem{Lemma}[definition]{Lemma} 

\newtheorem{assumption}{Assumption}

\newcommand{\Var}{\operatorname{Var}}

\newcommand{\const}{\operatorname{const}}
\newcommand{\Cov}{\operatorname{Cov}}

\newcommand{\prob}{\operatorname{Pr}}

\newcommand{\ra}[1]{\renewcommand{\arraystretch}{#1}}
\usepackage[hang,flushmargin]{footmisc}

\let\vec\bm

\title{Quasi Markov Chain Monte Carlo Methods}
\author{Tobias Schwedes and Ben Calderhead}

\begin{document}
\maketitle

\begin{abstract}
\noindent Quasi-Monte Carlo (QMC) methods for estimating integrals are attractive since the resulting estimators typically converge at a faster rate than pseudo-random Monte Carlo. However, they can be difficult to set up on arbitrary posterior densities within the Bayesian framework, in particular for inverse problems. We introduce a general parallel Markov chain Monte Carlo (MCMC) framework, for which we prove a law of large numbers and a central limit theorem. In that context, non-reversible transitions are investigated. We then extend this approach to the use of adaptive kernels and state conditions, under which ergodicity holds. As a further extension, an importance sampling estimator is derived, for which asymptotic unbiasedness is proven. We consider the use of completely uniformly distributed (CUD) numbers within the above mentioned algorithms, which leads to a general parallel quasi-MCMC (QMCMC) methodology. We prove consistency of the resulting estimators and demonstrate numerically that this approach scales close to $n^{-2}$ as we increase parallelisation, instead of the usual $n^{-1}$ that is typical of standard MCMC algorithms. In practical statistical models we observe multiple orders of magnitude improvement compared with pseudo-random methods.

\end{abstract}

\section{Introduction}

For many problems in science MCMC has become an indispensable tool due to its
ability to sample from arbitrary probability distributions known up only
to a constant. Comprehensive introductions on MCMC methods can be found in \cite{newman1999monte,liu2008monte,robert2004monte,landau2014guide}. Estimators resulting from MCMC scale independently of dimensionality. However, they have the fairly slow universal convergence rate of $n^{-1}$, where $n$ denotes the number of samples generated, in the mean squared error (MSE), same as classic Monte Carlo methods using pseudo-random numbers. 
For the latter, faster convergence rates of order close to $n^{-2}$ can be achieved when samples are generated by a suitable low-discrepancy sequence, i.e.\ points which are homogeneously distributed over space (\cite{dick2013high}). These so called quasi-Monte Carlo (QMC) methods, despite their generally deteriorating performance with increasing (effective) dimension (\cite{wang2003effective,caflisch1997valuation}), can nonetheless lead to significant computational savings compared to standard Monte Carlo. However, they generally require the integral of interest to be expressible in terms of an expectation with respect to a unit hypercube, which limits their general application.

The first applications of QMC in the context of MCMC go back to \cite{chentsov1967pseudorandom} and \cite{sobol1974pseudo}, which assume a discrete state space. In \cite{chentsov1967pseudorandom}, the driving sequence of uniformly distributed independent and identically distributed (IID) random numbers is replaced by a completely uniformly distributed (CUD) sequence. The same approach is used in \cite{owen2005quasi} and \cite{chen2011consistency}. In \cite{liao1998variance}, a Gibbs sampler that runs on randomly shuffled QMC points is introduced. Later, \cite{chaudhary2004acceleration} uses a weighting of rejected samples to generate balanced proposals. Both successfully applied QMC in MCMC, albeit without providing any theoretical investigation. \cite{craiu2007acceleration} uses QMC in multiple-try Metropolis-Hastings, and \cite{lemieux2006exact} within an exact sampling method introduced by \cite{propp1996exact}. In \cite{l2008randomized} the so called array-randomised QMC (RQMC) was introduced that uses quasi-Monte Carlo to update multiple chains that run in parallel. Further, the roter-router model, which is a deterministic analogue to a random walk on a graph, was applied in \cite{doerr2009deterministic} on a number of problems. We note that most of these approaches resulted in relatively modest performance improvements over non-QMC methods \cite{chen2011consistency}.

Based on the coupling argument by Chentsov from \cite{chentsov1967pseudorandom}, it was proven in \cite{owen2005quasi} that an MCMC method defined on a finite state space still has the correct target as its stationary distribution when the driving sequence of IID numbers is replaced by weakly CUD (WCUD) numbers. Subsequently, \cite{tribble2008construction} provided proofs of some theoretical properties of WCUD sequences, along with numerical results using a Gibbs sampler driven by WCUD numbers, which achieves significant performance improvements compared to using IID inputs. More recently, the result from \cite{owen2005quasi} was generalised to WCUD numbers and continuous state spaces by Chen (\cite{chen2011consistency}).

In this work, we consider the theoretical and numerical properties of 
the parallel MCMC method introduced in \cite{calderhead2014general}, 
which we here call multiple proposal MCMC (MP-MCMC) as it proposes 
and samples {\it multiple} points in each iteration. We extend this methodology to the use of non-reversible transition kernels and introduce an adaptive version, for which we show ergodicity. Further, we derive an importance sampling MP-MCMC approach, in which all proposed points from one iteration are accepted and then suitably weighted in order to consistently estimate integrals with respect to the posterior. We then combine these novel MP-MCMC algorithms with QMC by generalising them to use arbitrary CUD numbers as their driving sequence, and we establish conditions under which consistency holds. Due to the fact that the state space is covered by multiple proposals in each iteration, one might expect that using QMC numbers as the seed in MP-MCMC should harvest the benefits of low-discrepancy sequences more effectively than in the single proposal case previously considered.  Moreover, the importance sampling approach mentioned above enables MP-MCMC to remove the discontinuity introduced by the acceptance threshold when sampling from the multiple proposals, which improves the performance when using QMC numbers as the driving sequence.  Indeed, when combining the multiple proposal QMC approach together with the importance sampling method we observe in numerical simulations a convergence rate of order close to $n^{-2}$ for this novel MCMC method, similar to traditional QMC methods.

This work is, to the best of our knowledge, the first publication
showing substantial benefits in the use of QMC in MCMC for
arbitrary posteriors that are not known analytically and are not hierarchical, i.e.\ do not possess a lower-dimensional structure for their conditional probabilities. Hierarchical MCMC sampling problems using QMC for medium dimensions that have been considered in the literature
include the $11$-dimensional hierarchical Poisson model for pump failures 
from \cite{gelfand1990sampling}, which was treated via QMC Gibbs sampling 
methods by \cite{liao1998variance} and \cite{owen2005quasi}, respectively,
and a $42$-dimensional probit regression example from \cite{finney1947estimation}, treated in \cite{tribble2008construction} via the use of a QMC seed in a Gibbs sampling scheme introduced in \cite{albert1993bayesian}. In these problems however, conditional distributions are available explicitly such that direct sampling can be applied.

In this paper, we begin with re-defining the MP-MCMC algorithm previously introduced in \cite{calderhead2014general} and then consider a number of novel extensions, which result finally in a parallel CUD driven method that achieves a higher rate of convergence similar to QMC.  The list of novel algorithms we consider is presented in Table 1.  Throughout the paper, we also prove some theoretical results for the proposed algorithms as well as investigating their performance in practice.  For the purpose of clarity and readability, we will often state the lemma and refer the reader to the appropriate section in the appendix for its full proof.

\begin{table}[h]
\ra{1.2}
\centering
\caption{ 
Summary of algorithms introduced in this work, and associated properties}
\centering
\resizebox{.75\textwidth}{!}{

\begin{tabular}{  @{} *7c @{}}  \bottomrule
{Algorithm} & {Section} & {Adaptive} & {IS}
& {PSR} & {CUD} & {Num.\ conv.\ rate} \\ \midrule

\ref{algorithm:multiproposal_MH}    & 
\ref{subsec:derivation_mpmcmc}    & 
\xmark    &  \xmark  &     \cmark  &  \xmark    & $n^{-1}$ \\   

\ref{algorithm:adaptive_mp_mcmc}    &
\ref{subsubsec:an_adaptive_mpmcmc_algorithm}    &
\cmark    &  \xmark   &    \cmark   & \xmark     & $n^{-1}$\\ 

\ref{algorithm:importance_sampling_mp_mcmc}    &
\ref{subsubsec:algorithm_description_is_mpmcmc}     &
\xmark   &   \cmark  &    \cmark  &  \xmark     & $n^{-1}$\\ 

\ref{algorithm:adaptive_importance_sampling_mp_mcmc}    &
\ref{subsubsec:algorithm_description_adaptive_IS_mpmcmc}     &
\cmark    &  \cmark   &   \cmark   & \xmark    & $n^{-1}$\\ 

\ref{algorithm:multiproposal_quasi_MH}    &
\ref{subsubsec:algorithm_description_mpqmcmc}      &
\xmark     &  \xmark    &   \xmark   & \cmark     & $n^{-1}$\\ 

\ref{algorithm:importance_sampling_mp_qmcmc}    &
\ref{subsubsec:algorithm_description_IS_mpqmcmc}    &
\xmark     & \cmark    &    \xmark   & \cmark    & $\approx n^{-2}$ \\ 
      
\ref{algorithm:adaptive_importance_sampling_mp_qmcmc}    &
\ref{subsubsec:algorithm_description_adaptive_IS_mpqmcmc}		&
\cmark		 &  \cmark 	   &  \xmark  & \cmark		& $\approx n^{-2}$ \\ 
      
 \bottomrule

\end{tabular}
\label{table:results_bayesian_linear_regression}
}
\end{table}

In Section \ref{Section_Concepts_QMC} we introduce the basics of QMC and 
give a short review of the literature regarding CUD points, discuss some
CUD constructions and display the construction used in this work.

Next, in Section \ref{sec:multiple_proposal_mcmc}, we present the multiple proposal MCMC (MP-MCMC) framework from \cite{calderhead2014general}, first using pseudo-random numbers, and introduce two new formulations of MP-MCMC as a single state Markov chain over a product space, which we use for proving a number of theoretical properties.  We also formally prove a law of large numbers and central limit theorem for MP-MCMC and carefully consider a variety of novel extensions.  In particular, we consider the use of optimised and non-reversible transitions, as well as adaptivity of the proposal kernel, for which we prove ergodicity.  We then compare their relative performance through a simulation study.

In Section \ref{sec:importance_sampling} we consider the use of importance sampling within an MP-MCMC framework.  We suggest an adaptive version of this algorithm and prove its ergdocity, and consider the importance sampling approach as the limiting case of sampling from the finite state Markov chain on the multiple proposals.  We conclude by proving asymptotic unbiasedness of the proposed methods and empirically comparing their performance.

In Section \ref{sec:multiproposal_quasi_MH} we generalise the previously introduced MP-MCMC algorithms to the case of using CUD numbers as the driving sequence, instead of pseudo-random numbers.  We describe two regularity conditions that we then use to prove consistency of the proposed method, and we discuss how CUD numbers should be best incorporated within MCMC algorithms generally.  We prove asymptotic unbiasedness of the two proposed algorithms and demonstrate through a couple of numerical simulations an increased convergence rate of the empirical variance of our estimators, approaching $n^{-2}$ rather than usual $n^{-1}$ for traditional MCMC methods.

Finally, we present some conclusions and discuss the very many avenues for future work.

%
%


\section{Some Concepts from Quasi Monte Carlo}\label{Section_Concepts_QMC}

Quasi-Monte Carlo (QMC) techniques approximate integrals by an equal-weight quadrature rule similar to standard Monte Carlo. However, instead of using IID random samples as evaluation points, one uses low-discrepancy sequences designed to cover the underlying domain more evenly. Common choices for such sequences for QMC include Sobol sequences and digital nets \cite{dick2013high}.  
Due to the increased spatial coverage of the domain QMC generally yields better convergence rates than 
standard Monte Carlo.

\subsection{QMC background}

Standard QMC approaches generally permit the use of a high-dimensional hypercube as the domain of integration. However, using standard approaches, 
e.g., inverse transformations as introduced in \cite{devroye1986non}, 
samples from arbitrary domains may be constructed, as long 
as the inverse CDF is available. When the inverse of the CDF is not available
directly, one must resort to alternative sampling methods, which 
motivates the development of the MCMC methods later in this paper.

\subsubsection{Discrepancy and Koksma-Hlawka inequality}

Estimators based on QMC use a set of deterministic sample points 
$\vec{x}_i\in[0,1]^d$ for $i=1,...,n$ and $d \in \mathbb{N}$, that are members of a 
low-discrepancy sequence. Roughly speaking, these points are 
distributed inside $[0,1]^d$ such that the uncovered areas are 
minimised. Typically, the same holds true for 
projections of these points onto lower-dimensional faces of 
the underlying hypercube. Referring to \cite{niederreiter1992random},
for a set of QMC points $P=\{\vec{x}_1,...,\vec{x}_n\}$, 
the star discrepancy can be defined as
\begin{align}
D^{*d}_n(P) = \sup_{\vec{a} \in(0,1]^d} \left|\frac{1}{n}\sum_{i=1}^n \mathbf{I}_{(0,\vec{a}]}(\vec{x}_i) - \prod_{i=1}^n a_i \right|,
\label{eq:star_discrepancy}
\end{align}
where $\vec{a}$ has coordinates $0\le a_j \le 1$ for any $j=1,...,d$, 
respectively.  This gives us a measure of how well-spaced out a set 
of points is on a given domain. One of the main results 
in QMC theory, the Koksma-Hlawka inequality, provides an upper 
bound for the error of a QMC estimate based on \eqref{eq:star_discrepancy} by
\begin{align}
\left|\frac{1}{n}\sum_{i=1}^n f(\vec{x}_i) - \int_{[0,1]^d}f(\vec{x})\mathrm{d}\vec{x} \right| \le V\left(f\right) \cdot D^{*d}_n(P),
\label{eq:koksma_hlawka_inequality}
\end{align}
where $V(f)$ denotes the variation of $f$ in the sense of Hardy-Krause.
For a sufficiently smooth $f$, $V(f)$ can be expressed as the sum of terms
\begin{align}
\int_{[0,1]^k} \left| \frac{\partial^k f}{\partial x_{i_1} ... \partial x_{i_k}}
\right|_{x_j=1, j\neq i_1, ..., i_k} \mathrm{d}x_{i_1}...\mathrm{d}x_{i_k},
\end{align}
where $i_1 < ... < i_k$ and $k \le d$. A more general definition for
the case of non-smooth $f$ and in a multi-dimensional setting is beyond 
the scope of this work, but can be found in \cite{owen2005multidimensional}.
In \eqref{eq:koksma_hlawka_inequality}, $V(f)$ is assumed to be finite. 
Thus, the error of the approximation is deterministically bounded by a smoothness measure of the integrand and a quality measure for the point set. The Koksma-Hlawka equation \eqref{eq:koksma_hlawka_inequality} (for the case $d=1$) was first proven by Koksma \cite{koksma1942ageneral}, and the general case ($d\in \mathbb{N}$) was subsequently proven by Hlawka \cite{hlawka1961funktionen}. Note that for some functions arising in practise it holds $V(f) = \infty$, e.g.\ the inverse Gaussian map
from the hypercube to the hypersphere \cite{basu2016transformations}, so that
equation \eqref{eq:koksma_hlawka_inequality} cannot be applied.\\

\subsubsection{Convergence rates}

The use of low-discrepancy sequences instead of pseudorandom numbers may allow a faster convergence rate of the sampling error. Given an integrand $f$ with $V(f)<\infty$, constructions for QMC points can achieve convergence rates close to $\mathcal{O}(n^{-2})$ in the MSE, compared to $\mathcal{O}(n^{-1})$ for standard Monte Carlo \cite{dick2013high}. 
For smooth functions, it is possible to achieve convergence rates of order $\mathcal{O}(n^{-2\alpha} \log(n)^{2d\alpha})$ when $f$ is $\alpha$-times differentiable (\cite{dick2009quasi}). However, if $f$ has only bounded variation but is not differentiable, convergence rates of in general only $\mathcal{O}(n^{-2}\log(n)^{2d})$ hold true \cite{sharygin1963lower}. For practical applications where the dimensionality $d$ is large and the number of samples $n$ is moderate, QMC does therefore not necessarily perform better than standard Monte Carlo.
In some settings, using randomised QMC (RQMC) one can achieve convergence rates of 
$\mathcal{O}(n^{-3})$ \cite{l2006randomized, l2008randomized} in MSE, 
and $n^{-3}$ for the 
empirical variance in certain examples \cite{l2018sorting}.

\subsubsection{The curse of dimensionality}

The curse of dimensionality describes the phenomenon of 
exceeding increase in the complexity of a problem with the
dimensionality it is set in \cite{richard1957dynamic}.
Classical numerical integration methods such as quadrature
rules become quickly computationally infeasible to use when
the number of dimensions increases. This is since the number of
evaluation points typically increases exponentially with the the 
dimension, making such integration schemes impractical for
dimensions that are higher than say $d=6$.
However, in \cite{paskov1995faster} a high-dimensional ($d=360$)
problem from mathematical finance was successfully solved using
quasi-Monte Carlo (Halton and Sobol sequences). Since then, much research
has been undertaken to lift the curse
of dimensionality in QMC, referring to \cite{kuo2005lifting},
\cite{dick2010digital} and \cite{dick2013high}.

In general, a well-performing integration rule in a high-dimensional
setting will depend on the underlying integrand or a class of integrands.
\cite{caflisch1997valuation} introduced the notion of effective
dimension, which identifies the number of coordinates of a
function or indeed of a suitable decomposition (e.g.\ ANOVA), respectively,
which carries most of the information about the function.
This concept accounts for the fact that not all variables in a function
are necessarily informative about the variability in the function,
and may therefore be neglected when integrated. In practical applications 
the effective dimension can be very low ($d=2,3$) compared to the actual 
number of variables in the integrand.
To model such situations, weighted function spaces have been introduced
in \cite{sloan1998quasi}. In principle, the idea is to assign a weight
to every coordinate or to any subset of coordinates for a particular
decomposition of the integrand, thereby prioritising variables with high
degree of information on the integrand. Weighted function spaces have a
Hilbert space structure. For a particular class of such spaces, namely
reproducing kernel Hilbert spaces (RKHS), the worst-case error of the integration,
defined as the largest error for any function in the unit ball
of the RKHS, can be expressed explicitely in terms of the reproducing
kernel. Based on this, it is possible to prove the existence of 
low-discrepancy sets that provide an upper bound of the worst-case error
proportional to $N^{-1+\delta}$ for any $\delta >0$, where the constant 
does neither depend on $N$ nor on $d$. Furthermore, there exist
explicit constructions for such amenable point sets, e.g.\ the greedy
algorithm for shifted rank-1 lattice rules by \cite{sloan2002constructing} 
and the refined fast implementation based on Fast Fourier Transforms 
provided by \cite{nuyens2006fast}.
Modern quasi-Monte Carlo implementations can thus be useful in 
applications with up to hundreds and even thousands of dimensions.

The constructions of QMC point sets used for MCMC in this work are generic in the sense that
their construction does actually not depend on the underlying integrand.
Major performance gains compared to standard Monte Carlo
can still be achieved for moderately large dimensions, which we will see 
in Section \ref{sec:multiproposal_quasi_MH}. However, the incorporation of QMC constructions tailored to an inference problem solved by MCMC could a valuable future extension of this work.

\subsubsection{Randomised QMC}

Despite possibly far better convergence rates of QMC methods 
compared to standard MCMC, they
produce estimators that are biased and lack practical error 
estimates. The latter is due to the fact that evaluating the 
Koksma-Hlawka inequality requires not only computing the 
star discrepancy, which is an NP-hard problem \cite{gnewuch2009finding}, 
but also computing the total variation $V(f)$, which is generally 
even more difficult than integrating $f$. However, both drawbacks 
can be overcome by introducing a randomisation into the QMC
construction which preserves the underlying properties of the
QMC point distribution. For this task, there have been
many approaches suggested in the literature, such as shifting 
using Cranley-Patterson rotations 
\cite{cranley1976randomization}, digital shifting \cite{dick2013high}, 
and scrambling \cite{owen1997monte, owen1997scrambled}. In some
cases, randomisation can even improve the convergence rate of
the unrandomised QMC method, e.g.\ scrambling applied to digital nets
in \cite{owen1997scrambled, dick2011higher} under sufficient 
smoothness conditions. In these situations, the average of multiple QMC
randomisations yields a lower error than the worst-case QMC error.

\subsubsection{Completely uniformly distributed points}
\label{subsubsec:completely_uniformly_distributed_points}

Conceptually, QMC is based on sets of points which fill an underlying hypercube
homogeneously. Through suitable transformations applied to those points, samples 
are created which respresent the underlying target.
In constrast, MCMC relies on an iterative mechanism which makes use of ergodicity.
More presicely, based on a current state a subsequent state is proposed and 
then accepted or rejected, in such a way that the resulting samples represent 
the underlying target.
In that sense, QMC is about filling space, relying on equidistributedness,
while MCMC is about moving forward in time, relying on ergodicity.
Averages of samples 
can therefore be considered as space averages in QMC and time-averages in
MCMC, respectively. 

Standard MCMC works in the following way: based on a given 
$d$-dimensional sample, a new sample is proposed using $d$ IID random numbers
in $(0,1)$ using a suitable transformation. Then an accept/reject mechanism is employed, i.e.\ the proposed sample is accepted with
a certain probability, for which another random point in $(0,1)$ is required.
Thus, for $n$ steps we require $n(d+1)$ points, $u_1,...,u_{n(d+1)}\in (0,1)$.
The idea in applying QMC to MCMC is to replace the IID points
$u_i$, for $i=1,...,n(d+1)$, by more evenly distributed points.
There are two sources for problems connected to this approach: 
first, the sequence of states in the resulting process will not be Markovian, 
and thus consistency is not straightforward as 
the standard theory relies on the Markovian assumption. 
However, we know for instance from adaptive MCMC that even if 
the underlying method is non-Markovian ergodicity can still be proven
\cite{haario2001adaptive,haario2006dram, 
roberts2007coupling, latuszynski2013adaptive, andrieu2006ergodicity,
roberts2009examples, andrieu2008tutorial}.

Typically, computer simulations of MCMC are driven by a pseudo-random
number generator (PRNG). A PRNG is an algorithm that generates a sequence
of numbers which imitate the properties of random numbers. The generation procedure is
however deterministic as it is entirely determined by an initial value.
We remark that carefully considered, a sequence constructed by an MCMC
method using a PRNG does therefore actually not fulfill the 
Markov property either since the underlying seed is deterministic.
However, it is generally argued that given a good choice, a pseudo-random number 
sequence has properties that are sufficiently
similar to actual IID numbers as to consider the resulting algorithm
as probabilistic. A first formal criteria for a good choice of pseudo-random
numbers typically used in computer simulations
was formulated in Yao's test \cite{yao1982theory}. Roughly speaking,
a sequence of words passes the test if, given a reasonable computational
power, one is not able to distinguish from a sequence generated at random.
For modern versions
of empirical tests for randomness properties in PRNGs we
refer to the Dieharder test suite \cite{brown2017dieharder}
and the Test01 software library \cite{l2007testu01}. As an example,
the spacings of points which are selected according to the underlying PRNG
on a large interval are tested for being exponentially distributed. 
Asymptotically, this holds true for the spacings of truly randomly chosen 
points. 

A second source of problems in using QMC seeds in MCMC arises since MCMC is inherently sequential, which is a feature that QMC
methods generally do not respect. For example, the Van der Corput 
sequence (\cite{van1935b}), which will be introduced below, has been applied 
as a seed for MCMC in \cite{morokoff1993quasi}. In their example,
the first of an even number of heat particles, which are supposed to move
according to a symmetric random walk, always moves to the left, when sampled 
by the Van der Corput sequence. This peculiar behaviour occurs since,
although the VdC-sequence is equidistristributed over $[0,1]$, 
non-overlapping tupels of size $d= 2 m$ for $m\in \mathbb{N}$ are not 
equidistributed over $[0,1]^d$, as is shown later.

The convergence of a QMC method, i.e.\ the succesful integration of
a function on $\mathbb{R}^d$, relies on the equidistributedness of tupels
$(u_{(n-1)d+1}, ..., u_{nd})\in [0,1]^d$ for $n\rightarrow \infty$,
where $d$ is fixed. In order to prevent failure when using QMC in MCMC such as in \cite{morokoff1993quasi}, tupels of the form 
$(u_{(n-1)d'+1}, ..., u_{nd'}) \in [0,1]^{d'}$ must satisfy
equidistributedness for $n\rightarrow \infty$ while $d'$ is
variable. This naturally leads us to the definition
of CUD numbers: a sequence $(u_i)_{i}\subset [0,1]$ is called completely uniformly distributed (CUD) if for any $d\ge 1$ the points $\vec{x}^{(d)}_i = (u_i, \ldots, u_{i+d-1})\in [0,1]^d$ fulfill 
\begin{align*}
D_n^{*d}(\vec{x}_1^{(d)}, \ldots, \vec{x}_n^{(d)})\rightarrow 0, \quad \text{ as } \quad n\rightarrow \infty.
\end{align*}
In other words, any sequence of overlapping blocks of $u_i$ of size $d$ yield the desirable uniformity property $D_n^{*d}\rightarrow 0$ for a CUD sequence $(u_i)_{i\ge 1}$. It was shown in \cite{chentsov1967pseudorandom} that this is equivalent to any sequence of non-overlapping blocks of $u_i$ of size $d$ satisfying $D_n^{*d}\rightarrow 0$, i.e.\ 
\begin{align}
D_n^{*d}(\vec{\tilde{x}}_1^{(d)}, \ldots, \vec{\tilde{x}}_n^{(d)})\rightarrow 0, \quad \text{ as } \quad n\rightarrow \infty, 
\label{eq:cud_convergence_non_overlapping}
\end{align}
where $\vec{\tilde{x}}_i^{(d)}:=(u_{d(i-1)+1},\ldots, u_{di})\in [0,1]^d$.
In \cite{chen2011consistency}, Chen et al.\ prove that if in standard 
MCMC the underlying driving sequence of IID numbers is replaced by 
CUD numbers, the resulting algorithm consistently samples from the 
target distribution under certain regularity conditions. One can easily
show that every sequence of IID numbers is also CUD.

\vspace{4mm}
\noindent \textbf{Constructions in the literature}
\vspace{1.5mm}

\noindent There are a number of techniques to construct CUD 
sequences in the literature.
In \cite{levin1999discrepancy}, several constructions of CUD
sequences are introduced, but none of them amenable for actual
implementation \cite{chen2011consistencythesis}. In 
\cite{chen2011consistencythesis}, an
equidistributed linear feedback shift register (LFSR) sequence 
implemented by Matsumoto and Nishimura 
is used, which is shown to have the CUD property.
In \cite{owen2005quasi} the author uses a CUD sequence that 
is based on the linear congruential generator (LCG) developed 
in \cite{entacher1998quasi}. The lattice construction from 
\cite{niederreiter1977pseudo} and the shuffling strategy for 
QMC points from \cite{liao1998variance} are also both shown to 
produce CUD points in \cite{tribble2008construction}.
Furthermore, \cite{chen2012new} presents constructions of 
CUD points based on fully equidistributed LFSR, and antithetic and 
round trip 
sampling, of which the former we will use for our simulations later on.
The construction introduced in \cite{tribble2008construction} relies on 
a LCG with initial seed $1$ and increment equal to zero. For a given
sequence length, a good multiplier is found by the primitive roots values 
displayed in \cite{l1999tables}.

\vspace{4mm}
\noindent \textbf{Illustration of a CUD sequence}
\vspace{1.5mm}

\noindent As an illustration, we display in Figure 
\ref{fig:cuds_vs_psr} an implementation of the
the CUD construction, which was 
introduced in \cite{chen2012new} and relies 
on a LFSR with a transition mechanism based on primitive polynomials 
over the Galois field $GF(2)$. The resulting sequence is visually more 
homogeneously distributed than that generated using pseudo-random numbers.
For a complete description of the construction, sufficient for a reader
to implement the method themselves, we refer to section 3 in 
\cite{chen2012new}. Additionally, we provide our own Python implementation 
of this CUD generator in \cite{tobias_schwedes_2018_1255042}, as well as the one introduced by \cite{tribble2008construction}.

\vspace{4mm}
\noindent \textbf{Construction used in this work}
\vspace{1.5mm}

\noindent 
Given a target defined on a $d$-dimensional space we employ a technique
of running through a generated CUD sequence $d$ times, similar to
\cite{owen2005quasi}, thereby creating tupels of size $d$. The
resulting tupels are pairwise different from each other and every
tupel is used exactly once in the simulation.
Similarly to \cite{owen2005quasi}, we prepend
a tupel of values close to zero to the resulting tupel sequence, 
imitating the property of an integration lattice containing a point at the 
origin.\\
More precisely, we use the CUD construction based on section 3 in \cite{chen2012new}, which 
creates sequences of length $L=2^m-1$ for integers $10 \le m \le 32$. Given a
sequence $u_1,...,u_L \in (0,1)$ and dimensionality $d$, we cut off the
sequence at $T:= \lfloor L/d \rfloor \cdot d \le L$, leading to the trimmed 
sequence $u_1,...,u_T$. Since $L-T \le d \ll T,L$, trimming has no relevant influence
on the outcomes of simulations. In order to make efficient use of the generated
sequence, we generate tupels of size $d$ and of the form
\begin{align*}
&(u_1,...,u_d), (u_{d+1}, \ldots, u_{2d}), ..., (u_{T-d+1}, ..., u_T),\\
&(u_2,...,u_{d+1}), (u_{d+2}, \ldots, u_{2d+1}), ..., (u_{T-d+2}, ..., u_T, u_1),\\
&...\\
&(u_d,...,u_{2d-1}), (u_{2d}, \ldots, u_{3d-1}), ..., (u_{T}, u_1, ..., u_{d-1}).
\end{align*}
The sequence of points $v_n$, $n=1,...,dT$, given by 
$u_{1}, ..., u_T$, $u_2, ..., u_T, u_1$, $...$, $u_{d}, ..., u_T, u_1, ..., u_{d-1}$,
still satisfies the CUD property. This is true since the shifting of indices
in $u_1,...,u_T$ to $u_{k+1},...,u_T, u_1, ..., u_k$ for any $k\in \mathbb{N}$
does not influence the CUD property. Further, appending a CUD sequence
to another CUD sequence of the same length preserves the CUD property,
too. Finally, prepending a single tupel of size $d$ does 
not affect the CUD property for overlapping tupels of size $d$.

\begin{figure}[h]
    \centering
    \begin{subfigure}[b]{0.49\textwidth}
        \includegraphics[width=\textwidth]{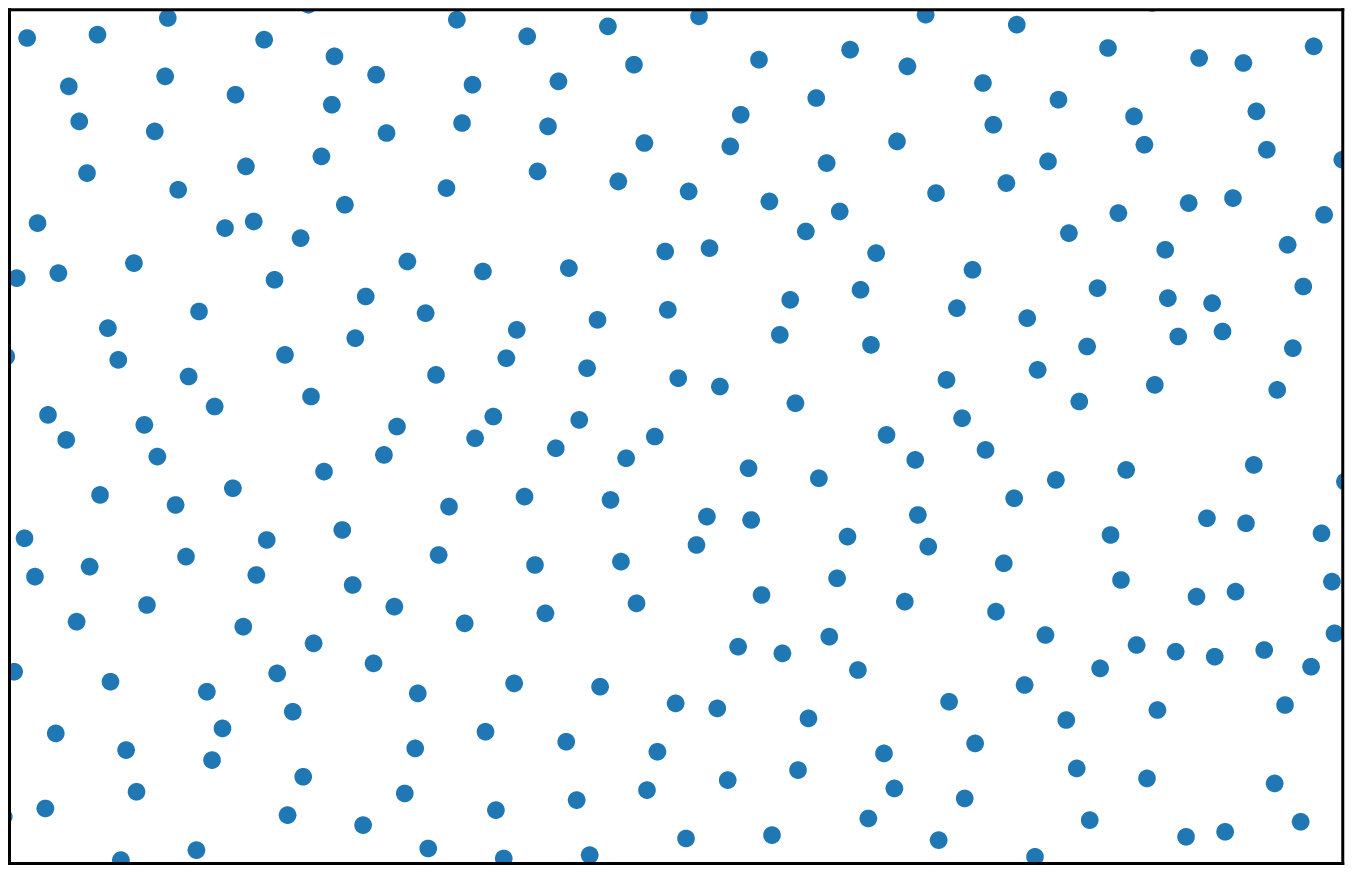}
        \caption{A completely uniformly distributed sequence.}
    \end{subfigure} 
    \begin{subfigure}[b]{0.49\textwidth}
    \includegraphics[width=\textwidth]{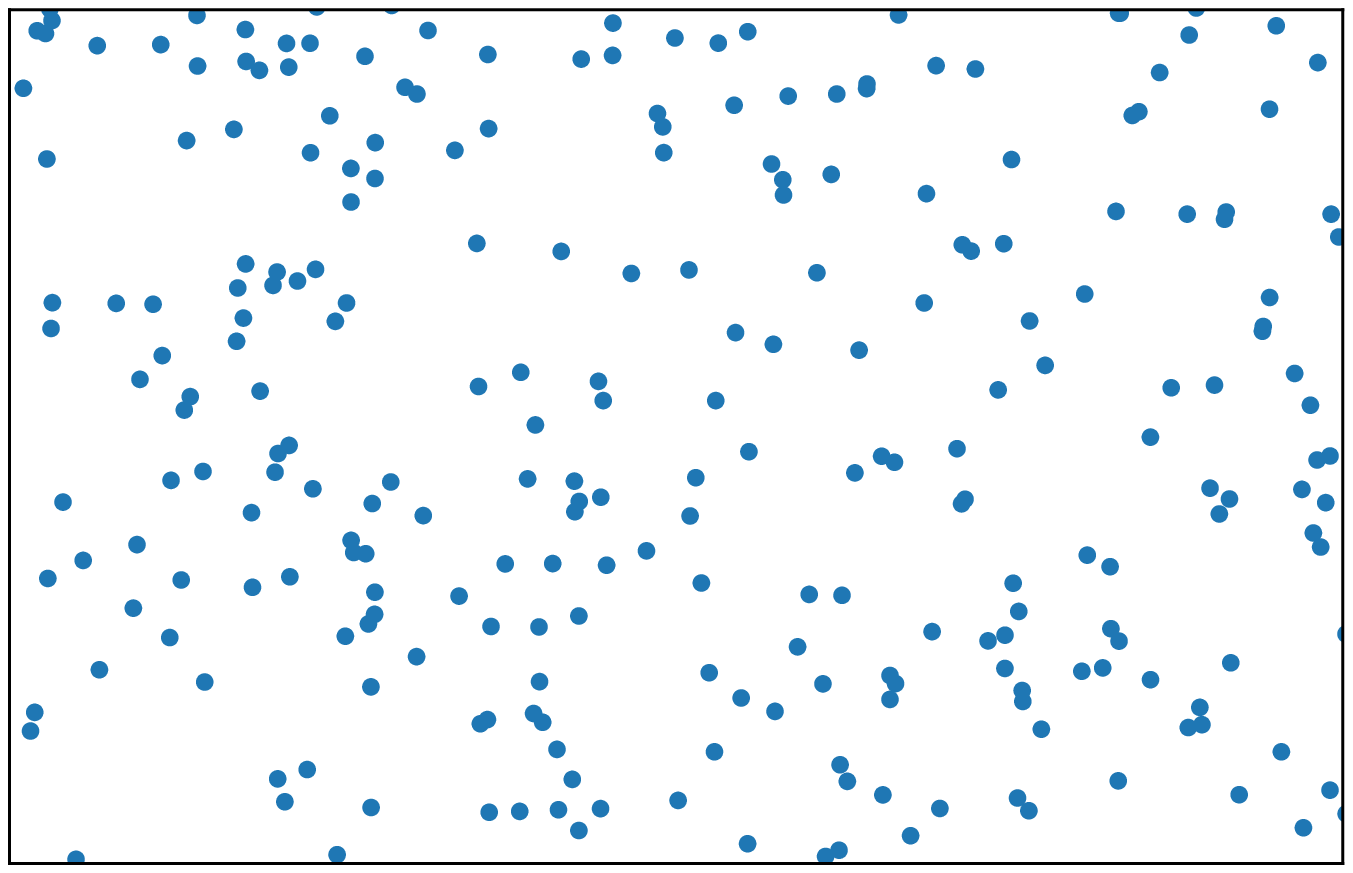}
        \caption{A pseudo-randomly generated sequence.}
    \end{subfigure} 
    \caption{\small{Segments of CUD and pseudo-random finite sequences in $(0,1)^2$.}}
    \label{fig:cuds_vs_psr}
\end{figure}

\vspace{4mm}
\noindent \textbf{QMC sequences are generally not CUD}
\vspace{1.5mm}

\noindent 
Finally we note that care must be taken in the choice of the QMC sequence 
applied to MCMC 
since not every low discrepancy sequence is a CUD sequence. For the 
\textit{van der Corput sequence} $(u_n)_{n}$ (\cite{van1935b}), any
$u_{2n}\in (0,1/2)$ and any $u_{2n-1}\in [1/2,1)$ 
for all $n\ge 1$. Thus, 
\begin{align}
\vec{x}^{(2)}_{2n} &\in (0,1/2)\times [1/2,1), \text{ and }\\
\vec{x}^{(2)}_{2n-1} &\in [1/2,1)\times (0,1/2).
\end{align}
Therefore, the sequence of overlapping tupels $\vec{x}^{(2)}_{n}$ never hits 
the square $(0,1/2)\times(0,1/2)$, which implies $D^{*2}_n \ge 1/4$ for any 
$n$. Note that the same holds true for non-overlapping tupels
$\tilde{\vec{x}}^{(2)}_{n}$. Hence, the van der Corput sequence is 
not a CUD sequence.


\section{Pseudo-random MP-MCMC}
\label{sec:multiple_proposal_mcmc}

In \cite{calderhead2014general}, a natural generalisation of the 
well-known Metropolis-Hastings algorithm (\cite{hastings1970monte}) 
that allows for parallelising a single chain is achieved by 
proposing multiple points in parallel. In every MCMC iteration, 
samples are drawn from a finite state Markov chain on the proposed 
points, which is constructed in such a way that the overall 
procedure has the correct target density $\pi$ on a state space
$\Omega \subset\mathbb{R}^d$ for $d \in \mathbb{N}$, as its 
stationary distribution. 
In this section we introduce 
this algorithm, demonstrate that this approach mathematically 
corresponds to a Metropolis-Hastings algorithm over a product 
space and prove its consistency, as well as some asymptotic 
limit theorems, before considering how to extend this algorithm 
to improve its sampling performance.

\subsection{Derivation}
\label{subsec:derivation_mpmcmc}

Before presenting the MP-MCMC algorithm we first note that any joint probability distribution $p(\vec{y}_{1:N+1})$, where $\vec{y}_{1:N+1}=\vec{y}_{[1:N+1]}=(\vec{y}_1,...,\vec{y}_{N+1})$ with $\vec{y}_i\in \Omega$ $\forall i=1,...,N+1$,
can be factorised in $N+1$ different ways, using conditional probabilities of the form, 
$p(\vec{y}_{1:N+1}) = p(\vec{y}_i)p(\vec{y}_{\setminus i}|\vec{y}_i)$,
where $\vec{y}_{\setminus i}:=\vec{y}_{[1:i-1,i+1:N+1]}$.
If the target $\pi$ is the marginal distribution for $\vec{y}_i$ of 
$\vec{y}_{1:N+1}\sim p$ and any $i=1,...,N+1$, then
\begin{align*}
p(\vec{y}_{1:N+1}) = \pi(\vec{y}_i)\kappa (\vec{y}_i, \vec{y}_{\setminus i}),
\end{align*}
for a proposal distribution $\kappa$ satisfying 
$\kappa(\vec{y}_i, \vec{y}_{\setminus i})\equiv p(\vec{y}_{\setminus i}|\vec{y}_i)$.
Thus, in the $i$th factorisation, $\vec{y}_i \sim \pi$, 
while the other $\vec{y}_{\setminus i}\sim \kappa(y_i, \cdot)$ 
Referring to \cite{tjelmeland2004using, calderhead2014general}, a uniform auxiliary variable
$I\in \{1,...,N+1 \}$ can be introduced that determines which factorisation is used,
such that 
\begin{align}
p(\vec{y}_{1:N+1}, I=i)
= \frac{1}{N+1}\pi(\vec{y}_i)\kappa(\vec{y}_i, \vec{y}_{\setminus i}).
\label{eq:factorisation_derivation_mpmcmc}
\end{align}

\begin{algorithm}[h]	
\SetAlgoLined
\KwIn{
Initialise starting point $\vec{x}_0=\vec{y}_1\in \Omega\subset \mathbb{R}^d$, 
number of proposals $N$, number of
accepted samples per iteration $M$, auxiliary 
variable $I=1$ and counter $n=1$\;}
\For{\textnormal{each MCMC iteration $\ell=1,2,...$}}{
 Sample $\vec{y}_{\setminus I}$ conditioned on $I$, i.e., draw $N$ new points from the proposal kernel $\kappa(\vec{y}_I, \cdot) = p(\vec{y}_{\setminus I}|\vec{y}_I)$ \;
  Calculate the stationary distribution of $I$ conditioned on $\vec{y}_{1:N+1}$,
  i.e.\ $\forall$ $i=1,...,N+1$, $p(I=i|\vec{y}_{1:N+1}) = \pi(\vec{y}_i)\kappa(\vec{y}_{{i}}, \vec{y}_{\setminus{i}}) / \sum_j \pi(\vec{y}_j)\kappa(\vec{y}_{{j}}, \vec{y}_{\setminus{j}})$, 
  which can be done in parallel\; 
 \For{$m=1,...,M$}{
 Sample new $I$ via the stationary distribution $p(\cdot|\vec{y}_{1:N+1})$\; 
 Set new sample $\vec{x}_{n+m} = \vec{y}_I$\;
}
 Update counter $n=n+M$
 }
\caption{Multiple-proposal Metropolis-Hastings}
 \label{algorithm:multiproposal_MH}
\end{algorithm}

\subsubsection{A Markov chain over a product space}
\label{subsubsec:markov_chain_over_product_space}

The MP-MCMC method generates $M\in \mathbb{N}$ new samples 
per iteration, and
can be considered as a single Markov chain over
the product space of proposal and auxiliary variables 
$(\vec{y}_{1:N+1}, I_{1:M})\in \Omega^{N+1}\times \{1,...,N+1\}^M$ 
by applying a combination 
of two transition kernels, each of which preserves the underlying joint stationary distribution. First, the states of a finite state Markov chain are created by updating the proposals $\vec{y}_{\setminus{i}}$ conditioned on $\vec{y}_{i}$ and $I_M=i$, which clearly preserves the joint target distribution as we sample directly from $\kappa(\vec{y}_i, \cdot)$; this is equivalent to a Gibbs sampling step. The choice of the proposal kernel is up to the practitioner, and kernels based on Langevin diffusion and Hamiltonian dynamics have successfully been applied (\cite{calderhead2014general}). 
Secondly, $I_m$ conditioned on $\vec{y}_{1:N+1}$ and $I_{m-1}$
is sampled $M$ times, i.e.\ for $m=1,...,M$,
using a transition matrix $A$, where $A(i,j)=A(i,j|\vec{y}_{1:N+1})$ denotes the 
probability of transitioning 
from $I_{m-1}=i$ to $I_m=j$. Here, $I_0$ denotes the $M$th (i.e. last) sample of $I$,
i.e.\ $I_M$, from the previous iteration.
An illustration of this procedure
is given in Figure \ref{fig:illustration_mpmcmc_single_chain}.
Using the factorisation from
\eqref{eq:factorisation_derivation_mpmcmc}, the joint distribution of 
$(\vec{y}_{1:N+1}, i_m)$ for $I_m=i_m$ denoting the $m$th sample of $I$
and $m=1,...,M$, can be expressed as,
\begin{align*}
p(\vec{y}_{1:N+1}, I_m=i_m) = \frac{1}{N+1} \pi(\vec{y}_{i_m}) \kappa(\vec{y}_{i_m}, \vec{y}_{\setminus{i_m}}).
\end{align*}
Observe that $\vec{y}_{i_m}$ has the correct density $\pi(\vec{y}_{i_m})$
for any $m=1,...,M$.
Thus, those are the samples 
we collect in every iteration. For the particular case where 
$A(i,j)=p(I=j | \vec{y}_{1:N+1})$, independent of $i$, denotes the stationary 
transition matrix on the states $I_1,...,I_M$ given $y_{1:N+1}$, the entire 
procedure described here 
is given in Algorithm \ref{algorithm:multiproposal_MH}. Note that, for the 
sake of clarity, we make a distinction between proposals $\vec{y}_{1:N+1}$ from
one iteration, and the accepted samples $\vec{x}_{n}$ with $n\in \mathbb{N}$.

\begin{figure}
\centering
\resizebox{\linewidth}{!}{
\begin{tikzpicture}[scale=.33, font=\sffamily, dot/.style = 
					{state, fill=gray!20!white, line width=0.01mm, 
					inner sep=1pt, minimum size=0.1pt, minimum width=0.02cm},
					>=triangle 45]

        \node at (0, -2)   (x1) 	  {};
        
        \node at (5, 10)   (yiphantom) 	  {};        
        \node at (7.5, 10)  (yi) 	  {\text{\small ${y}^{(i)}_{1:N+1}$}};
		\node at (16, 7.5) (Ii1)   {\text{\small $I_{1}^{(i)}|{y}^{(i)}_{1:N+1}$}};
		\node[minimum size=20pt] at (16, 3) (Ii2)   	{$...$};
		\node at (16, -1.5) (IiM)   {\text{\small $I_{M}^{(i)}|{y}^{(i)}_{1:N+1}$}};

		\draw[ black, line width=0.05mm] [->]  (x1) -- (yiphantom)
		node[midway, color=black, below right= -0.2cm] 
		{\text{\small $\kappa({y}^{(i-1)}_{I_M^{(i-1)}}, \cdot )$}};

		\draw[ black, line width=0.05mm] [->]  (yi) -- (Ii1) ;
		\draw[ black, line width=0.05mm] [->]  (Ii1) -- (Ii2)
		node[midway, color=black, left] {\text{\scriptsize $A(I^{(i)}_1, I^{(i)}_2)$}};
		\draw[ black, line width=0.05mm] [->]  (Ii2) -- (IiM)
		node[midway, color=black, left] {\text{\scriptsize $A(I^{(i)}_{M-1}, I^{(i)}_M)$}};

        \node at (20, -2)   (x2) 	  {};
        \node at (25, 10)   (yiplus1phantom) 	  {};        
        \node at (27.5, 10)  (yiplus1) 	  {\text{\small ${y}^{(i+1)}_{1:N+1}$}};
		\node at (36, 7.5) (Iiplus11)   {\text{\small $I_{1}^{(i+1)}|{y}^{(i+1)}_{1:N+1}$}};
		\node[minimum size=20pt] at (36, 3) (Iiplus12)   	{$...$};
		\node at (36, -1.5) (Iiplus1M)   {\text{\small $I_{M}^{(i+1)}|{y}^{(i+1)}_{1:N+1}$}};

		\draw[ black, line width=0.05mm] [->]  (x2) -- (yiplus1phantom)
		node[midway, color=black, below right= -0.2cm] 
		{\text{\small $\kappa({y}^{(i)}_{I_M^{(i)}}, \cdot )$}};
		\draw[ black, line width=0.05mm] [->]  (x2) -- (yiplus1phantom) ;
		\draw[ black, line width=0.05mm] [->]  (yiplus1) -- (Iiplus11) ;
		\draw[ black, line width=0.05mm] [->]  (Iiplus11) -- (Iiplus12) 
		node[midway, color=black, left] {\text{\scriptsize $A(I^{(i+1)}_1, I^{(i+1)}_2)$}};
		\draw[ black, line width=0.05mm] [->]  (Iiplus12) -- (Iiplus1M)
		node[midway, color=black, left] {\text{\scriptsize $A(I^{(i+1)}_{M-1}, I^{(i+1)}_M)$}};

        \node at (40, -2)   (x3) 	  {};
        \node at (45, 10)   (yiplus2phantom) 	  {};        

		\draw[ black, line width=0.05mm] [->]  (x3) -- (yiplus2phantom)
		node[midway, color=black, below right= -0.2cm] 
		{\text{\small $\kappa({y}^{(i+1)}_{I_M^{(i+1)}}, \cdot )$}};

		\node at (13, -4.5) (iter_i)   {$i$th iteration};
		\node at (33, -4.5) (iter_i)   {$(i+1)$th iteration};

    \begin{scope}
    [on background layer]{\fill[rounded corners= 10pt, shading = axis, left color=white, right color=black!20,black!60,thick,dotted,fill=black!20] ($(-5, 11.5)$)  rectangle ($(0, -3)$);}    
    [on background layer]{\draw[rounded corners= 10pt, black!60,thick,dotted,fill=black!20] ($(5, 11.5)$)  rectangle ($(20, -3)$);}
    [on background layer]{\draw[rounded corners= 10pt,black!60,thick,dotted,fill=black!20] ($(25, 11.5)$)  rectangle ($(40., -3)$);}
    [on background layer]{\fill[rounded corners= 10pt,shading = axis, left color=black!20, right color=white,black!60,thick,dotted,fill=black!20] ($(45, 11.5)$)  rectangle ($(50, -3)$);}
    \end{scope}

        \node at (40, -2)   (x3) 	  {};

\end{tikzpicture} 
}

\caption{Visualisation of MP-MCMC as a single Markov chain on the product space of proposal samples and auxiliary variables}
\label{fig:illustration_mpmcmc_single_chain}
\end{figure}
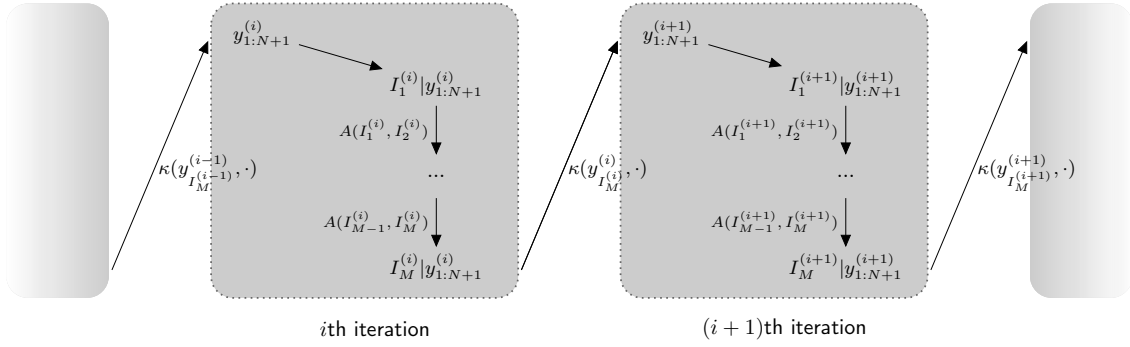

Considering MP-MCMC as a Markov chain over the product space of proposals
and auxiliary variables has the advantage that a formula for the
transition kernel of the resulting chain can be derived.
This is useful for later statements considering ergodicity of
adaptive versions of MP-MCMC in Section \ref{subsec:adaptive_mpmcmc}, 
which require computations on the transition probabilities.

\subsubsection{Transition probabilities on the product space}
\label{subsubsec:transition_probabilities_product_space}

An explicit form for the transition kernel density $\hat{P}(\tilde{z}, z)$, 
from state $\tilde{z}=(\tilde{\vec{y}}_{1:N+1}, \tilde{i}_{1:M})$
to state $z=(\vec{y}_{1:N+1}, i_{1:M})$, where $\tilde{z},z \in 
\Omega^{N+1}\times \{1,...,N+1\}^M$, is given by,
\begin{align}
\hat{P}(\tilde{z}, z) 
= \kappa(\vec{y}_{i_0}, \vec{y}_{\setminus{i_0}}) \prod_{m=1}^M A(i_{m-1}, {i_m})
\label{eq:last_eq_transition_kernel_mp_mcmc}
\end{align}
where we implicitly used that $i_0 = \tilde{i}_M$ and
$\vec{y}_{i_0}=\tilde{\vec{y}}_{\tilde{i}_M}$. Waiving the latter assumption,
we need to add 
the term $\delta_{\tilde{\vec{y}}_{\tilde{i}_M}}(\vec{y}_{i_0})$ to the 
expression of the transition kernel.
A more thorough derivation of equation
\eqref{eq:last_eq_transition_kernel_mp_mcmc}, as well as the subsequent 
equation \eqref{eq:transition_kernel_set_mp_mcmc1}, is presented in Appendix
\ref{appendix:transition_probabilities_product_space}.
Let us introduce the notation $B_{1:{n}}=B_1 \times ... \times B_{n}$ for any
sets $B_1, ..., B_n$ and $n \in \mathbb{N}$.  
Let $B\in \mathcal{B}(\Omega^{N+1}) \times \mathcal{P}(\{1,...,N+1\}^M)$ such
that $B= C_{1:{N+1}} \times D_{1:M} \subset \Omega^{N+1}\times \{1,...,N+1\}^M$,
where $\mathcal{P}(\{1,...,N+1\}^M)$ denotes the power set of $\{1,...,N+1\}^M$.
The probability $\prob(z \in B |\tilde{z})=\hat{P}(\tilde{z}, B)$ of 
a new state $z=(\vec{y}_{1:N+1},i_{1:M})\in B$ given a current 
state $\tilde{z}=(\tilde{\vec{y}}_{1:N+1}, \tilde{i}_{1:M})$ with $\tilde{i}_M=i_0$ 
can be expressed as,
\begin{align}
\hat{P}(\tilde{z}, B) = \chi_{C_{i_0}(\tilde{\vec{y}}_{\tilde{i}_M})}\int_{C_{\setminus{i_0}}}\kappa(\vec{y}_{i_0}=\tilde{\vec{y}}_{\tilde{i}_M},  \vec{y}_{\setminus{i_0}}) \sum_{i_{1:M}\in D_{1:M}} \prod_{m=1}^M A(i_{m-1}, i_m| \vec{y}_{i_0}=\tilde{\vec{y}}_{\tilde{i}_M})
\mathrm{d}\vec{y}_{\setminus{i_0}},
\label{eq:transition_kernel_set_mp_mcmc1}
\end{align}
where we use the notation 
$B_{\setminus{i}} = B_1\times ...\times B_{i-1} \times B_{i+1} \times .. \times B_n$ 
for any sets $B_1, ..., B_n$ and $i=1,...,n\in \mathbb{N}$. Note that conditioning
on $\tilde{z}$ in \eqref{eq:transition_kernel_set_mp_mcmc1} reduces to conditioning on the last accepted
sample $\tilde{\vec{y}}_{\tilde{i}_M}$, which is the only sample of relevance
for the subsequent iteration. Thus, 
the domain of the transition kernel $P$ can be reduced to 
$\Omega$ by using the identification 
$\hat{P}(\tilde{\vec{y}}_{\tilde{i}_M}, B)\equiv \hat{P}(\tilde{z}, B)$.

\subsubsection{Equivalence of MP-MCMC to a chain on the accepted samples}
\label{subsubsec:equivalence_mpmcmc}

An alternative representation of MP-MCMC as a Markov chain over the product space
of proposals and auxiliary variables $(\vec{y}_{1:N+1}, I_{1:M})$ in one 
iteration is to understand it as a chain over the space of accepted samples
$(\vec{x}_1, ..., \vec{x}_M)$ in one iteration. Indeed, given the current 
accepted set of samples, any set of samples generated in a future iteration
are independent of the past iterations. This representation is useful since
it allows to see MP-MCMC as a Markov chain over a single real space, which will be
used to prove limit theorems in Section \ref{subsection:limit_theorems}. 
Further, explicit transition probabilities for this representation are derived 
in what follows, which are then used to prove ergodicity statements of adaptive 
versions of MP-MCMC in Section \ref{subsec:adaptive_mpmcmc}. Note that 
since $\vec{x}_i \in \Omega$ for any $i=1,...,M$, we have 
$(\vec{x}_1, ..., \vec{x}_M) \in \Omega^M$.

\subsubsection{Transition probabilities on the accepted samples}
\label{subsubsec:trans_probs_sample_state_space}

Clearly, we would like to have an expression for the transitions between
actual accepted states rather than proposal and auxiliary variables.
It is possible to derive from the
transition kernel $\hat{P}$, corresponding to the states of proposals and auxiliary variables $(\vec{y}_{1:N+1},I_{1:M})$, the transition kernel $P$, corresponding to only the actually accepted states $\vec{x}_{1:M} \in \Omega^M$, 
i.e.\ where $\vec{x}_m = \vec{y}_{I_m}$ for $m=1, ..., M$. 
The probability of accepted states
$\vec{x}_{1:M}\in B=B_{1:M}\in \mathcal{B}({\Omega^M})$ given a
previous state $\tilde{\vec{x}}_{1:M}$ can be  expressed in terms of the 
transition kernel $\tilde{P}$ on the model state space of proposals 
and auxiliary variables from \eqref{eq:transition_kernel_set_mp_mcmc1} 
as follows,
\begin{align}
{P}(\tilde{\vec{x}}_{1:M}, B) &= \hat{P}\left( \tilde{\vec{x}}_M, 
\bigcup_{i_1,...,i_M=1}^{N+1}\ \bigcap_{m=1}^M S_{i_m}(B) \right)
\label{eq:mp_mcmc_transition_kernel_general_case1a}
\end{align}
where
\begin{align}
S_{i_m}(B)=\Omega^{d(i_m-1)}\times B_m \times \Omega^{d(N+1-i_m)} \times \{1:N+1\}^{m-1} \times \{i_m\}\times \{1:N+1 \}^{M-m}
\label{eq:Sim_set}
\end{align}
for any $i_m=1,...,N+1$ and $m=1, ..., M$. 
The sets $S_{i_m}(B)$ are pairwise disjoint. A comprehensive derivation 
of equation \eqref{eq:mp_mcmc_transition_kernel_general_case1a} can be 
found in Appendix \ref{appendix:transition_probabilities_accepted_samples}.

\subsection{Consistency}
\label{subsection:consistency}

The question then arises how to choose the transition probabilities $A(i,j)$, for any $i,j=1,...,N+1$, such that the target distribution $\pi$ for accepted
samples is preserved. 
Given the transition matrix $[A(i,j)]_{i,j}$ on the finite states $i,j$ 
determining the transitions of $I$, and using the factorisation in 
\eqref{eq:factorisation_derivation_mpmcmc}, the detailed balance condition for 
updating $I$ in $(\vec{y}_{1:N+1}, I)$ is given by
\begin{align*}
\frac{1}{N+1}\pi(\vec{y}_i)\kappa(\vec{y}_i, \vec{y}_{\setminus{i}}) A(i,j) 
= \frac{1}{N+1}\pi(\vec{y}_i)\kappa(\vec{y}_i, \vec{y}_{\setminus{i}}) A(j,i),
\end{align*}
for all $i,j=1,...,N+1$.
Clearly, the detailed balance condition implies the balance condition,
\begin{align}
\frac{1}{N+1} \pi(\vec{y}_i)\kappa(\vec{y}_{i}, \vec{y}_{\setminus{i}})
= \sum_{j=1}^{N+1} \pi(\vec{y}_j) \kappa(\vec{y}_{j}, \vec{y}_{\setminus{j}}) A(j,i),
\label{eq:balance_condition_mpmcmc}
\end{align}
for any $i=1,...,N+1$. If \eqref{eq:balance_condition_mpmcmc} holds true, 
the joint distribution $p(\vec{y}_{1:N+1}, I)$
is invariant if $I$ is sampled using the transition matrix $A$.
We say that the sequence $(\vec{x}_i)_i$ of an MP-MCMC algorithm consistently 
samples $\pi$, if
\begin{align}
\lim_{n\rightarrow \infty} \frac{1}{n} \sum_{i=1}^n f(\vec{x}_i) 
= \int_\Omega f(\vec{x}) \pi(\vec{x}) \mathrm{d}\vec{x},
\label{eq:consistency}
\end{align}
for any continuous bounded $f$ on $\Omega \subset \mathbb{R}^d$. 
We have chosen this notion of consistency to represent the fact
that implicitely the formulation of low-discprenacy sets is closely
related to the Riemannian integral, whose well-definedness is 
ensured for continuous and bounded functions. For further details
we refer to the integrability condition introduced in Section
\ref{subsection:consistency}.
If the underlying Markov chain on the states $(\vec{y}_{1:N+1}, I)$ 
satisfies \eqref{eq:balance_condition_mpmcmc} and
is positive Harris with invariant distribution $p(\vec{y}_{1:N+1}, I)$, then 
\eqref{eq:consistency} holds true, which is an immediate consequence 
of the ergodic theorem, Theorem 17.1.7, in 
\cite{meyn2012markov}. For a definition of positive 
Harris we refer to Section 10.1.1 in \cite{meyn2012markov}.

\subsubsection{Sampling from the stationary distribution of $I$}
\label{subsubsection:sampling_from_stationary_distribution}

As stated in Algorithm \ref{algorithm:multiproposal_MH}, one can sample 
directly from the steady-state distribution of the Markov chain, 
conditioned on $\vec{y}_{1:N+1}$. The stationary distribution of $I$, 
given $\vec{y}_{1:N+1}$, also used in \cite{calderhead2014general}, 
equals
\begin{align}
p(I=j | \vec{y}_{1:N+1}) = \frac{\pi(\vec{y}_j)\kappa(\vec{y}_j, \vec{y}_{\setminus j})}{ \sum_{k=1}^{N+1}\pi(\vec{y}_k) \kappa(\vec{y}_k, \vec{y}_{\setminus k})},
\label{eq:barker_acceptance}
\end{align}
for any $j=1,...,N+1$. One can easily see that detailed balance holds for the stationary transition matrix $A(i,j)= p(I=j|\vec{y}_{1:N+1})$. Note that 
\eqref{eq:barker_acceptance} is a generalisation of Barker's algorithm 
\cite{barker1965monte} for multiple proposals in one iteration. For $N=1$, 
the term on the right hand side reduces to Barker's acceptance probability. 
Since this probability is always smaller or equal to Peskun's
acceptance probability, which is used in the usual Metropolis-Hastings
algorithm, samples generated by Barker's algorithm yield mean estimates
with an asymptotic variance that is at least as large as when generated by
Metropolis-Hastings \cite{peskun1973optimum}[Theorem 2.2.1]. A generalisation 
for Peskun's acceptance probability in the multiple proposal setting is
introduced in \cite{tjelmeland2004using}, which aims to minimise the 
diagonal entries of the transition matrix iteratively starting from
the stationary transition matrix, while preserving its reversibility.
The resulting MP-MCMC method is investigated numerically in \ref{subsec:extensions}.

\subsubsection{Sampling from the transient distribution of $I$}

Instead of sampling $I$ from the stationary finite state Markov 
chain conditioned on $\vec{y}_{1:N+1}$, \cite{calderhead2014general} proposes the choice
$A(i,j)=A(i,j|\vec{y}_{1:N+1})$, defined by
\begin{align}
A(i,j) = \begin{cases} \frac{1}{N}\min(1, R(i,j)) &\mbox{if } j \neq i \\
1-\sum_{j\neq i} A(i,j) &\mbox{otherwise,} \end{cases} 
\label{eq:transient_transition}
\end{align}
where $R(i,j) = {\pi(\vec{y}_j)\kappa(\vec{y}_j,\vec{y}_{\setminus j})} / {[\pi(\vec{y}_i)\kappa(\vec{y}_i,\vec{y}_{\setminus i})]}$. Referring to Proposition 1 in
\cite{calderhead2014general}, detailed balance is fulfilled for $A$
as given in \eqref{eq:transient_transition}. 
Note that this choice is a generalisation 
of the original Metropolis-Hastings acceptance probability. For $N=1$, i.e.\ 
if only a single state is proposed and a single state is accepted in each 
iteration, and if we replace the choice of $A$ in Algorithm 
\ref{algorithm:multiproposal_MH} by \eqref{eq:transient_transition}, the resulting algorithm reduces to the usual Metropolis-Hastings algorithm.

\subsection{Some practical aspects}

In this section, we discuss some practical considerations regarding the use
of MP-MCMC and some properties unique to this approach.

\subsubsection{Parallelisation}

In recent years, much effort has been focused on developing parallelisable MCMC strategies for performing inference on increasingly more computationally challenging problems.  A number of specific approaches have been considered previously that incorporate different levels of parallelisation, for example subsampling data to scale MCMC algorithms to big data scenarios (\cite{neiswanger2013asymptotically}, \cite{wang2013parallelizing}), parallelising geometric calculations used in designing efficient proposal mechanisms \cite{welling2011bayesian} and \cite{ahn2012bayesian}, and for certain cases, parallelising the likelihood computation (\cite{agarwal2011distributed}, \cite{smola2010architecture}).




One major advantage of MP-MCMC compared to many MCMC methods, including standard Metropolis-Hastings and other single proposal algorithms, is that it is inherently parallelisable as a single chain.  More precisely, the likelihoods associated with the multiple proposals in any iteration can be computed {\it in parallel} as these expressions are independent of each other.  Evaluating the likelihood is typically far more expensive than prior or proposal densities, and once all proposal likelihoods are computed within one iteration of MP-MCMC, sampling from the finite state chain typically requires minimal computational effort.

In standard single proposal Metropolis-Hastings the likelihood calculations are computed sequentially. In contrast, the computational speed-up of using MP-MCMC is close to a factor $N$, if in every iteration $N$ samples are drawn and $N$ computing cores are available.  We note that it is natural to match the number of proposals to the number of cores that are available for the simulation, or indeed to a mutiple of that number. The latter does not yield further computational speed-up compared to using $N$ proposals but other amenable features arise from an increased number of proposals, as we will see later in this paper.  Indeed, it is for this reason that MP-MCMC outperforms the obvious approach of running multiple single MCMC algorithms in parallel and subsequently combining their samples.


\subsubsection{Computation time}

Compared to $N$ independent chains, MP-MCMC will generally be expected to perform slightly slower on an identical parallel machine with $N$ cores, due to the communication overhead following likelihood computations. More precisely, before sampling from the finite state chain in MP-MCMC, all likelihood evaluations must be completed and communicated.  The overall time for a single iteration will therefore be dependent on the slowest computing time among all individual likelihood evaluations, although we note that measuring computation times is generally dependent on the underlying
operating system architecture, hardware, and the quality or optimality of the implementation.  At the current experimental state of our code we found it therefore not helpful to include computing times in our simulation studies,
but rather investigate platform, language and implementation independent
performance by comparing statistical efficiency with a fixed number of total samples.


\subsubsection{Minimal number of iterations and information gain}

In practice we need to make a couple of choices regarding the number of iterations to use, as well as the number of proposals to make within each iteration.  When using MP-MCMC with proposals that depend on the previous
iteration, employing too small a number of iterations together with a large number of proposals typically leads to a less useful estimate than
single proposal MCMC (Barker to Barker comparison).  What is meant here is that the MSE of global estimates using MP-MCMC, e.g.\ arithmetic mean, 
becomes large. This can be explained by a limited relative global
information gain by increasing the proposal number: in a single 
MCMC iteration, proposals are typically determined using a single
previously generated sample, for instance based on posterior information, such as the local geometry, e.g. MALA and its Riemannian versions. Increasing the
proposal number in a particular MCMC iteration increases the \textit{local} information gain around this point in the posterior. Visually, proposed samples in one iteration of MP-MCMC can be considered as a cloud of points, with some centre and covariance structure, which covers a certain region of the state space.

This local coverage improves with increasing proposal numbers. Thus, increasing the number of proposals will in turn increase the \textit{global} information gain about the posterior by covering the state space more thoroughly, only if sufficiently many MCMC iterations are taken, which each time moves the centre point of our cloud of points.
There is therefore clearly a trade off to be made between local coverage, with the corresponding increased parallelisation of this approach, and more global sequential moves, such that the target is sufficiently well explored.
Through a number of numerical experiments for increasing proposal
numbers, we found that typically at least $\ge 250$ iterations will be 
sufficient to achieve good results.

\subsection{Empirical results}

In the following, the performance of MP-MCMC is investigated
in terms of its MSE convergence and in comparison with single
proposal MCMC for a generic one-dimensional Gaussian posterior. 
We consider two types of acceptance probabilities
used in the transition kernel, one is Barker's acceptance probability 
\eqref{eq:barker_acceptance} and the other is 
Peskun's acceptance probabilities \eqref{eq:transient_transition}.
Note that both generalisations of single proposal acceptance
probabilities to multiple proposals are not unique, and other ways 
of generalising exist, e.g.\ see Section
\ref{subsubsec:optimised_transition_kernels}. To make a 
fair comparison between single and multiple proposal methods, we 
accept $N$ samples in every iteration, which is equal to the number of proposals.
Proposals are generated using a simplified manifold MALA (SmMALA)
kernel \cite{girolami2011riemann}, 
\begin{align}
\kappa(\vec{x}, \cdot) = N \left(\vec{x}+\varepsilon^2/2 G(\vec{x})^{-1}
\nabla \log \pi(\vec{x}), \varepsilon^2 G(\vec{x})^{-1}\right),
\label{eq:smMALA_kernel}
\end{align}
where $G(\vec{x})$ denotes the local covariance matrix given by the
expected Fisher information \cite{girolami2011riemann} for any
$\vec{x}\in \Omega$.
Referring to Figure \ref{fig:mse_SingleVsMulti1}, a performance
gain from switching from MCMC (Barker) to standard MP-MCMC (Barker) is achieved, resulting in an average MSE reduction of ca.\ $30 \%$ overall. There is no significant difference between MP-MCMC (Barker) and MP-MCMC (Peskun). However, usual Metropolis-Hastings outperforms all other methods; in comparison with standard MP-MCMC, this corresponds to an average MSE reduction of ca.\ $30 \%$ overall.
Thus, although average acceptance rates are significantly increased
by using the multiple proposal approach, referring to
\cite{calderhead2014general}, the resulting samples are not necessarily
more informative about the underlying posterior.
However, MP-MCMC still yields the advantage of enabling computations
of likelihoods to be performed in parallel and may be extended in many ways to further improve performance.

\begin{figure}[h]
    \centering
    \begin{subfigure}[b]{0.60\textwidth}
        \includegraphics[width=\textwidth]{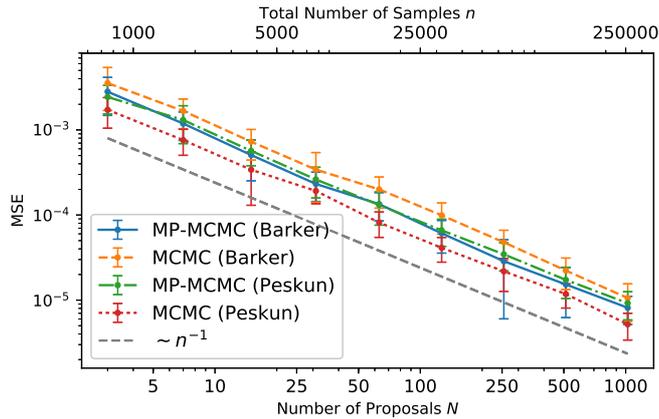}
    \end{subfigure} 
    \caption{\small{MSE of arithmetic mean for MCMC using Barker's and Peskun's
    acceptance probabilities, and MP-MCMC, resp., sampling from 
    a one-dimensional standard Normal posterior for increasing proposal numbers and 
    sample sizes. The results are based on $25$ MCMC runs, and the error bars correspond
    to twice a standard deviation, respectively. The difference between all convergence
    rates is only a constant}}
    \label{fig:mse_SingleVsMulti1}
\end{figure}

\subsection{Extensions of standard MP-MCMC}
\label{subsec:extensions}

We now consider the following extensions to the MP-MCMC, which can be made 
to improve sampling performance and which we investigate empirically.

\subsubsection{Introducing an auxiliary proposal state}
\label{intro_aux_prop_state}

When sampling from $A$ as in \eqref{eq:transient_transition},
the probability of transitioning from $I=i$ to $I=j$ may become small 
when the number of proposals is large. This is since the acceptance 
ratio $R(i,j)$ depends not 
only on the states $\vec{y}_i$ and $\vec{y}_j$, but all proposed states.
One may therefore introduce an auxiliary variable $\vec{z}$ as proposed
in \cite{calderhead2014general, tjelmeland2004using} in order to make
proposals of the form 
$\tilde{\kappa}(\vec{y}_i, \vec{y}_{\setminus i}) 
= \kappa(\vec{y}_i,\vec{z})\kappa(\vec{z},\vec{y}_{\setminus i})$. 
Throughout this work, we make use of this extension for numerical simulations.
Assuming that $\kappa$ samples the proposals $\vec{y}_{\setminus i}$ 
independently from each other, then the acceptance ratio simplifies to
\begin{align*}
R(i,j) = \frac{\pi(\vec{y}_j)\tilde{\kappa}(\vec{y}_j,\vec{y}_{\setminus j})}{\pi(\vec{y}_i)\tilde\kappa(\vec{y}_i,\vec{y}_{\setminus i})} = \frac{\pi(\vec{y}_j)\kappa(\vec{y}_j,\vec{z})\kappa(\vec{z},\vec{y}_i)}{\pi(\vec{y}_i)\kappa(\vec{y}_i,\vec{z})\kappa(\vec{z},\vec{y}_j)}.
\end{align*}
For a symmetric sample distribution $\kappa$, the acceptance ratio further 
simplifies to $R(i,j)={\pi(\vec{y}_j)}/\pi(\vec{y}_i)$.

\subsubsection{Non-reversible transition kernels}
\label{subsubsec:non_reversible_transition_kernels}

\noindent In the context of a finite state Markov chain, \cite{suwa2010markov} 
and \cite{todo2013geometric} introduce an MCMC method that 
redefines the transition matrix, allocating the
probability of an individual state to the transition probabilities
to other states, with the aim of minimising the
average rejection rate. In application, the resulting acceptance
rate is close to $1$, or even equal to $1$ in many iterations.
The resulting transition matrix is no longer reversible;
thus, it does not fulfill the detailed balance condition,
however it still satisfies the balance condition such that 
transitioning from one state
to another preserves the underlying stationary distribution. 
The proposed algorithm is immediately applicable for the finite state
sampling step in MP-MCMC. The resulting algorithm is a non-reversible
MP-MCMC, which preserves the stationary distribution $\pi$ of 
individual samples. 
Numerical experiments of this method can be found 
in Section \ref{subsubsec:empirical_results_extensions}.

\subsubsection{Optimised transition kernels}
\label{subsubsec:optimised_transition_kernels}

\noindent Given a Markov chain over finitely many proposed states, 
\cite{tjelmeland2004using}[Section 4] proposes an algorithm
that iteratively updates the transition matrix of an MCMC
algorithm, starting from Barker's acceptance probabilities defined on
finitely many states. The matrix is updated
until only at most one diagonal element is non-zero.
Since every update leaves the detailed balance condition valid, 
the resulting MCMC algorithm is reversible. Again, this method
is straightforward to apply in the finite state sampling step
of MP-MCMC, resulting in a reversible MP-MCMC, which clearly
leaves the stationary distribution $\pi$ of individual samples
unaltered. For a single proposal, this algorithm reduces to
the usual Metropolis-Hastings.
We now consider the performance in numerical experiments, 
comparing this method to the 
MP-MCMC with non-reversible transition kernels from 
\ref{subsubsec:non_reversible_transition_kernels}
and the standard MP-MCMC.

\subsubsection{Empirical results of MP-MCMC extensions}
\label{subsubsec:empirical_results_extensions}

In what follows, we compare the performance of standard MP-MCMC
to the MP-MCMC algorithms with improved transition kernels introduced
in Section \ref{subsubsec:non_reversible_transition_kernels} and
Section \ref{subsubsec:optimised_transition_kernels}. As posterior
distribution we consider a one-dimensional Gaussian, and for the underlying
proposal kernel we use the SmMALA formalism from \eqref{eq:smMALA_kernel}.
As measures
of performance we consider MSE convergence, acceptance rate and
mean squared jumping distance (MSJD) for increasing numbers of
proposals. Here, the acceptance rate is defined by the probability
of transitioning from one state to any of the $N$ different ones.
Further, MSJD $= 1/n \sum_{\ell =1}^n \| \vec{x}_{n+1} - \vec{x}_n \|^2$,
which is related to measuring the lag $1$ autocorrelation, and can
be applied to find the optimal choice of parameters determining the
proposal kernel (\cite{pasarica2010adaptively}).

Referring to Figure \ref{fig:acpt_msjd_improved_transitions},
switching from MP-MCMC to any of the MP-MCMC algorithms introduced in the 
previous two sections increases the average acceptance rate and the MSJD
for small proposal numbers significantly, and to a similar extend.
While the number of proposals increases, the difference to standard
MP-MCMC disappears, as they tend to the same maximal value; for 
the acceptance rate, this is $1$. At the same time, a performance gain
in terms of MSE, referring to Figure \ref{fig:mse_nonrev},
is only achieved by switching from the standard MP-MCMC transition kernel
to the non-reversible kernel (\ref{subsubsec:non_reversible_transition_kernels}), 
resulting in an average MSE reduction of 
ca.\ $20 \%$ overall. Applying the optimised transition algorithm
(\ref{subsubsec:optimised_transition_kernels})
does not significantly reduce the MSE, i.e.\ on average less than
$5 \%$ overall, compared to standard MP-MCMC. This is interesting
since it makes the point that, although the choice of proposals
is exactly the same, an increased acceptance rate does not
imply that the resulting samples are significantly
more informative about the posterior. In contrast to this observation we will see in Section \ref{sec:importance_sampling} that actually, more informative estimates, i.e.\ exhibiting lower variance, can indeed be achieved by accepting all proposed samples (i.e.\ acceptance rate $=1$) by suitably weighting them.
The number of iterations in the optimisation procedure used 
to generate the transition kernel from Section \ref{subsubsec:optimised_transition_kernels} increases significantly with the number of
proposals, and therefore the computation cost.
To sample a constant number of $250$ iterations in our toy example, we found the cost 
of performing the optimisation for more than $\approx 125$ proposals prohibitive.
For reference and comparison reasons, we displayed the results for 
proposal numbers only up to this number for standard MP-MCMC and non-reversible
MP-MCMC, too. 
Summarising, we were able to improve the constant in front of the convergence rate by using non-reversible and optimised transition kernels, however not the rate itself.

\begin{figure}[h]
    \centering
    \begin{subfigure}[b]{0.45\textwidth}
        \includegraphics[width=\textwidth]{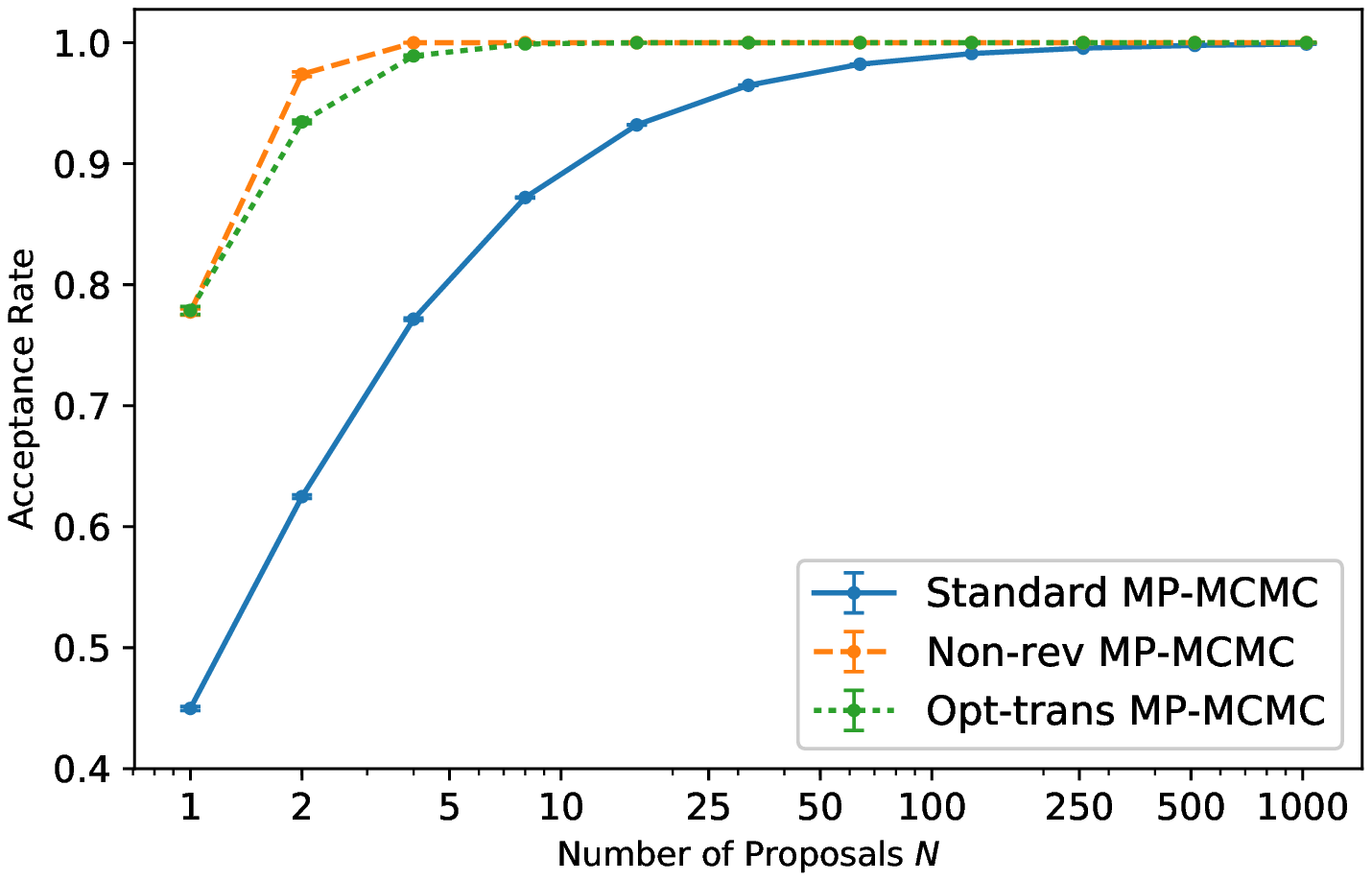}
    \end{subfigure} 
    \begin{subfigure}[b]{0.45\textwidth}
        \includegraphics[width=\textwidth]{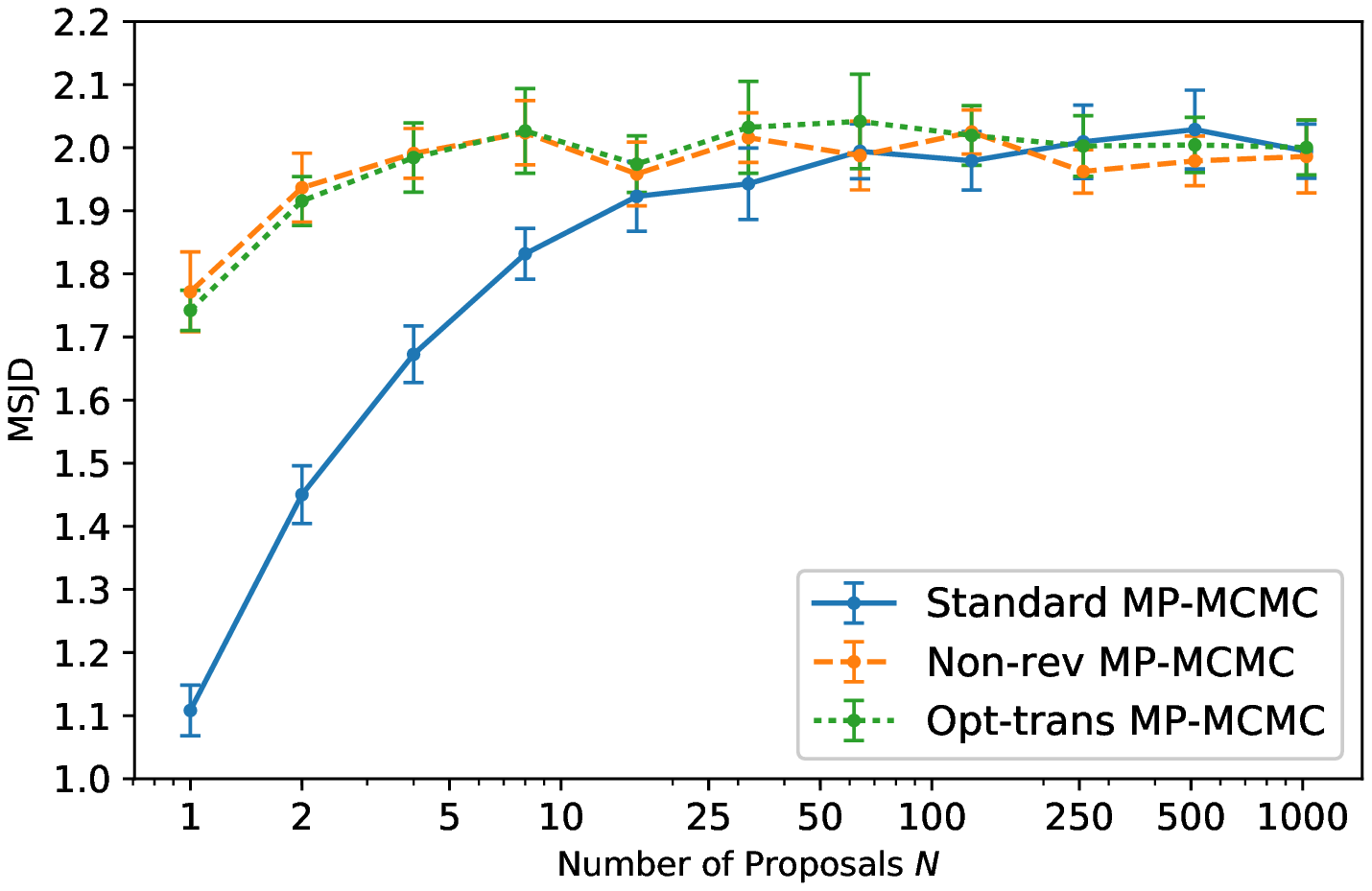}
    \end{subfigure}     
    \caption{\small{Acceptance rates and MSJD for standard MP-MCMC, MP-MCMC with 
    non-reversible (Non-rev) and optimised transitions (Opt-trans), sampling from 
    a one-dimensional standard Normal posterior for increasing proposal numbers and 
    sample sizes. The results are based on $10$ MCMC runs, and the error bars correspond
    to twice a standard deviation}}
    \label{fig:acpt_msjd_improved_transitions}
\end{figure}

\begin{figure}[h]
    \centering
    \begin{subfigure}[b]{0.60\textwidth}
        \includegraphics[width=\textwidth]{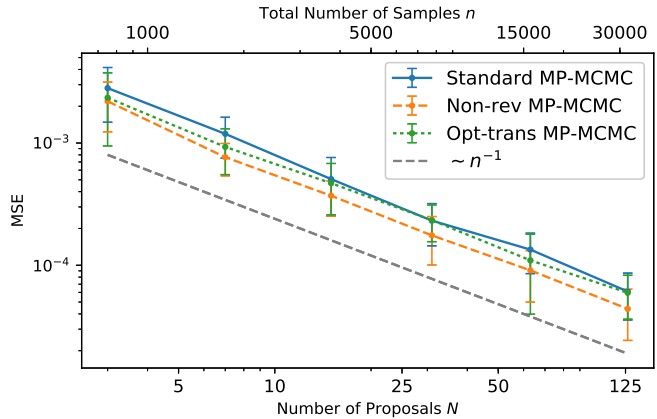}
    \end{subfigure} 
    \caption{\small{MSE of arithmetic mean for MP-MCMC, MP-MCMC with reversible 
    transitions (Non-rev) and MP-MCMC with optimised transitions (Opt-trans), resp., 
    sampling from 
    a one-dimensional standard Normal posterior for increasing proposal numbers and 
    sample sizes. The results are based on $25$ MCMC runs, and the error bars 
    correspond to twice a standard deviation, resp.}}
    \label{fig:mse_nonrev}
\end{figure}

\subsection{Limit theorems}
\label{subsection:limit_theorems}

In this section, the law of large numbers (LLN), a central limit theorem (CLT), 
and an expression for the asymptotic variance is derived. Essential for the
derivation is the observation that MP-MCMC can be considered as a single 
Markov chain on the product space of variables $(\vec{y}_{1:N+1},I_{1:M})$, 
i.e.\ of proposals and auxiliary variables, when
$N$ proposals are generated and $M$ samples are accepted in every iteration, respectively.
Here, $\vec{y}_i\in\Omega$
and $I_m \in\{1,...,N+1\}$ for $i=1,...,N+1$ and $m=1,...,M$. The joint 
probability on the associated space is 
given by
\begin{align*}
p(\vec{y}_{1:N+1},I_{1:M}) &= p(\vec{y}_{i}) \kappa(\vec{y}_{i}, \vec{y}_{\setminus{i}})  p(I_{1:M} |\vec{y}_{1:N+1})
\end{align*}
for any $i=1,...,N+1$, where $\kappa(\vec{y}_i,\vec{y}_{\setminus{i}})$ denotes the proposal
distribution. If ergodicity holds, then $p(\vec{y}_i)=\pi(\vec{y}_i)$ asymptotically, i.e.\ the samples
collected at each iteration will be distributed according to the target. In that case,
updating $(\vec{y}_{1:N+1},I_{1:M})$ leaves the joint distribution invariant.
Throughout this section, we will state the main results and definitions while
referring to the appendix for the full proof of results to aid readability.




\subsubsection{Law of large numbers}
\label{subsubsec:lln}

Given a scalar-valued function $f$ on $\Omega$, and samples 
$\vec{x}_1, \vec{x}_2, ...$ of an MP-MCMC
simulation according to Algorithm \ref{algorithm:multiproposal_MH}, 
we wish to prove that 
$\hat\mu_{n}\rightarrow \mu$ a.s., where $\hat\mu_{n} = (1/n) \sum_{i} f(\vec{x}_i)$ 
and $\mu = \int f(\vec{x}) \pi(\vec{x}) \mathrm{d}\vec{x}$, which is equivalent to 
$\hat\mu_{n,M,N}\rightarrow\mu$ a.s.\ for $n\rightarrow \infty$,
where
\begin{align}
    \hat\mu_{n,M,N} = \frac{1}{nM} \sum_{i=1}^n\sum_{m=1}^{M} f(\vec{x}_m^{(i)}),
    \label{eq:lln_mp_mcmc_inner_and_outer_iterations}
\end{align}
where $\vec{x}_m^{(i)}$ denotes the $m$th sample in the $i$th MCMC iteration. 

\begin{Lemma}[Law of Large Numbers]
\label{lemma:lln_mp_mcmc}
Assuming that the single Markov chain on the product space of
accepted samples per iteration, as defined by MP-MCMC, is
positive Harris \cite[10.1]{meyn2012markov}, then the law of large numbers holds true.
\begin{proof}
A proof is given in Appendix \ref{appendix:proof_lemma_lln_mp_mcmc}.
\end{proof}
\end{Lemma}

\subsubsection{Central limit theorem}

\noindent The result we would like to have is the following: given a scalar-valued function $f$ on $\Omega$, and samples $\vec{x}_1, \vec{x}_2, ...$ of an MP-MCMC simulation, then
\begin{align*}
    \sqrt{n}\left(\hat\mu_n - \mu \right) \xrightarrow{\mathcal{D}} \mathcal{N}(0,\sigma^2),
\end{align*}
for some $\sigma^2>0$, which is equivalent to 
\begin{align*}
    \sqrt{nM}\left(\hat\mu_{n,M,N} - \mu \right) \xrightarrow{\mathcal{D}} \mathcal{N}(0,\sigma^2) \quad \text{for} \quad n\rightarrow\infty,
\end{align*}
where we used the same notation as in section \ref{subsubsec:lln}. In order to
prove the CLT we assume the chain to be positive Harris, and that the asymptotic 
variance can be expressed by the limit of the variances at iteration 
$n$ for $n \rightarrow \infty$, and is well-defined and positive. Since this 
seems to be natural to assume, we refer to this assumption by saying the MP-MCMC 
Markov chain is \textit{well-behaved}. Since this assumption is not easily 
verifiable in practice, we also give a formal definition of a uniform ergodicity 
condition on the Markov chain from \cite{meyn2012markov} that ensures the above. 
For the sake of readability, we refer to Appendix \ref{appendix:proof_lemma_ctl_mp_mcmc} 
for this formal condition.

\begin{Lemma}[Central Limit Theorem]
\label{lemma:clt_mp_mcmc}
Assuming that the single Markov chain on the product space of
accepted samples per iteration, as defined by MP-MCMC, is 
positive Harris and \textit{well-behaved}, then the central limit 
theorem holds true, and the asymptotic variance of the sequence 
$(f(\vec{x}_i))_{i \ge 1}$ is given by
\begin{align*}
    \sigma^2 = \zeta^{(0)} 
    + 2\sum_{1\le \ell<m \le M} \zeta^{(0)}_{\ell,m} 
    + \frac{2}{M}\sum_{k=1}^\infty \sum_{\ell,m=1}^{M} \zeta^{(k)}_{\ell,m}
\end{align*}
where $\zeta^{(0)}=\Var_{\pi}(f(\vec{x}))$, 
$\zeta_{m,n}^{(0)} = \Cov(f(\vec{x}_m),f(\vec{x}_{n}))$, and
$\zeta^{(k)}_{\ell,m} = \Cov( f(\vec{x}_\ell^{(i)}, f(\vec{x}_m^{(i+k)}))$ 
for any $ i,k \in \mathbb{N}$ and $\ell, m=1,...,M\in\mathbb{N}$.
\begin{proof}
For a proof, we refer to Appendix \ref{appendix:proof_lemma_ctl_mp_mcmc}.
\end{proof}
\end{Lemma}

\subsection{Adaptive MP-MCMC}
\label{subsec:adaptive_mpmcmc}


We now introduce adaptive versions of MP-MCMC within a general 
framework of adaptivity for Markov chains, and present theory based on 
\cite{roberts2007coupling} that allows us to prove ergodicity in adaptive MP-MCMC. Further, an explicit adaptive MP-MCMC method is introduced as Algorithm \ref{algorithm:adaptive_mp_mcmc} in Section \ref{subsubsec:an_adaptive_mpmcmc_algorithm}, for which we prove ergodicity based on the results mentioned above. The performance of IS-MP-MCMC compared to MP-MCMC is then investigated in a simulation study.

\subsubsection{Adaptive transition kernels in MP-MCMC}

In the following we consider MP-MCMC as a single Markov chain over the 
accepted samples in one iteration. For any $n\in \mathbb{N}$, let 
$\vec{Z}_n = \vec{x}_{1:M}^{(n)} \in \mathbb{R}^{Md}$
denote the state of the MP-MCMC algorithm 
at time $n$, i.e.\ the vector of accepted samples in iteration $n$. Further let $\Gamma_n$ denote a $\mathcal{Y}$-valued random variable which 
determines the kernel choice for
updating $\vec{Z}_n$ to $\vec{Z}_{n+1}$. We have
\begin{align*}
p(\vec{Z}_{n+1} =\vec{z}_{n+1}| \vec{Z}_i=\vec{z}_i, \Gamma_i=\gamma_i, i=1,...,n ) &= p(\vec{Z}_{n+1} =\vec{z}_{n+1}| \vec{Z}_n=\vec{z}_{n}, \Gamma_n=\gamma_n) \\
&= P_{\gamma_n}(\vec{z}_n,\vec{z}_{n+1}).
\end{align*}
The dependency of $\Gamma_n$ on previous samples and kernel choices, i.e.\
$\Gamma_n|\vec{Z}_i, \Gamma_i, i=1,...,n$, will be determined by the corresponding
adaptive algorithm. For given starting values $\vec{Z}_0=\vec{z}, \Gamma_0=\gamma$, let
\begin{align*}
P^{(n)}_{\gamma} (\vec{z}, \vec{z}_n) = p(\vec{Z}_n=\vec{z}_n| \vec{Z}_0=\vec{z}, \Gamma_0=\gamma)
\end{align*}
denote the corresponding $n$-step conditional probability density of the 
associated adaptive algorithm. Finally, let us define the total variation
between the joint distribution $p$ of variables $\vec{z}=\vec{x}_{1:M}$ and 
$P^{(n)}_{\gamma}$ by
\begin{align*}
T(\vec{z},\gamma, n) = \left\| P^{(n)}_{\gamma}(\vec{z},\cdot) - p(\cdot) \right\| = \sup_{B} \left|\int_B P^{(n)}_{\gamma}(\vec{z},\vec{\tilde{z}}) - p(\vec{\tilde{z}})\mathrm{d}\vec{\tilde{z}} \right|.
\end{align*}
Note that under the joint distribution $p$, all individual samples 
from every iteration are distributed according to the target $\pi$.

\subsubsection{Ergodicity for adaptive chains}

Following \cite{roberts2007coupling}, we call the underlying adaptive algorithm
ergodic if $T(\vec{z},\gamma, n)\rightarrow 0$ for $n\rightarrow \infty$. Referring to
their Theorem 1 and Theorem 2, the following results give sufficient 
conditions, under which ergodicity holds true. An integral requirement in both
theorems is that the changes of the transition kernel due to adaptation tend to 
zero. To that end, we define the random variable
\begin{align*}
D_n = D_n(\Gamma_{n+1},\Gamma_n) = \sup_{\vec{z}}\| P_{\Gamma_{n+1}}(\vec{z},\cdot) - P_{\Gamma_n}(\vec{z}, \cdot)\|.
\end{align*}
Note that $D_n \rightarrow 0$ for $n\rightarrow \infty$ does not mean that
$\Gamma_n$ necessarily converges. Also, the amount of adaptation to be
infinite, i.e.\ $\sum_n D_n = \infty$, is allowed.
We note that there exist more general 
adaptive schemes in the literature
(\cite{atchad2011adaptive}) which allow that different transition kernels $P_\gamma$ 
have different stationary distributions, however we do not consider these here.

\begin{theorem}[Theorem 1, \cite{roberts2007coupling}]
\label{thm:theorem1roberts2007}
Given an adaptive MP-MCMC algorithm with adaptation space $\mathcal{Y}$ and such
that $p$ is the stationary distribution for any transition kernel $P_\gamma$, 
$\gamma \in \mathcal{Y}$. If,
\begin{itemize}
\item (Simultaneous Uniform Ergodicity) For any $\varepsilon>0$ there is
$N=N(\epsilon)\in \mathbb{N}$ such that $\| P^{(N)}_\gamma(\vec{z},\cdot) - p(\cdot) \|\le \varepsilon$
for any $\vec{z}$ and $\gamma\in \mathcal{Y}$; and
\item (Diminishing Adaptation) $D_n\rightarrow 0$ for $n \rightarrow \infty$ in probability,
\end{itemize}
then the adaptive algorithm is ergodic.
\end{theorem}

The first condition in the previous result is relatively strong, and might
not always be verifiable in practice. However, ergodicity still holds if the uniform
convergence of $P_\gamma$ is relaxed to the following containment condition: 
roughly speaking, for given starting values for $\vec{Z}_0$ and $\Gamma_0$, and 
sufficiently large $n$, $P^n_\gamma$ is close to $p$ with high probability. 
To formalise the containment condition, let us define for any $\varepsilon>0$,
$\vec{z}\in \mathbb{R}^{Md}$ and $\gamma \in \mathcal{Y}$,
\begin{align*}
M_\epsilon (\vec{z},\gamma) = \inf \{ n\ge 1: \| P_\gamma^n(\vec{z},\cdot) - p(\cdot) \| \le \epsilon \},
\end{align*}
and state the following ergodicity result.

\begin{theorem}[Theorem 2, \cite{roberts2007coupling}]
\label{thm:theorem2roberts2007}
Consider an adaptive MP-MCMC algorithm that has diminishing adaptation, and let
$\vec{z}^* \in \Omega^M, \gamma^* \in \mathcal{Y}$. If 
\begin{itemize}
\item (Containment) For any $\delta>0$ there exists a $N\in \mathbb{N}$ such that
$P(M_\epsilon (\vec{Z}_n,\Gamma_n) \le N| \vec{Z}_0=\vec{z}^*, \Gamma_0=\gamma^*)\ge 1-\delta$ for
any $n \in \mathbb{N}$,
\end{itemize}
then the adaptive algorithm is ergodic.
\end{theorem}

The proofs of both previous theorems are based on coupling
the adaptive chain with another chain that is adaptive only up to a certain iteration. 
Referring to \cite{bai2011containment} containment can actually be derived 
via the much easier to verify simultaneous polynomial ergodicity or 
simultaneous geometrical ergodicity. The latter immediately implies containment.

\begin{Lemma}[Asymptotic distribution of adaptive MP-MCMC]
Suppose that either the conditions of Theorem \ref{thm:theorem1roberts2007} or
\ref{thm:theorem2roberts2007} are satisfied, then the accepted samples 
$\vec{x}^{(n)}_m\in \Omega$, i.e.\ the $m$th sample from the $n$th
iteration for $n\in \mathbb{N}$ and $m=1,...,M\in \mathbb{N}$, given by 
$\vec{y}^{(n)}_{I_m^{(n)}}=\vec{x}^{(n)}_m$, are asymptotically distributed 
according to the target $\pi$.
\begin{proof}
As the asymptotic behaviour of the Markov chain, defined on states $\vec{x}_{1:M}$, 
is not influenced by the initial distribution, we may assume the joint target $p$
as initial distribution. The statement then follows immediately.
\end{proof}
\end{Lemma}

\subsubsection{An adaptive MP-MCMC algorithm}
\label{subsubsec:an_adaptive_mpmcmc_algorithm}

In this section, we consider an adaptive version of the MP-MCMC algorithm,
which allows for iterative updates of the proposal covariance. The
underlying proposal distribution is formulated in a general fashion, 
however ergodicity will be proven for the special case of a Normal
distribution. In that case, the
resulting algorithm can be considered as a generalisation of the adaptive
MCMC algorithm introduced by Haario et al.\ \cite{haario2001adaptive}
allowing for multiple proposals. However, a different covariance
estimator than the standard empirical covariance used in 
\cite{haario2001adaptive} is applied here: the estimate for the proposal
covariance in iteration $n+1$ incorporates information
from all previous proposals of iterations $1,2,..., n\in \mathbb{N}$. The 
proposals are thereby weighted according to the stationary distribution of
the auxiliary variable. Note that weighting proposed states does not
necessarily decrease the asymptotic variance of the resulting mean estimates
(\cite{delmas2009does}), although in many cases it does 
(\cite{frenkel2006waste,ceperley1977monte}). This holds in particular when
the number of proposed states is large, which is why we find the weighting
estimator preferable over the standard empirical covariance estimate.
The resulting method is displayed as Algorithm
\ref{algorithm:adaptive_mp_mcmc}, where we have highlighted the differences compared to Algorithm \ref{algorithm:multiproposal_MH}.


\begin{algorithm}[h]
\SetAlgoLined
\KwIn{Initialise starting point $\vec{x}_0=\vec{y}_1\in\Omega$, number of proposals $N$, number of
accepted samples per iteration $M$, auxiliary 
variable $I=1$, counter $n=1$, initial \hl{mean estimate $\vec{\mu}_1$ and covariance estimate $\Sigma_1$}\;}
\For{\textnormal{each MCMC iteration $\ell=1,2,...$}}{
 Sample $\vec{y}_{\setminus I}$ conditioned on $I$ \hl{and $\Sigma_\ell$}, i.e., draw $N$ 
 new points from the proposal kernel $\kappa_{\text{\hl{$\Sigma_\ell$}}}(\vec{y}_I, \cdot) = p(\vec{y}_{\setminus I}|\vec{y}_I, $\hl{$\Sigma_\ell$}$)$ \;
  Calculate the stationary distribution of $I$ conditioned on $\vec{y}_{1:N+1}$ 
  \hl{and $\Sigma_\ell$}, i.e.\ $\forall$ $i=1,...,N+1$, $p(I=i|\vec{y}_{1:N+1}, $\text{\hl{$\Sigma_\ell$}}$) = \pi(\vec{y}_i)\kappa_{\text{\hl{$\Sigma_\ell$}}}(\vec{y}_{{i}}, \vec{y}_{\setminus{i}}) / \sum_j \pi(\vec{y}_j)\kappa_{\text{\hl{$\Sigma_\ell$}}}(\vec{y}_{{j}}, \vec{y}_{\setminus{j}})$, 
  which can be done in parallel\;
 \For{$m=1,...,M$}{
 Sample new $I$ via the stationary distribution $p(\cdot|\vec{y}_{1:N+1},$\hl{$\Sigma_\ell$}$)$\; 
 Set new sample $\vec{x}_{n+m} = \vec{y}_I$\;
}
Update counter $n=n+M$ \;
\hl{Compute $\tilde{\vec{\mu}}_{\ell+1}=\sum_{i} p(I=i| {{\vec{y}}}_{1:N+1}, \Sigma_\ell)\vec{y}_i$}\;
\hl{Set $\vec{\mu}_{\ell+1} = \vec{\mu}_{\ell} + \frac{1}{\ell+1}\left(\tilde{\vec{\mu}}_{\ell+1} - \vec{\mu}_{\ell}\right)$}\;
\hl{Compute $\tilde{{\Sigma}}_{\ell+1} = \sum_{i}p(I=i| {\vec{y}}_{1:N+1}, {{\Sigma}}_\ell)[{\vec{y}}_i-\vec{\mu}_{\ell+1}][{\vec{y}}_i-\vec{\mu}_{\ell+1}]^T$}\;
\hl{Set $\Sigma_{\ell+1} = \Sigma_{\ell} + \frac{1}{\ell+1}(\tilde{\Sigma}_{\ell+1} - \Sigma_{\ell})$}\;
 }
\caption{Adaptive MP-MCMC \newline 
All code altered compared to original MP-MCMC, Algorithm \ref{algorithm:multiproposal_MH},
is highlighted}
 \label{algorithm:adaptive_mp_mcmc}
\end{algorithm}

\vspace{4mm}
\noindent \textbf{Ergodicity}
\vspace{1.5mm}

\noindent In what follows we prove ergodicity of the underlying adaptive MP-MCMC
method with $\kappa_\gamma = N_\gamma$, i.e.\ the proposal distribution
being normally distributed, based on the sufficient conditions provided by Theorem
\ref{thm:theorem2roberts2007}. We prove three different ergodicity 
results, each based on slightly different requirements. We begin with
the case of when proposals are sampled independently of previous samples.
In all cases we assume the target $\pi$ to be absolutely continuous with respect
to the Lebesgue measure. Further, we say that $\mathcal{Y}$ is bounded if there are
$0<c_1<c_2<\infty$ such that $c_1 I \le \gamma \le c_2I$ for any
$\gamma \in \mathcal{Y}$, where the ``$\le$'' is understood in the usual way
considering matrices: For two matrices $A,B \in \mathbb{R}^{n\times n}$, 
$A\le B$ means that $B-A$ is positive semi-definite.

\begin{theorem}
\label{thm:ergodicity_adap_mpmcmc_independent}
Let us assume that the proposal distribution 
$\kappa_\gamma=N_{\gamma}, \gamma \in \mathcal{Y}$, 
depends on previous samples only through the parameter $\gamma$ but is 
otherwise independent. If $\mathcal{Y}$
is bounded, then the adaptive MP-MCMC method described by 
Algorithm \ref{algorithm:adaptive_mp_mcmc} is ergodic.
\begin{proof}
The proof is based on Theorem \ref{thm:theorem2roberts2007}, and can be
found in Appendix \ref{subsec:proof_ergodicity_adapt_mpmcmc_independent}.
\end{proof}
\end{theorem}

The following result allows for adapting the mean value of the Normal
distribution in addition to its covariance adaptively. The mean value
is estimated via weighted proposals, as defined in equation 
\eqref{eq:def_importance_sampling_estimator}.

 
\begin{corollary}
\label{cor:ergodicity_adapt_mpmcmc_independent_mean}
Let us assume that the proposal distribution 
$\kappa_\gamma= N_{\gamma}, \gamma=(\vec{\mu}, \Sigma) \in \mathcal{Y}$ 
depends on previous samples only through the parameter $\gamma$ but is 
otherwise independent.
If $\mathcal{Y}$ is bounded, i.e.\ both mean and covariance estimates
are bounded, then the adaptive MP-MCMC method described by Algorithm
\ref{algorithm:adaptive_mp_mcmc} is ergodic.
\begin{proof}
The containment condition follows analogously to the proof for Theorem
\ref{thm:ergodicity_adap_mpmcmc_independent}. The proof of diminishing 
adaptation requires integration of $N_{\gamma_n+s(\gamma_{n+1}-\gamma_n)}$, 
where $\gamma_n = (\vec{\mu}_n, \Sigma_n)$, similar to equation 
\eqref{eq:upper_bound_dim_adapt}. An estimate similar to 
\eqref{eq:exp_term_estimate_dim_adapt2}
follows, where the bound is given by additional constant terms
multiplied by either $\|\vec{\mu}_{n+1}- \vec{\mu}_n\|$ or $\|\Sigma_{n+1}- \Sigma_n\|$.
Both terms can be bounded by a constant multiplied by $1/n\rightarrow 0$ for $n\rightarrow \infty$, which concludes the proof.
\end{proof}
\end{corollary}

Now, we consider the case where we assume that the target $\pi$ has bounded 
support, i.e.\ there is $S\subset \mathbb{R}^d$ with $\lambda(S)<\infty$
such that $\pi(\vec{z})=0$ for any $\vec{z} \not\in S$. The dependence of
the proposal distribution on previous samples is not restricted anymore only to the
adaptation parameters.

\begin{theorem}
\label{thm:ergodicity_adapt_mpmcmc_bounded}
If $\kappa_\gamma=N_{\gamma}, \gamma \in \mathcal{Y}$, and the target distribution 
$\pi$ has bounded support in $\mathbb{R}^d$ and its density is continuous, then
the adaptive MP-MCMC method described by Algorithm \ref{algorithm:adaptive_mp_mcmc} 
is ergodic.
\begin{proof}
The proof is again based on Theorem \ref{thm:theorem2roberts2007}, and can be
found in Appendix \ref{subsec:ergodicity_adapt_mpmcmc_bounded}.
\end{proof}
\end{theorem}

\begin{corollary}
Under the same conditions as in Theorem \ref{thm:ergodicity_adapt_mpmcmc_bounded},
except that we do not assume continuity, the statement still holds true
if we instead assume that $\pi$ is bounded from above and below on its support, i.e.\
$\exists \ 0<\eta \le \rho < \infty$ such that
\begin{align}
\eta \le \pi(\vec{x}) \le \rho 
\label{eq:boundedness_target_support}
\end{align}
for any $\vec{x}$ in the support of $\pi$.
\begin{proof}
The statement follows similarly to the proof of Theorem 
\ref{thm:ergodicity_adapt_mpmcmc_bounded}, except that the boundedness of
the transition probability on the finite state chain in 
\eqref{eq:boundness_finite_chain_transition_probability}
is ensured by \eqref{eq:boundedness_target_support} instead of by continuity.
\end{proof}
\end{corollary}

In the following case, we waive both the independence as well as the
bounded support condition, however we require a few other assumptions
to show ergodicity in Theorem \ref{thm:ergodicity_adaptive_mpmcmc_positive},
among which is the positivity of the target distribution. Therefore, 
Theorem \ref{thm:ergodicity_adapt_mpmcmc_bounded} and Theorem
\ref{thm:ergodicity_adaptive_mpmcmc_positive} can be considered as
complementing each other.
The containment condition is generally hard to prove in practise
wihtout further assumptions, and for the case of unbounded support 
of a dependent proposal kernel we have not managed to achieve such
a direct proof. Hence, we turned to \cite{craiu2015stability} which states
sufficient and practically verifiable assumptions under which the
rather technical containment condition holds true.
More precisely, it
is assumed that the underlying Markov chain can only jump within a finite
distance of any current sample. Furthermore, the transition
kernel is only adapted within a compact region of the state space. Outside
of this region, a fixed proposal kernel is used, which defines a chain that
converges to the correct stationary distribution. In what follows,
$K_D$ is the set of all states within an Euclidean distance $D$ from $K$
for any bounded set $K \subset \Omega^M$.

\begin{assumption}[Bounded jump condition]
\label{assumption:bounded_jump_condition}
Assume that there is a $D< \infty$ such that
\begin{align}
P_\gamma(\tilde{\vec{z}}, \{ \vec{z}\in \Omega^M: 
\|\vec{z} - \tilde{\vec{z}} \| \le D
\}) = 1 \quad \forall \ \tilde{\vec{z}}\in \Omega^M, \gamma \in \mathcal{Y}.
\label{eq:bounded_jump_condition}
\end{align}
\end{assumption}

\begin{assumption}[Non-adaptive kernel condition]
\label{assumption:adaptation_within_compact}
Assume that there is a bounded $K\subset \Omega^M$ such that
\begin{align}
P_\gamma(\vec{z}, B) = P(\vec{z}, B) \quad \forall \ B\in \mathcal{B}(\Omega^M), 
\vec{z}\in \Omega^M\setminus{K},
\label{eq:no_adaptation_outside_compact}
\end{align}
for some fixed transition kernel $P$ defining a chain that converges to
the correct stationary distribution $p$ on $\Omega^M$ in total variation
for any initial point. Further, it is assumed that
\begin{align}
\exists \ M<\infty \quad \text{such that} \quad P(\tilde{\vec{z}}, \mathrm{d}\vec{z})
\le M \lambda(\mathrm{d}\vec{z}),
\label{eq:boundedness_above_transition_kernel}
\end{align}
for any $\tilde{\vec{z}}\in K_D\setminus{K}$ and any 
$\vec{z} \in K_{2D} \setminus{K_D}$, where $D$ is as in 
Assumption \ref{assumption:bounded_jump_condition}. Moreover, there are 
$\varepsilon, \delta >0$ such that
\begin{align}
P(\tilde{\vec{z}}, \mathrm{d}\vec{z}) \ge \varepsilon \lambda(\mathrm{d}\vec{z})
\label{eq:boundedness_below_transition_kernel}
\end{align}
for any $\tilde{\vec{z}}, \vec{z}$ with 
$\| \tilde{\vec{z}} - \vec{z} \|< \delta$ 
in some bounded rectangle contained in $\Omega^M$ and
that contains $K_{2D}\setminus{K_D}$.
\end{assumption}

We remark that the conditions from equation \eqref{eq:bounded_jump_condition} 
and \eqref{eq:no_adaptation_outside_compact} are easily
enforced upon Algoritm \ref{algorithm:adaptive_mp_mcmc} by making
the following changes to the algorithm: if 
a proposal is generated that does not satisfy the first equation
simply remove the proposal and sample a new proposal and repeat until 
it the condition satisfied. In order to ensure the second
assumption, set the proposal distribution to a fixed Gaussian
distribution outside of the set $K$. One easily verifies that
the remaining conditions \eqref{eq:boundedness_above_transition_kernel}
and \eqref{eq:boundedness_below_transition_kernel} are then
also satisfied.

\begin{theorem}[Ergodicity of adaptive MP-MCMC]
\label{thm:ergodicity_adaptive_mpmcmc_positive}
Let the conditions from Assumption \ref{assumption:bounded_jump_condition}
and  Assumption \ref{assumption:adaptation_within_compact} be satisfied.
Further, let $\pi$ be continuous and positive, $\Omega \subset \mathbb{R}^d$ 
open, and $\mathcal{Y}$ be
bounded. Then, the adaptive MP-MCMC method described by Algorithm
\ref{algorithm:adaptive_mp_mcmc} is ergodic.
\begin{proof}
For a proof, we refer to Appendix \ref{subsec:ergodicity_adapt_mpmcmc_positive}.
\end{proof}
\end{theorem}

\subsubsection{Adaptive MP-MCMC with non-Gaussian proposals}

The question arises what theoretical guarantees hold if the underlying
proposal distribution is not Gaussian, different to the case of Algorithm
\ref{algorithm:adaptive_mp_mcmc}. As before, we may proceed by proving 
diminishing adaptation and containment. According to
\cite{craiu2015stability}, the first condition, which basically says
that the changes in the process become smaller and smaller as time
goes by, is typically simple to achieve by carefully choosing the
proposal distribution and designing what adaptation is used.
As pointed out above, the second condition is hard to prove directly
without further assumptions. However, by raising the further
two conditions \ref{assumption:bounded_jump_condition} and
\ref{assumption:adaptation_within_compact} upon the transition kernel
ensures containment. Both conditions are closely related to the 
choice of the proposal distribution and the design of adaptation, which
both are in the hand of the user. This suggests that for 
algorithms that apply a different adaptation as in Algorithm
\ref{algorithm:adaptive_mp_mcmc} and have non-Gaussian
proposals similar results as in the last section can be achieved.

\section{Pseudo-random MP-MCMC with Importance Sampling}
\label{sec:importance_sampling}

In MP-MCMC, samples are drawn from a Markov chain
defined on all $N+1$ proposals in each iteration. These samples are
in turn typically used to compute some quantity of interest, which can
be expressed as an integral with respect to the target.
The same can be achieved by weighting all proposed states from
each iteration appropriately without any sampling on the finite state 
chain. Before explaining the details of the resulting method, we state
an intuitive motivation for why we should make use of weighting.

\subsection{Introduction and motivation}

We start by arguing that increasing the number of accepted samples per iteration 
while keeping the number of proposals and the number of outer iterations 
constant is typically beneficial in terms of reducing the empirical 
variance of estimates. In order to see this, note that for the variance $\sigma_{n,M,N}^2$ 
of the mean estimator $\hat{\mu}_{n,M,N}$ for $n,M,N\in\mathbb{N}$ it holds that
\begin{align*}
    \sigma_{n,M,N}^2 &= \frac{1}{n}\sum_{i=1}^n \Var 
    \left( F\left( \vec{x}_{1:M}^{(i)} \right) \right)\\
    &\quad\quad+ \frac{2}{n} \sum_{1\le i<j\le n} \Cov
    \left( F \left( \vec{x}_{1:M}^{(i)} \right), 
    F\left( \vec{x}_{1:M}^{(j)} \right) \right),
\end{align*}
where $\vec{x}_{1:M}^{(i)}$ states the accepted set of $M$ samples
in the $i$th iteration, and $F$ is defined as
\begin{align}
	 F(\vec{x}_{1:M}) =  \frac{1}{M} \sum_{m=1}^{M} f(\vec{x}_{m}).
    \label{eq:definition_capital_F}
\end{align}
Further,
\begin{align*}
\Var \left( F \left( \vec{x}_{1:M}^{(i)} \right) \right) = 
\frac{1}{M}\Var\left(f(\vec{x}^{(i)})\right) 
+ \frac{1}{M^2}\sum_{\ell,m=1}^M \Cov\left(f(\vec{x}^{(i)}_\ell), f(\vec{x}^{(i)}_m \right).
\end{align*}
Clearly, the first term decreases with increasing $M$.
For the second term, note that for sufficiently large $M$, the 
relative frequency of accepting a proposal $\vec{y}_\ell^{(i)}$ as
a sample among all 
$N+1$ proposals of the $i$th iteration is approximately equal to 
the stationary distribution $p(I=\ell|\vec{y}_{1:N+1}^{(i)})$. 
Therefore,
\begin{align}
\begin{split}
\frac{1}{M^2}\sum_{\ell,m=1}^M \Cov\left(f(\vec{x}^{(i)}_\ell), f(\vec{x}^{(i)}_m \right)
&\approx
\sum_{\ell,m=1}^{N+1} p(I=\ell|\vec{y}^{(i)}_{1:N+1}) p(I=m|\vec{y}^{(i)}_{1:N+1})  \\
&\quad\quad \cdot  \Cov\left(f(\vec{y}_\ell^{(i)}), f(\vec{y}_m^{(i)})  \right),
\end{split}
\label{eq:motivation_is_cov1}
\end{align}
which does not depend on $M$. Similarly,
\begin{align}
\Cov\left( F\left( \vec{x}_{1:M}^{(i)} \right), F\left( \vec{x}_{1:M}^{(j)} \right) \right) 
&= \frac{1}{M^2} \sum_{\ell,m=1}^M \Cov\left(f(\vec{x}_\ell^{(i)}), f(\vec{x}_m^{(j)})  \right)
\nonumber\\
\begin{split}
&\approx
\sum_{\ell, m =1}^{N+1}p(I=\ell|\vec{y}^{(i)}_{1:N+1}) p(I=m|\vec{y}^{(j)}_{1:N+1})\\
&\quad\quad \cdot \Cov\left(f(\vec{y}_\ell^{(i)}), f(\vec{y}_m^{(j)})  \right),
\end{split}
\label{eq:motivation_is_cov2}
\end{align}
for sufficiently large $M$, which again is independent of $M$. In summary, 
we have
\begin{align*}
\sigma_{M,N,n}^2 \gtrsim \sigma_{M',N,n}^2 \quad\quad \text{for} \quad M'>M.
\end{align*}
An increase of $M$ for a given iteration and increasing proposal numbers 
$N$ has been investigated numerically: for any $N$, we analysed the 
empirical variance of a mean estimate for $M\in\{2^\alpha N: \alpha=0,...,4\}$. 
The case of $\alpha = \infty$ corresponds to MP-MCMC with importance sampling
(IS-MP-MCMC),
described by Algorithm \ref{algorithm:importance_sampling_mp_mcmc}.
The underlying target is a one-dimensional Gaussian distribution, and we set the
proposal sampler to the SmMALA kernel defined in \eqref{eq:smMALA_kernel}. The 
corresponding results are displayed in Figure 
\ref{fig:mp-mcmc_varying_m} (\textit{left}). Indeed, an increase in $M$ 
yields a decrease in variance as expected. At the same time, the magnitude of the reduction in variance also decreases with increasing $M$. The limiting case, which corresponds to accepting all proposals and then suitably weighting them, exhibits the lowest variance. In some sense this contrast the general observation that an increased acceptance rate in MCMC does not necessarily produce more informative estimates.

\begin{figure}[h]
    \centering
    \begin{subfigure}[b]{0.65\textwidth}
        \includegraphics[width=\textwidth]{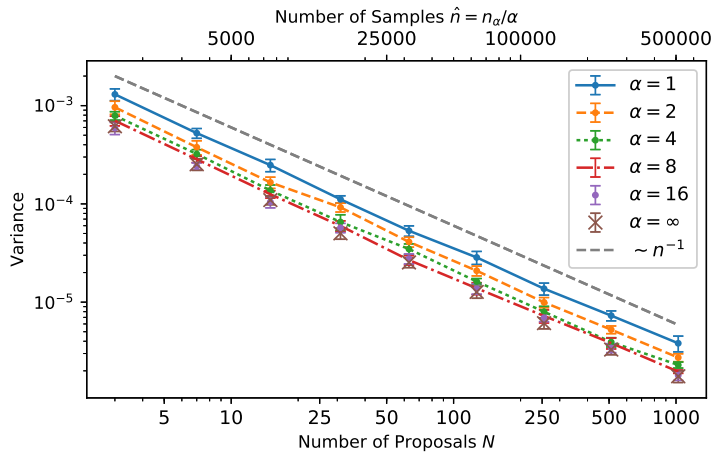}
        \label{a}
    \end{subfigure}
    \caption{Variance convergence of the arithmetic mean from MP-MCMC with Riemannian 
    SmMALA proposals on a one-dimensional standard Normal posterior for increasing proposal numbers 
    $N$ and for changing number of 
    accepted proposals per iteration $M=\const$ ({\it right}). Results based on 
    $10$ MCMC simulations. The error bars corresponds to three times a standard 
    deviation}
    \label{fig:mp-mcmc_varying_m}
\end{figure}


In the limiting case, $M\rightarrow \infty$, the two approximations in
\eqref{eq:motivation_is_cov1}-\eqref{eq:motivation_is_cov2} become equalities,
and sampling from the finite state chain in one iteration 
corresponds in principle to accepting all proposals $\vec{y}_{1:N+1}$ but 
weighting each $\vec{y}_i$ according to $p(I=i|\vec{y}_{1:N+1})$. This can 
be formalised as an importance sampling approach for MCMC with multiple 
proposals. A visualisation of this method is given in Figure
\ref{fig:importance_sampling_visualisation}.
Due to the considerations above, this approach typically
produces a smaller variance than the standard MP-MCMC, 
where $p(I|\vec{y}_{1:N+1})$ is used to sample from the finite state 
chain in every iteration. 

\begin{figure}%
\centering
\begin{tikzpicture}[scale=.3, font=\sffamily, dot/.style = 
					{state, fill=gray!20!white, line width=0.1mm, 
					inner sep=1pt, minimum size=0.5pt, minimum width=0.02cm},
					>=triangle 45]

\draw[rotate=45, black!50, line width=0.1mm] \boundellipse{0,1}{10}{4};
\draw[rotate=45, black!50, line width=0.1mm] \boundellipse{0.3,1.3}{7.5}{3};
\draw[rotate=45, black!50, line width=0.1mm] \boundellipse{0.6,1.6}{5}{2};
\draw[rotate=45, black!50, line width=0.1mm] \boundellipse{0.8,1.8}{2.5}{1};

\node[dot, minimum size=5pt, rectangle, label={[label distance=-7	pt]-80:{\small ${y}^{(k)}_1$}}] at (4.5, -4)     (x1l0) {}; 
\node[dot, minimum size=1.2pt, label={[label distance=-7pt]-60:{\small ${y}^{(k)}_2$}}] at (3.5, -8.0)     (x2l0) {}; 
\node[dot, minimum size=1.2pt, label={[label distance=-7pt]-60:{\small ${y}^{(k)}_3$}}] at (8.5, -2)     (x3l0) {}; 
\node[dot, minimum size=6pt, rectangle, label={[label distance=2pt]0:{\small ${y}^{(k)}_4=y^{(k+1)}_4$}}] at (2, 0)     (x4l0) {}; 
\node[dot, minimum size=4pt, label={[label distance=-5pt]-75:{\small ${y}^{(k)}_5$}}] at (0,-5.5)     (x5l0) {};

\node[dot, minimum size=3pt, label={[label distance=-0pt]0:{\small ${y}^{(k+1)}_1$}}] at (4, 4.5)     (x1l1) {}; 
\node[dot, minimum size=10pt, label={[label distance=-5pt]30:{\small ${y}^{(k+1)}_2$}}] at (-1., 3)     (x2l1) {}; 
\node[dot, minimum size=3.pt, label={[label distance=-5pt]-60:{\small ${y}^{(k+1)}_3$}}] at (-5, -3)     (x3l1) {}; 
\node[dot, minimum size=1.2pt, label={[label distance=-5pt]90:{\small ${y}^{(k+1)}_5$}}] at (-7., -0)     (x5l1) {};

\draw[dashed, red!60, line width=0.25mm] [->]  (x1l0) -- (x2l0) ;
\draw[dashed, red!60, line width=0.25mm] [->]  (x1l0) -- (x3l0) ;    
\draw[ red!60, line width=0.25mm] [->]  (x1l0) -- (x4l0) ;
\draw[dashed, red!60, line width=0.25mm] [->]  (x1l0) -- (x5l0) ;

\draw[dashed, red!60, line width=0.25mm] [->]  (x4l0) -- (x1l1) ;
\draw[dashed, red!60, line width=0.25mm] [->]  (x4l0) -- (x2l1) ;
\draw[dashed, red!60, line width=0.25mm] [->]  (x4l0) -- (x3l1) ;
\draw[dashed, red!60, line width=0.25mm] [->]  (x4l0) -- (x5l1) ;


\end{tikzpicture} 
\caption{\small In IS-MP-MCMC, we associate to every proposal 
${\vec{y}}^{(\ell)}_{i}$ its weight $w^{(\ell)}_i=p(I^{(\ell)}=i|\vec{y}_{1:N+1}^{(\ell)})$, 
thereby prioritising proposals that are most informative about the posterior}
\label{fig:importance_sampling_visualisation}
\end{figure}
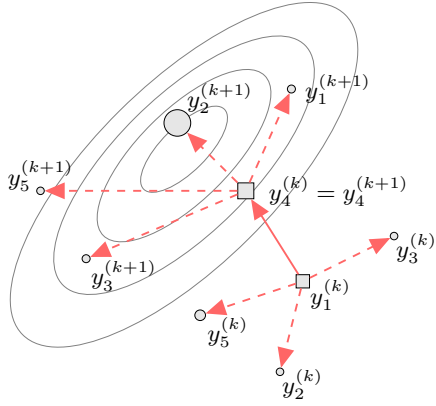

\subsubsection{Waste-Recyling}

Using a different heuristic, \cite{tjelmeland2004using} introduced the importance sampling technique from above, as well as \cite{ceperley1977monte,frenkel2004speed,frenkel2006waste,delmas2009does}.
In some of the literature, e.g.\ \cite{frenkel2006waste}, this technique is referred to as Waste-Recycling due to the fact that every proposal is used, including the ones rejected by MCMC. Compared to standard MP-MCMC, i.e.\ using Barker's acceptance probabilities, and when $M=1$, IS-MP-MCMC has been shown to be superior in terms of asymptotic variance \cite{delmas2009does}. However, \cite{delmas2009does} construct an example for the single proposal case
where importance sampling (Waste-recycling) can perform worse than the Metropolis-Hastings algorithm if Peskun's acceptance probability is employed.

\subsection{Importance sampling MP-MCMC Algorithms}
\label{subsubsec:algorithm_description_is_mpmcmc}

We now present the importance sampling version of MP-MCMC
to estimate the integral 
$\boldsymbol{\mu} = \boldsymbol{\mu}(f) = \int f(\vec{x})\pi(\vec{x})\mathrm{d}\vec{x}$
for a given function $f:\Omega\rightarrow \mathbb{R}^{d'}$ for $d'\in \mathbb{N}$. 
In every
iteration, each proposal $f(\vec{y}_i)$ is weighted according to the stationary
distribution $p(I=i|\vec{y}_{1:N+1})$. The sum of weighted proposals
yields an estimate for the mean $\boldsymbol{\mu}$. The resulting method is 
described by Algorithm \ref{algorithm:importance_sampling_mp_mcmc}.

Note that this defines a Markov chain driven algorithm: in every iteration, 
according to $p(\cdot|{{\vec{y}}}_{1:N+1})$ one sample
from the $N+1$ proposals is drawn (line $6$), conditioned on which then $N$ new 
proposals are drawn in the subsequence iteration. This chain corresponds 
to the standard MP-MCMC with $M=1$. When we mention the underlying Markov 
chain corresponding to importance sampling MP-MCMC, we refer to this chain.

The importance sampling estimator $\boldsymbol{\mu}_L=\boldsymbol{\mu}_L(f)$ 
for $L\in \mathbb{N}$ can also be written as
\begin{align}
\boldsymbol{\mu}_L = \frac{1}{L} \sum_{\ell=1}^L \sum_{i=1}^{N+1} w^{(\ell)}_i f({\vec{y}}_i^{(\ell)}),
\label{eq:def_importance_sampling_estimator}
\end{align}
where $ w^{(\ell)}_i  = p(I=i| {{\vec{y}}}_{1:N+1}^{(\ell)})$ for
$i=1,...,N+1$ and $\ell=1,...,L$.

\begin{algorithm}[h]
\SetAlgoLined
\KwIn{Initialise starting point (proposal) $\vec{y}_1\in \Omega$, number of proposals $N$, auxiliary variable $I=1$, \hl{integrand $f$ and initial mean estimate $\boldsymbol{\mu}_1 = \boldsymbol{\mu}_1(f)$}\; }
\For{\textnormal{each MCMC iteration $\ell=1,2,...$}}{
 Sample ${\vec{y}}_{\setminus I}$ conditioned on $I$, i.e., draw $N$ new points from the proposal kernel $\kappa({\vec{y}}_I, \cdot) = p({\vec{y}}_{\setminus I}|{\vec{y}}_I)$ \;
  Calculate the stationary distribution of $I$ conditioned on ${\vec{y}}_{1:N+1}$, i.e.\ $\forall$ $i=1,...,N+1$, $p(I=i|{{\vec{y}}}_{1:N+1}) = $ $\pi({{\vec{y}}}_i)\kappa({{\vec{y}}}_{{i}}, {{\vec{y}}}_{\setminus{i}}) / \sum_j \pi({{\vec{y}}}_j)\kappa({{\vec{y}}}_{{j}}, {{\vec{y}}}_{\setminus{j}})$, which can be done in parallel\;
\hl{ Compute $\tilde{\boldsymbol{\mu}}_{\ell+1}=\sum_{i} p(I=i| {{\vec{y}}}_{1:N+1})f({{\vec{y}}}_i)$}\;
\hl{Set $\boldsymbol{\mu}_{\ell+1} = \boldsymbol{\mu}_{\ell} + \frac{1}{\ell+1}\left(\tilde{\boldsymbol{\mu}}_{\ell+1} - \boldsymbol{\mu}_{\ell}\right)$}\;
{Sample new $I$ via the stationary distribution $p(\cdot|{{\vec{y}}}_{1:N+1})$}\;
 }
\caption{Importance sampling MP-MCMC \newline
All code altered compared to original MP-MCMC, Algorithm \ref{algorithm:multiproposal_MH},
is highlighted}
 \label{algorithm:importance_sampling_mp_mcmc}
\end{algorithm}

\subsubsection{Lack of samples representing $\pi$}

Despite the amenable properties of the importance sampling approach for
MP-MCMC, compared to the standard MP-MCMC, Algorithm 
\ref{algorithm:importance_sampling_mp_mcmc} has the disadvantage that it 
does not produce samples that are directly informative and approximately distributed according to a target, but rather an approximation of an integral with respect to the target instead.

\subsubsection{Adaptive importance sampling}
\label{subsubsec:algorithm_description_adaptive_IS_mpmcmc}

In many situations, it may make sense to adaptively 
learn the proposal kernel about the target, based on the past history of the algorithm.  We therefore extend Algorithm \ref{algorithm:importance_sampling_mp_mcmc} to make use of importance sampling, which is described in Algorithm \ref{algorithm:adaptive_importance_sampling_mp_mcmc}.
This can be achieved by making use of estimates for global parameters which are 
informative about the target, e.g.\ mean and covariance. Clearly, the Markov property
of the stochastic process resulting from this approach will not hold. This 
is generally
problematic for convergence, and thus for the consistency of the importance sampling
estimate. However, given the usual diminishing adaptation condition,
i.e.\ when the difference in subsequent updates converges to zero, and some further
assumptions on the transition kernel, referring to Section
\ref{subsec:adaptive_mpmcmc}, consistency can be shown. Since in the importance 
sampling method there are no actual samples generated following the target 
distribution asymptotically, but an estimate for an integral over the target, we
understand asymptotic unbiasedness of the resulting estimate when we talk about
consistency (\cite{delmas2009does}).

With the same notation as in Section \ref{subsec:theoretical_results_adaptive_is_mpmcmcm},
we assume that the proposal kernel $\kappa=\kappa_\gamma$ depends on a 
parameter $\gamma$ belonging to some space $\mathcal{Y}$. Examples of 
$\gamma$ are mean and covariance estimates of the posterior distribution.
A proof for asymptotic unbiasedness in the case where $\kappa$ is the Normal
distribution is given in Corollary \ref{cor:asympt_unbias_adapt_ismpmcmc}.
\noindent In the particular case where sampling proposals depends on
previous samples only though the adaptation parameters but is otherwise
independent,
i.e.\ $\kappa_{\gamma}(\vec{x}, \cdot) = \kappa_\gamma(\cdot)$ $\forall \vec{x} \in \Omega$,
we found the use of adaptivity most beneficial in applications for
the Bayesian logistic regression from \ref{adaptive_is_mp_qmcmc_empirical_results} 
and Bayesian linear regression from \ref{adaptive_is_mp_qmcmc_empirical_results},
compared to dependent proposals.

\begin{algorithm}[h]
\SetAlgoLined
\KwIn{Initialise starting point (proposal) $\vec{y}_1$, number of proposals $N$, auxiliary variable $I=1$, integrand $f$, initial mean estimate $\boldsymbol{\mu}_1 = \boldsymbol{\mu}_1(f)$ \hl{and
covariance estimate $\Sigma_1$}\;}
\For{\textnormal{each MCMC iteration $\ell=1,2,...$}}{
Sample ${\vec{y}}_{\setminus I}$ conditioned on $I$ \hl{and $\Sigma_\ell$},
i.e., draw $N$ new points from the proposal kernel $\kappa_{\text{\hl{$\Sigma_\ell$}}}({\vec{y}}_I,\cdot) = p({\vec{y}}_{\setminus I}|{\vec{y}}_I,$\hl{$\Sigma_\ell)$} \;
  Calculate the stationary distribution of $I$ conditioned on ${\vec{y}}_{1:N+1}$ and \hl{$\Sigma_\ell$}, i.e.\ $\forall$ $i=1,...,N+1$, $p(I=i|{{\vec{y}}}_{1:N+1},$\hl{$\Sigma_\ell$})$ = $ $\pi({{\vec{y}}}_i)\kappa_{\text{\hl{$\Sigma_\ell$}}}(
  {\vec{y}}_i, {{\vec{y}}}_{\setminus{i}}) / \sum_j \pi({{\vec{y}}}_j)\kappa_{\text{\hl{$\Sigma_\ell$}}}(
  {\vec{y}}_j,{{\vec{y}}}_{\setminus{j}})$, which can be done in parallel\;
 Compute $\tilde{\boldsymbol{\mu}}_{\ell+1}=\sum_{i} p(I=i| {{\vec{y}}}_{1:N+1}$,
 \text{\hl{$\Sigma_\ell$}}$)f({{\vec{y}}}_i)$\;
Set $\boldsymbol{\mu}_{\ell+1} = \boldsymbol{\mu}_{\ell} + \frac{1}{\ell+1}\left(\tilde{\boldsymbol{\mu}}_{\ell+1} - \boldsymbol{\mu}_{\ell}\right)$\;
  Sample new $I$ via the stationary distribution $p(\cdot|{{\vec{y}}}_{1:N+1},$ \hl{$\Sigma_\ell)$}\;
\hl{Compute $\tilde{{\Sigma}}_{\ell+1} = \sum_{i}p(I=i| {\vec{y}}_{1:N+1}, {{\Sigma}}_\ell)[{\vec{y}}_i-\vec{\mu}_{\ell+1}][{\vec{y}}_i-\vec{\mu}_{\ell+1}]^T$}\;
\hl{ Set ${\Sigma}_{\ell+1} = {\Sigma_\ell} + \frac{1}{\ell+1}(\tilde{{\Sigma}}_{\ell+1} - {\Sigma}_{\ell})$}\;
 }
\caption{Adaptive importance sampling MP-MCMC \newline
All code altered compared to IS-MP-MCMC, Algorithm \ref{algorithm:importance_sampling_mp_mcmc},
is highlighted 
}
 \label{algorithm:adaptive_importance_sampling_mp_mcmc}
\end{algorithm}

\subsection{Asymptotic unbiasedness of IS-MP-MCMC}
\label{subsec:theoretical_results_adaptive_is_mpmcmcm}

In this section, we prove the asymptotic unbiasedness of mean and covariance estimates from Algorithm \ref{algorithm:importance_sampling_mp_mcmc}
and Algorithm \ref{algorithm:adaptive_importance_sampling_mp_mcmc}, respectively. We further refer to an existing result in the literature which
states the asymptotic normality of the IS-MP-MCMC mean estimator.


\begin{Lemma}[Asymptotic unbiasedness of IS-MP-MCMC]
\label{lemma:asymptotic_unbiasedness_is}
Given that the underlying Markov chain is positive Harris,,
the IS-MP-MCMC sequence of estimators $(\boldsymbol{\mu}_L)_{L \ge 1}$ from Algorithm \ref{algorithm:importance_sampling_mp_mcmc} is asymptotically unbiased.
\begin{proof}
The statement is proven in Appendix \ref{proof:lemma:asymptotic_unbiasedness_is}.
\end{proof}
\end{Lemma}

Lemma \ref{lemma:asymptotic_unbiasedness_is} states that after having discarded a
sufficiently large burn-in period of (weighted) samples, the importance sampling 
estimator defined by the remaining samples is unbiased.

\begin{corollary}[Asymptotic unbiasedness of adaptive IS-MP-MCMC]
\label{cor:asympt_unbias_adapt_ismpmcmc}
Under any of the conditions stated in
Theorem \ref{thm:ergodicity_adap_mpmcmc_independent},
Theorem \ref{thm:ergodicity_adapt_mpmcmc_bounded} or
Theorem \ref{thm:ergodicity_adaptive_mpmcmc_positive}, 
the sequence of estimators 
$(\boldsymbol{\mu}_L)_{L \ge 1}$ from Algorithm
\ref{algorithm:adaptive_importance_sampling_mp_mcmc} is asymptotically unbiased.
\begin{proof}
Ergodicity of the adaptive MP-MCMC follows by the respective theorem used.
Thus, we may argue analogously to the proof of Lemma
\ref{lemma:asymptotic_unbiasedness_is}.
\end{proof}
\end{corollary}

\begin{corollary}
\label{corollary:asymptotic_unbiasedness_is_covariance}
Under the same conditions as in Lemma \ref{lemma:asymptotic_unbiasedness_is},
the sequence of covariance estimates $({\Sigma}_{L})_{L\ge 1}$ from Algorithm
\ref{algorithm:adaptive_mp_mcmc} and Algorithm
\ref{algorithm:adaptive_importance_sampling_mp_mcmc}
is asymptotically unbiased. 
\begin{proof}
For a proof, we refer to Appendix \ref{proof:corollary:asymptotic_unbiasedness_is_covariance}.
\end{proof}
\end{corollary}

The following result states that IS-MP-MCMC produces asymptotically normal estimates,
that outperform standard MP-MCMC (using Barker's acceptance probabilities
\eqref{eq:barker_acceptance}) for $M=1$ in terms of their asymptotic variance. Thus,
making use of all proposals in every iteration is better than accepting only a single 
one per iteration.

\begin{Lemma}[Proposition 4.1, \cite{delmas2009does}]
The IS-MP-MCMC sequence of estimators $(\boldsymbol{\mu}_L)_{L \ge 1}$ from Algorithm \ref{algorithm:importance_sampling_mp_mcmc} is asymptotically normal. Its asymptotic
variance is smaller or equal to the asymptotic variance of MP-MCMC with $M=1$.
\end{Lemma}

\subsection{Bayesian logistic regression}
\label{subsubsec:emp_results_adaptive_IS_mp_mcmc}

In what follows we consider the Bayesian logistic regression model as formulated in \cite{girolami2011riemann}. The dependent 
variable $y$ is categorical with binary outcome. The probability of $y$ 
is based on predictor variables defined by the design matrix 
$X \in \mathbb{R}^{n\times d}$, and is given by 
$P(y=1|X,\vec{\theta})=\sigma(X \vec{\theta})$ and 
$P(y=0|X,\theta)=1-\sigma(X \theta)$, where $\sigma$ denotes the logistic 
function. Our goal is to perform inference over the regression parameter 
$\vec{\theta} \in \mathbb{R}^d$, which has the Gaussian prior 
$\pi(\theta) = \mathcal{N}(\vec{0},\alpha \mathbf{I}_d)$, with $\alpha=100$.
For further details, we refer to \cite{girolami2011riemann}. For the above mentioned logistic regression, there are 
overall $5$ different underlying data sets of varying dimensionality at our disposal, which we denote by \textit{Ripley} (d=3), \textit{Pima} (d=8), \textit{Heart} (d=14), \textit{Australian} (d=15) and \textit{German} (d=25). For brevity, we consider only the lowest-dimensional model data in the following experiments. In later experiments where we use a QMC seed to run the above introduced MCMC algorithms, we investigate their performance on all data sets.

\subsubsection{Empirical results}

We now compare the performance of IS-MP-MCMC and adaptive IS-MP-MCMC in the context of the Bayesian logistic regression model introduced above. As a reference we also consider the standard, i.e.\ single proposal, random-walk Metropolis-Hastings algorithm. To ensure fairness the total number of samples $n$ produced by the single and multiple proposal algorithms are equal, i.e.\ $n=LN$ if $L$ denotes the number if iterations and $N$ the number of proposals in the multiple proposal case.

In all algorithms we choose a Gaussian proposal sampler. For the importance sampling methods, proposals are generated independently of previous samples. As an initial proposal mean and covariance a rough estimate of the posterior mean and covariance is employed. The former is used to initialise the Metropolis-Hastings algorithm. In the adaptive algorithm, proposal mean and covariance estimates are iteratively updated after every iteration. The results for the empirical variance associated to the posterior mean estimates in the above mentioned algorithms are displayed in Figure \ref{fig:is_vs_adapt_mpmcmc}. The importance sampling algorithms outperform Metropolis-Hastings by over an order of magnitude. Further, the adaptive algorithm produces slightly better results than the non-adaptive importance sampler, with an average empirical variance reduction of over $20 \%$.

\begin{figure}[h]
    \centering
    \begin{subfigure}[b]{0.65\textwidth}
        \includegraphics[width=\textwidth]{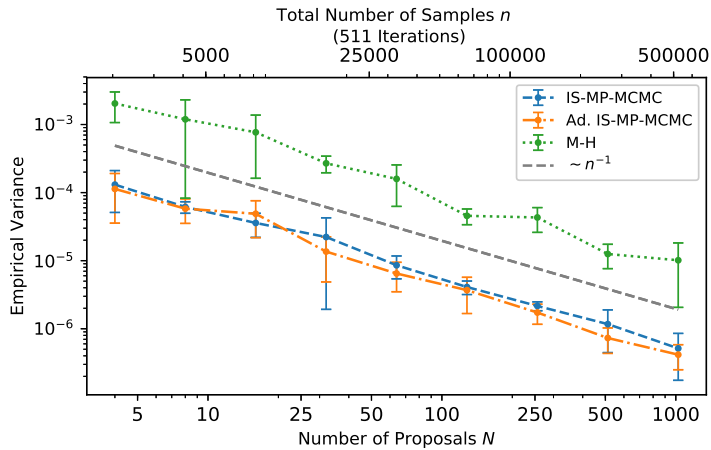}
        \label{c}
    \end{subfigure}
    \caption{Empirical variance of the arithmetic mean for IS-MP-MCMC and adaptive IS-MP-MCMC in the Bayesian logistic regression model
    from \cite{girolami2011riemann} ($d=3$) for increasing proposal numbers 
    $N$ and total number of samples $n$; also displayed is Metropolis-Hastings (M-H) for reference. Here, M-H was tuned
to an approximately optimal acceptance rate of $20$-$25 \%$.
    Results are based on $25$ MCMC 
    simulations. The error bars correspond to three times a standard deviation}
    \label{fig:is_vs_adapt_mpmcmc}
\end{figure}

\section{Combining QMC with MP-MCMC}
\label{sec:multiproposal_quasi_MH}


In this section, we introduce a general framework for using CUD numbers
in the context of MP-MCMC, leading to a method we shall call MP-QMCMC. The motivation for this is that since in each iteration $N$ proposals are provided as alternatives to any current state, all of which contribute to the exploration of the underlying local region of state space, we might reasonably expect a higher gain using QMC points for these than for the single proposal case, where there is always only one alternative to the current state.  We prove the consistency of our proposed method and illustrate in numerical experiments an increased performance.  We also extend our methodology to importance sampling, for which we observe an improved rate of convergence of close to $n^{-2}$ in simulations, instead of the standard MCMC rate of $n^{-1}$ in MSE. Summarising, we generalise Algorithms \ref{algorithm:multiproposal_MH},
\ref{algorithm:importance_sampling_mp_mcmc} and \ref{algorithm:adaptive_importance_sampling_mp_mcmc} to using any choice of
CUD numbers as the driving sequence.

\subsection{MP-QMCMC}

We now introduce the above mentioned MP-QMCMC algorithm, and subsequently prove its consistency based on regularity conditions. Moreover, we investigate the algorithm's performance by simulations based on the Bayesian logistic regression model introduced in Section \ref{subsubsec:emp_results_adaptive_IS_mp_mcmc} and $5$ different underlying data sets. We conclude the section with some practical advice on effectively harvesting the benefits of QMC based MCMC algorithms.

\subsubsection{Algorithm description}
\label{subsubsec:algorithm_description_mpqmcmc}

We exchange the usual driving sequence of pseudo-random numbers by CUD numbers. In every iteration, this comprises two situations: in  the first, proposals are generated, and in the second one, given the transition matrix for the finite state chain on the proposed states, auxiliary variables are generated. In order to create the $N$ new proposals in $\Omega \subset \mathbb{R}^d$ 
we utilise $Nd$ numbers from a CUD sequence in $(0,1)$. Further, to sample from 
the finite state chain $M$ times we utilise another $M$ numbers from the underlying CUD sequence. Our algorithm is designed such that the entire CUD sequence is used, thereby making full use of the spatial homogeneity. A pseudo-code description of the resulting MP-QMCMC is given in Algorithm \ref{algorithm:multiproposal_quasi_MH}, which represents an extension of Algorithm \ref{algorithm:multiproposal_MH} to using any choice of CUD numbers as the driving sequence. 

The function $\Psi_{\vec{y}_I}$
in line three of the algorithm denotes the inverse of the cumulative
distribution function (CDF) of the proposal distribution. Thus, 
given $N$ vectors $\vec{u}_i \in (0,1)^d$, represented by the 
joint vector $\vec{u}=(\vec{u}_1, ..., \vec{u}_N)$, $\Psi_{\vec{y}_I}$
assigns $\vec{u}$ to $N$ new proposals $\vec{y}_{\setminus{I}}\in \Omega^N$.
Practically, each sub-vector $\vec{u}_i$ of $\vec{u}$ is assigned to one
new proposal. For $d>1$, sampling via inversion can be expressed as
iteratively sampling from the one-dimensional conditional proposal
distribution, i.e.\ sampling the first coordinate of the new proposal
via inversion of the conditional CDF of the first coordinate, then 
given that coordinate sampling the second one accordingly and so on.

\begin{algorithm}[H]
\SetAlgoLined
\KwIn{Initialise starting point ${\vec{x}}_0=\vec{y}_1\in\Omega$, number of proposals $N$, 
number of accepted samples per iteration $M$, auxiliary variable $I=1$, counter $n=1$ 
\hl{and number of MCMC iterations $L$}\;}
\hl{Generate a CUD sequence $u_1, ..., u_{L(Nd+M)}\in (0,1)$}\; 
\For{\textnormal{each MCMC iteration $\ell=1,...,$\hl{$L$}  }}{
\hl{Set $\vec{u} = (u_{(\ell-1)(Nd+M)+1}, ..., u_{(\ell-1) (Nd+M)+ Nd}) \in (0,1)^{Nd}$}, and sample $\vec{y}_{\setminus I}$ conditioned on $I$, i.e., draw $N$ new points from $\kappa(\vec{y}_I, \cdot) = p(\vec{y}_{\setminus I}|\vec{y}_I)$ \hl{by the inverse $\Psi_{\vec{y}_I}(\vec{u})$} \;
Calculate the stationary distribution of $I$ conditioned on $\vec{y}_{1:N+1}$,
i.e.\ $\forall$ $i=1,...,N+1$, $p(I=i|\vec{y}_{1:N+1}) = $ $\pi(\vec{y}_i)\kappa(\vec{y}_i,\vec{y}_{\setminus{i}}) / \sum_j \pi(\vec{y}_j)\kappa(\vec{y}_j,\vec{y}_{\setminus{j}})$, which can be done in parallel\;
\hl{Set $\vec{v}' = (u_{(\ell-1) (Nd+M)+Nd+1}, ..., u_{\ell (Nd+M)}) \in (0,1]^M$}\;
  \For{$m=1,...,M$}{
 \hl{If $v'_{m} \in (\gamma_{j -1}, \gamma_{j}]$, where $\gamma_{j} = \sum_{i=1}^j p(I=i| \vec{y}_{1:N+1})$ 
 for $j= 1,...,N+1$ and $\gamma_0:=0$, set $x_{n+m}=\vec{y}_j$}\;
 }
 Update counter $n=n+M$
 }
\caption{MP-QMCMC \newline
All code altered compared to original (pseudo-random) MP-MCMC, Algorithm \ref{algorithm:multiproposal_MH},
is highlighted
}
\label{algorithm:multiproposal_quasi_MH}
\end{algorithm}

\subsubsection{Consistency}

In the following section we prove that MP-MCMC driven by CUD numbers 
instead of pseudo-random numbers produces samples according to the 
correct stationary distribution under regularity conditions. In 
general, using CUD points is not expected to yield consistency for 
any MCMC algorithm that is not ergodic when sampling with IID 
numbers \cite{chen2011consistency}. Similar to Chen et al., our 
proof is based on the so called Rosenblatt-Chentsov transformation.

\vspace{4mm}
\noindent \textbf{Rosenblatt-Chentsov transformation}
\vspace{1.5mm}
 
\noindent Let us assume that there is a generator $\psi_\pi$ that produces 
samples according to the target distribution $\pi$, i.e.\ 
$\psi_\pi(\vec{u})=\vec{x} \sim \pi$ if
$\vec{u}\sim \mathcal{U}[0,1]^d$. For example, $\psi_\pi$ could
be based on the inversion method applied to the one-dimensional conditional
distributions of the target, assumed that they are available.
For $n=LM$ the $N$-proposal Rosenblatt-Chentsov transformation of 
$\vec{u}_0 \in(0,1)^d$ and a finite 
sequence of points $u_1, ..., u_{L(Nd+M)}\in (0,1)$ is defined as 
the finite sequence $\vec{x}_0,\vec{x}_1, ...\vec{x}_{LM}\in \Omega$, where 
$\vec{x}_0 = \psi_\pi(\vec{u}_0)$ and $\vec{x}_1, ...\vec{x}_{LM}$ are generated
according to Algorithm \ref{algorithm:multiproposal_quasi_MH} using 
$u_1, ..., u_{L(Nd+M)}\in (0,1)$ as driving sequence and $\vec{x}_0$ as initial 
point. 

Since the standard version of
MP-MCMC fulfills the detailed balance condition, updating samples preserves
the underlying stationary distribution $\pi$. Thus, whenever one sample 
follows $\pi$, all successive samples follow $\pi$. That means, whenever 
$\vec{x}_0\sim\pi$, all points generated by MP-MCMC follow $\pi$. If the sequence 
of points $u_1, ..., u_{L(Nd+M)}$ in the $N$-proposal Rosenblatt-Chentsov transformation 
are uniformly distributed, then this holds for the samples generated by
Algorithm \ref{algorithm:multiproposal_quasi_MH}, too. This observation will be used
in the following to show the consistency of MP-QMCMC. Before that, we formulate
some regularity condition that will be used in the proof.

\vspace{4mm}
\noindent \textbf{Regularity conditions}
\vspace{1.5mm}

\noindent Similarly to \cite{chen2011consistency}, the consistency proof 
which is 
given below relies on two regularity conditions. The first one defines 
coupling properties of the sampling method, and the second one suitable 
integrability over the sample space.

\vspace{2mm}

\noindent \textbf{1) Coupling}: Let
$\phi(\tilde{\vec{x}}_M, (u_1, ..., u_{Nd+M})) = (\vec{x}_1,...,\vec{x}_M) $ denote the
innovation operator of MP-MCMC, which assigns to the last sample $\tilde{\vec{x}}_M$ 
of the current iteration the $M$ new samples from the subsequent iteration.
Let $\mathcal{C}\subset (0,1)^{Nd+M}$ have positive Jordan measure. If for any
$\vec{u}\in \mathcal{C}$ it holds 
$\phi(\vec{x}, \vec{u}) = \phi(\vec{x}', \vec{u})$ $\forall$ 
$\vec{x},\vec{x}' \in \Omega$, then $\mathcal{C}$ is called a 
\textit{coupling region}. \\
Let $\vec{z}_i = \phi(\vec{x}, \vec{u}_i)$ and 
$\vec{z}_i' = \phi(\vec{x}', \vec{u}_i)$ be two iterations from
Algorithm \ref{algorithm:multiproposal_quasi_MH} based 
on the same innovations $\vec{u}_i \in (0,1)^{Nd+M}$ but possibly 
different current states $\vec{x}$ and $\vec{x}'$, respectively.. 
If $\vec{u}_i\in \mathcal{C}$, then $\vec{z}_j = \vec{z}_j'$ for any $j\ge i$. 
In other words, if 
$\vec{u}_i \in \mathcal{C}$, two chains with the same innovation operator
but potentially different starting points coincide for all $j \ge i$.
As a non-trivial example, 
standard MP-MCMC with independent proposal sampler has a coupling 
see Lemma \ref{lemma:coupling_region_barker_independent}.

\vspace{2mm}
\noindent \textbf{2) Integrability}: For $k\ge 1$, let
$\vec{x}_k=\vec{x}^i_j=\vec{x}^i_j(\vec{u}^1,...,\vec{u}^i)$ 
with $i\ge 1, 1\le j\le M$ and $k=(i-1)M+j$ denote the $k$th 
$N$-proposal MCMC update, i.e.\ the $j$th sample in the $i$th 
iteration, according to Algorithm 
\ref{algorithm:multiproposal_quasi_MH}. The method is called 
regular if the function $g:(0,1)^{i(Nd+M)}\rightarrow \mathbb{R}$, 
defined by
$g(\vec{u^1}, \ldots, \vec{u}^i) = f(\vec{x}_k(\vec{u^1}, \ldots, \vec{u}^i))$, 
is Riemann integrable for any bounded and continuous scalar-valued
$f$ defined on $\Omega$.


\vspace{2mm}
\noindent With reference to \cite{chen2011consistency}, it may seem odd at first 
to use the Riemann integral instead of the Lebesgue integral in 
the previous formulation. However, QMC numbers are typically designed to meet equi-distribution criteria over rectangular sets or are based upon a spectral condition, such that both have a formulation that is naturally closely related to the Riemann integral.

The following theorem is the main result of this section, and states 
that under the above conditions MP-MCMC is consistent when driven 
by CUD numbers.

\begin{theorem}
\label{theorem:consistency_multi_prop_mcmc}
Let $\vec{x}_0 \in \Omega$. For $k\ge 1$, let 
$\vec{x}_k=\vec{x}^i_j$ with $i\ge 1, 1\le j\le M$ 
be the $k$th $N$-proposal MCMC update according to 
Algorithm \ref{algorithm:multiproposal_quasi_MH}, 
which is assumed to be positive Harris with stationary 
distribution $\pi$ for an IID random seed. 
The method is also assumed to be regular and to have a coupling 
region $\mathcal{C}$. Further, let 
\begin{align*}
\vec{u}^i = (v^i_1, ..., v^i_{Nd+M}),
\end{align*}
for a CUD sequence $(v_i)_{i\ge 0}$ with $v^i_\ell := v_{i(Nd+M)+\ell}$
for $i\ge 0$ and $\ell=1,...,Nd+M$. Then, 
the sequence $(\vec{x}_k)_{k \ge 1}$ consistently samples $\pi$.
\end{theorem}
\begin{proof}
A proof can be found in Appendix \ref{app:sec:proof_consistency_mpqmcmc}.
\end{proof}

Instead of requiring a coupling region, consistency for a continuous 
but bounded support $\Omega$ of $\pi$ can be achieved in the classical single
proposal case by using a 
contraction argument \cite{chen2011consistencythesis,chen2011consistency}.
Given an update function both continuous on the last state and 
the innovations, one further requires continuity and integrability 
conditions.

In the following lemma, we show that standard MP-MCMC, i.e.\ 
Algorithm \ref{algorithm:multiproposal_MH}, has a coupling region
when proposals are sampled independently of previous samples.

\begin{Lemma}[Coupling region for MP-MCMC with independent sampler]
\label{lemma:coupling_region_barker_independent}

Let $\Psi_{\vec{y_I}}$ denote the inverse of the proposal distribution and $\vec{y}_I$
the last accepted sample from the previous MP-MCMC iteration, i.e.\
\begin{align}
\Psi_{\vec{y_I}}(\vec{u}_1, \ldots, \vec{u}_{N}) = \vec{y}_{\setminus{I}},
\label{eq:proposal_sampler_coupling_region}
\end{align}
are the new proposals in one MCMC iteration, where $\vec{u}_i \in (0,1)^{d}$
for $i=1, ..., N$, and $I \neq 1$ without loss of generality. 
We assume that proposals are sampled independently of previous samples, i.e.\
$\Psi_{\vec{y_I}} = \Psi$ and 
$\kappa(\vec{y}_I, \vec{y}_{\setminus{I}}) = \kappa(\vec{y}_{\setminus{I}})$. 
The proposal $\vec{y}_1$ is always accepted, i.e.\ 
$\phi(\vec{y_I}, (\vec{u}_1, \ldots, \vec{u}_{N}, \vec{\tilde{u}}))=(\vec{y}_1, ..., \vec{y}_1)$
with $\vec{\tilde{u}}\in (0,1)^M$, if
\begin{align*}
0 \le \tilde{u}_{m} \le \frac{\pi(\vec{y}_1)K(\vec{y}_{\setminus 1})}{\sum_i \pi(\vec{y}_i)\kappa( \vec{y}_{\setminus i})}
\quad\quad \text{for all }m=1,...,M.
\end{align*} 
Let us assume that
\begin{align}
\rho = \sup_{\vec{y}_1,\ldots, \vec{y}_{N+1}\in \Omega} \frac{\sum_{i=1}^{N+1} \pi(\vec{y}_i)}{\kappa(\vec{y}_{\setminus 1})} < \infty.
\label{eq:kappa_sup_condition}
\end{align}
Moreover, let us assume that there is a rectangle 
$[\vec{a}, \vec{b}]\subset [0,1]^{Nd}$ 
of positive volume with
\begin{align}
\eta = \inf_{\substack{(\vec{u}_1, \ldots, \vec{u}_{N}) \in [\vec{a}, \vec{b}],\\
			\vec{y}_I\in \Omega}} 
\frac{\pi(\Psi(\vec{u}_1))}{\prod_{j \neq I}\kappa(\Psi(\vec{u}_j)) 
+ \sum_{i\neq I} \kappa(\vec{y}_I) \prod_{j \neq i,I}\kappa(\Psi(\vec{u}_j))\cdot } >0.
\label{eq:eta_inf_condition}
\end{align}
Then, $\mathcal{C}:= [\vec{a}, \vec{b}] \times [0, \eta/\rho]^M$ is a coupling region.
\begin{proof}
Let us assume that $(\vec{u}_1, ..., \vec{u}_N, \vec{\tilde{u}})\in \mathcal{C}$,
where $\vec{u}_i \in (0,1)^d$ for $i=1,...,N$ and $\vec{\tilde{u}} \in (0,1)^M$.
The set $\mathcal{C}$ has by definition positive Jordan measure. Note that by the
definition of $\rho$ and $\eta$, and equation \eqref{eq:proposal_sampler_coupling_region},
it holds
\begin{align}
\kappa(\vec{y}_{\setminus 1}) \ge \frac{1}{\rho}\sum_{i=1}^{N+1} \pi(\vec{y}_i), \quad 
\text{and}, \quad\quad
\pi(\vec{y}_1) \ge \eta \sum_{i=1}^{N+1} K(\vec{y}_{\setminus i}).
\end{align}
Therefore,
\begin{align*}
\pi(\vec{y}_1)\kappa(\vec{y}_{\setminus 1}) 
&\ge \eta \sum_{i=1}^{N+1} 
\kappa(\vec{y}_{\setminus i})\cdot \frac{1}{\rho}\sum_{i=1}^{N+1} \pi(\vec{y}_i) \\
&\ge \tilde{u}_{m} \sum_{i=1}^{N+1} \pi(\vec{y}_i)\kappa(\vec{y}_{\setminus i})
\end{align*}
for any $m=1,...,M$. Hence, 
$\phi(\vec{y}_I, \vec{u}_1, \ldots, \vec{u}_{N}, \vec{\tilde{u}})=(\vec{y}_1, ..., \vec{y}_1)$ 
for any $\vec{y}_I \in \Omega$.
\end{proof}
\end{Lemma}

Following a similar argumentation to Section 5.3 and 5.4 in
\cite{chen2011consistency}, one can prove that the Rosenblatt-Chentsov
transformation is regular under continuity and boundedness assumptions.

\begin{theorem}
\label{theorem:regularity_rosenblatt_chentsiv}
If boundedness and continuity holds for the
generator $\Psi_\pi$, the inverse CDF 
$\Psi_{\cdot}: \Omega \times (0,1)^{Nd}\rightarrow \Omega^{N}$, 
the proposal density $\kappa$ and the target density $\pi$,
then the Rosenblatt-Chentsov transformation is regular.
\begin{proof}
Here, we refer to Theorem 6 in \cite{chen2011consistency}. The
idea of the proof is to use that
the composition of a continuous function defined on a bounded interval
with Riemann integrable functions leads again to a Riemann integrable
function, and that the sampling step on
the finite state chain does not break Riemann integrability.
\end{proof}
\end{theorem}

\begin{example}[MP-MCMC with independent Gaussian proposals]
Let the proposals be Gaussian, and independently sampled 
from previous samples. Further, let the target distribution be bounded 
and continuous.
According to Lemma \ref{lemma:coupling_region_barker_independent} 
and Theorem \ref{theorem:regularity_rosenblatt_chentsiv}, the 
resulting MP-MCMC satisfies the regularity conditions that ensure
consistency when run by a CUD driving sequence.
\end{example}

\subsubsection{Empirical Results: Bayesian logistic regression}

In Figure \ref{fig:variance_bias_mpmcmc_qmc_vs_psr} we compare the performance of SmMALA MP-MCMC 
with its CUD driven counterpart on a Bayesian logistic regression problem from Section \ref{subsubsec:emp_results_adaptive_IS_mp_mcmc}
for increasing proposal numbers and sample sizes. What we compare is the empirical variance and the squared
bias of estimates for the posterior mean. Since the actual posterior
mean is not analytically available, we computed a gold-standard mean based on
$25$ simulations of approximately $8 \cdot 10^6$ weighted proposals according
to Algorithm \ref{algorithm:importance_sampling_mp_qmcmc}. The bias is then
calculated using the gold-standard mean estimate. Here, a proposal mechanism is applied that makes use of an auxiliary proposed state, as described in Section \ref{intro_aux_prop_state}. More precisely, in a first step of the proposal procedure, an auxiliary point is generated independently of previous samples. In a second step, based on geometric information on the auxiliary point, the $N$ proposals, from which MP-MCMC samples are drawn, are generated using the SmMALA kernel \eqref{eq:smMALA_kernel}. Note that formally, this algorithm has an independent proposal sampler, which is why this algorithm satisfies the coupling region condition according to Lemma \ref{lemma:coupling_region_barker_independent} and therefore consistently samples from the posterior.

In the lower dimensional models, we see a slight improvement in the rate of convergence for increasing proposal numbers and sample sizes between pseudo-random and QMC driven SmMALA MP-MCMC. However, the improvement basically disappears in higher dimensions. This may first seem as a setback in exploring beneficial aspects of QMC driven MCMC methods since actually we hope for significantly improved rates of convergence. However, some further thought may explain the observed behaviour, which is done in what follows, and identify its source. As a result we are able to implement a methodology presented later in this work that circumvents the problem and therefore allows for significantly improved convergence rates for certain QMC driven MCMC methods compared to pseudo-random methods.

\begin{figure}[h]
    \centering
    \begin{subfigure}[b]{0.45\textwidth}
        \includegraphics[width=\textwidth]{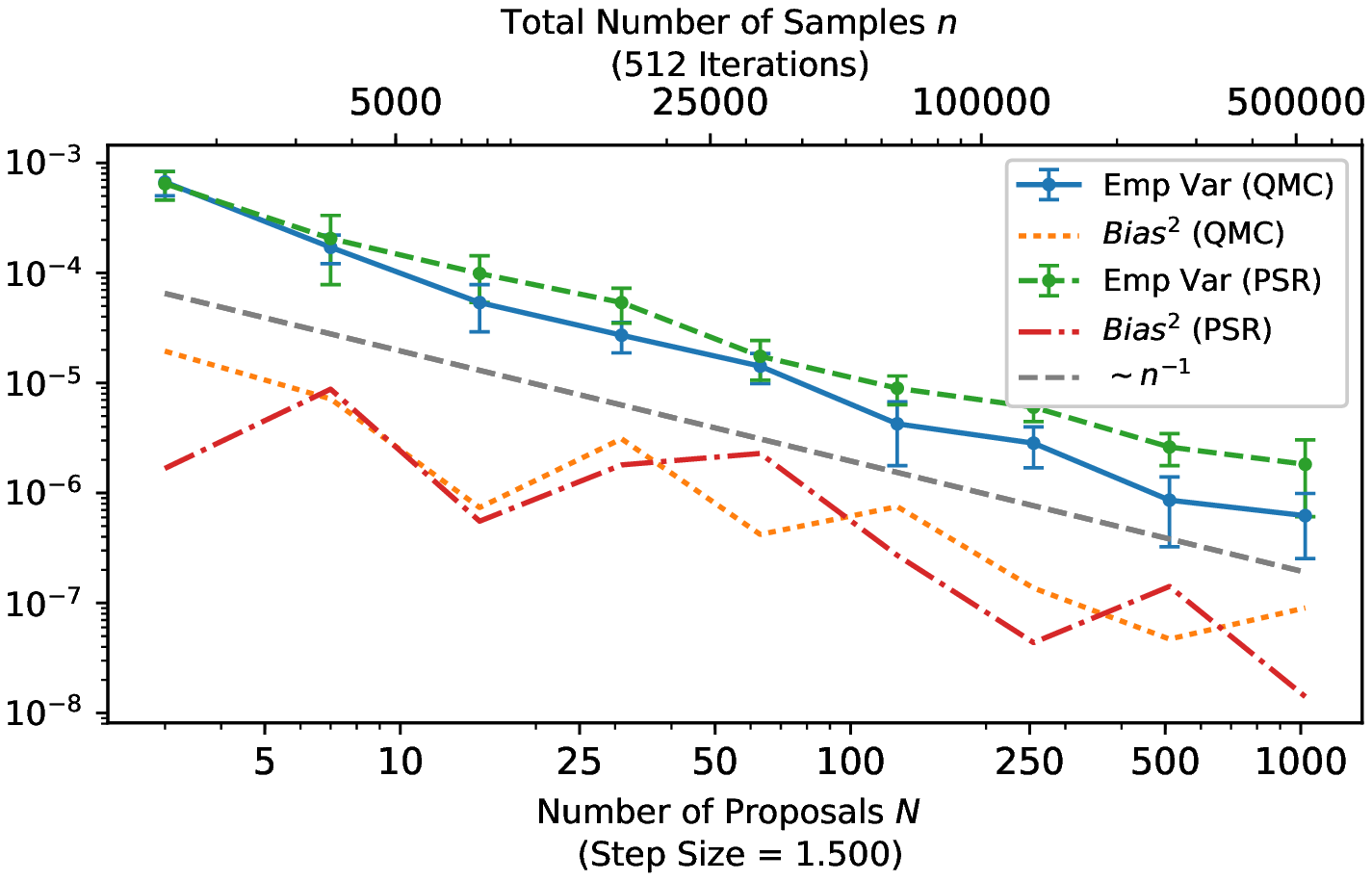}
        \subcaption{Ripley, $d=3$}
    \end{subfigure} 
        \begin{subfigure}[b]{0.45\textwidth}
        \includegraphics[width=\textwidth]{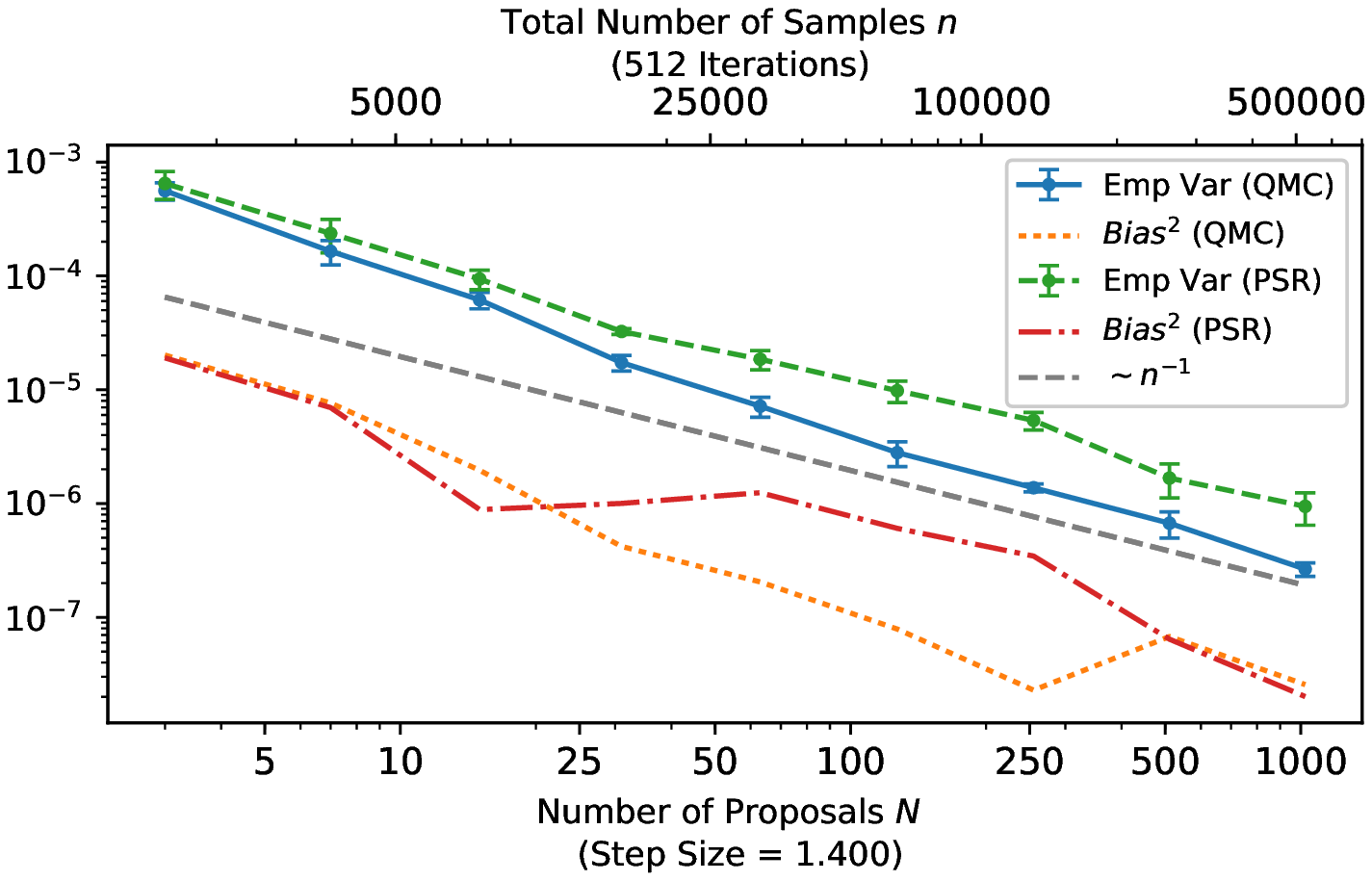}
        \subcaption{Pima, $d=8$}
    \end{subfigure} 
    \begin{subfigure}[b]{0.45\textwidth}
        \includegraphics[width=\textwidth]{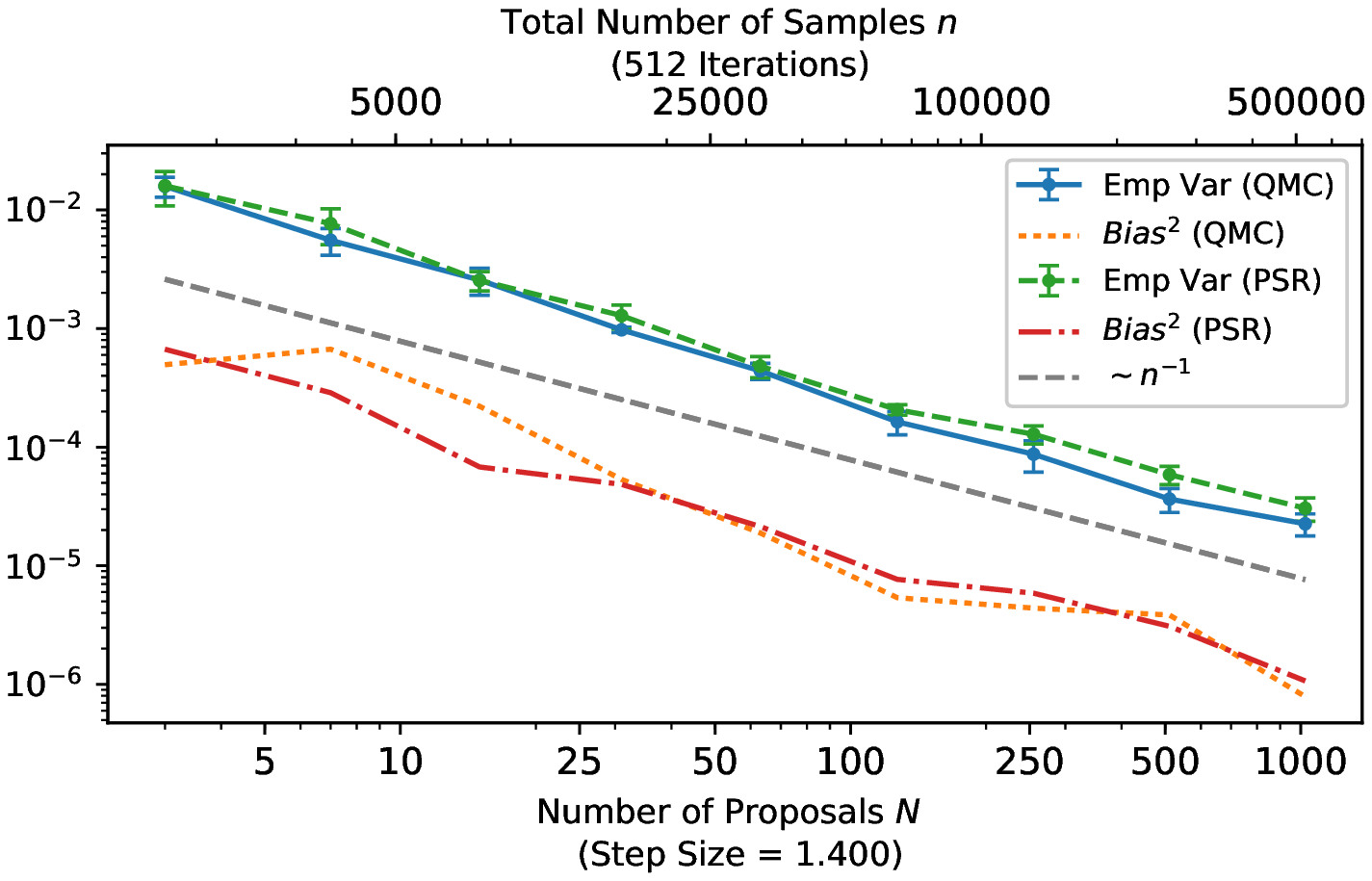}
        \subcaption{Heart, $d=14$}
    \end{subfigure} 
        \begin{subfigure}[b]{0.45\textwidth}
        \includegraphics[width=\textwidth]{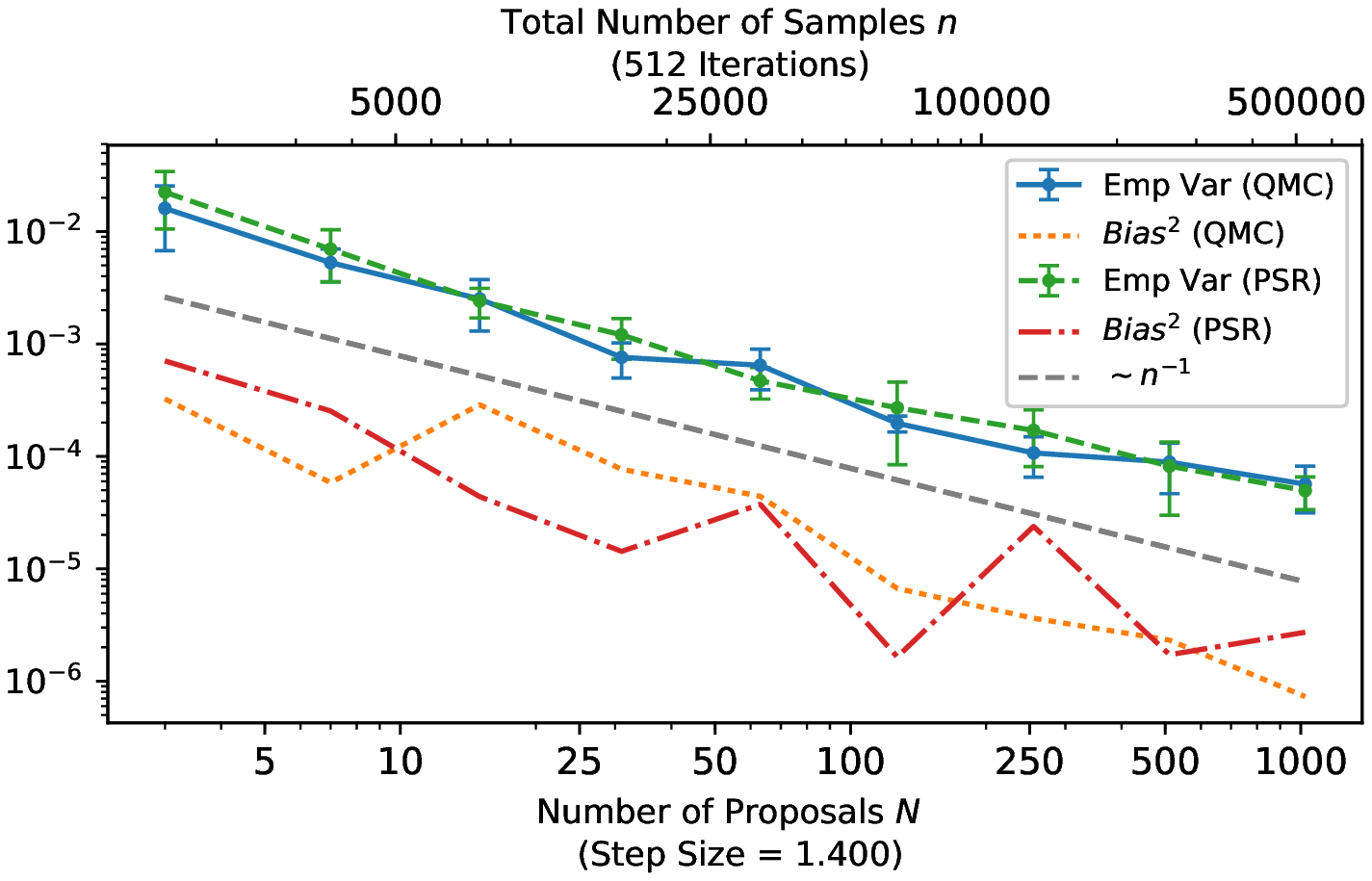}
        \subcaption{Australian, $d=15$}
    \end{subfigure}     
        \begin{subfigure}[b]{0.45\textwidth}
        \includegraphics[width=\textwidth]{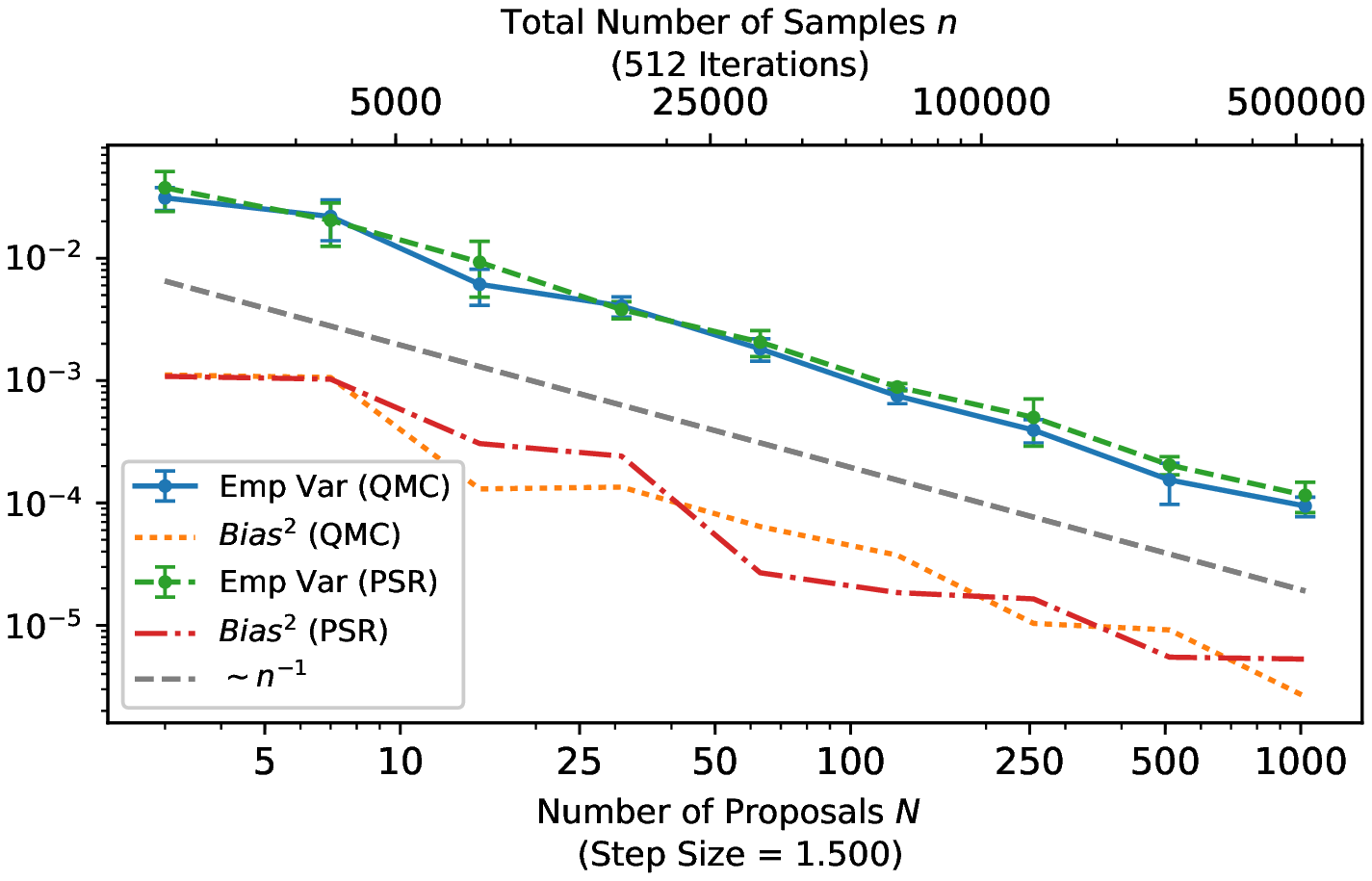}
        \subcaption{German, $d=25$}
    \end{subfigure}       
    \caption{\small{Empirical variance and squared bias considered in an 
    estimation associated to the Bayesian logistic regression problem 
    from \cite{girolami2011riemann} using a pseudo-random (PSR) vs.\ CUD (QMC) seed, 
    resp., for increasing proposal numbers and sample sizes. The results 
    are based on $25$ MCMC simulations and the error bars 
    correspond to three times a standard deviation}}
    \label{fig:variance_bias_mpmcmc_qmc_vs_psr}
\end{figure}

\subsubsection{The ``right'' way of using CUD numbers}
\label{subsubsec:right_way}

Due to the acceptance threshold in standard MP-MCMC, some proposals are accepted
at least once while others are not accepted at all. In the latter case,
the underlying QMC points that generate such proposals and that contribute 
to the homogeneous spatial coverage of the hypercube of QMC points
are neglected. Since this homogeneity is the core of performance gain due to
QMC, we do not expect the best results using CUD points when applying it to standard MP-MCMC. More successful approaches do not exhibit any discontinuity due to an acceptance threshold, e.g.\ Gibbs sampling \cite{owen2005quasi,tribble2008construction}. Our approach for making use of all information carried by the underlying CUD sequence is by incorporating all proposals, which is achieved using the importance sampling approach for MP-MCMC introduced in Section \ref{sec:importance_sampling}. Indeed, as we will see in the following section, with standard MP-MCMC we achieve only a reduction in the constant associated to the rate $n^{-1}$, whereas using importance sampling there are situations where the convergence rate improves to close to $n^{-2}$.

\subsection{IS-MP-QMCMC}
\label{subsec:IS-MP-QMCMC}

As previously discussed in Section \ref{subsubsec:right_way}, importance sampling seems to be a solid approach in the context of MP-MCMC driven by CUD numbers, since it respects all proposals and therefore all numbers in the underlying sequence, thereby making full use of its spatial homogeneity. We now
introduce IS-MP-QMCMC, prove its consistency and show an improved convergence
rate compared to standard MP-MCMC or other MCMC methods in numerical
experiments. The two algorithms introduced here are extensions of Algorithms
\ref{algorithm:importance_sampling_mp_mcmc} and
\ref{algorithm:adaptive_importance_sampling_mp_mcmc} to using
any CUD sequence as their driving sequences.

\subsubsection{Algorithm description}
\label{subsubsec:algorithm_description_IS_mpqmcmc}

Analogously to pseudo-random IS-MP-MCMC, introduced in Section
\ref{sec:importance_sampling}, all proposals from one iteration 
are \textit{accepted}. Thus, only a single sample of $I$ from the finite 
state Markov chain is generated, which determines the distribution from which
the subsequent $N$ proposals are generated. The resulting method is displayed as Algorithm \ref{algorithm:importance_sampling_mp_qmcmc}, which can be viewed as an extension of Algorithm \ref{algorithm:importance_sampling_mp_mcmc} to the CUD case.

\begin{algorithm}[h]
\SetAlgoLined
\KwIn{Initialise starting point (proposal) $\vec{y}_1\in \Omega$, number of proposals $N$, auxiliary variable $I=1$, integrand $f$, initial mean estimate $\boldsymbol{\mu}_1 = \boldsymbol{\mu}_1(f)$ \hl{and number of MCMC iterations $L$} \;}
\hl{Generate a CUD sequence $u_1, ..., u_{L(Nd+1)}\in (0,1)$}\;
\For{\textnormal{each MCMC iteration $\ell=1,...,$\hl{$L$} }}{
\hl{Set $\vec{u} = (u_{(\ell-1)(Nd+M)+1}, ..., u_{(\ell-1) (Nd+1)+ Nd}) \in (0,1)^{Nd}$}, and 
sample ${\vec{y}}_{\setminus I}$ conditioned on $I$, i.e., draw $N$ new points from $\kappa({\vec{y}}_I, \cdot) = p({\vec{y}}_{\setminus I}|{\vec{y}}_I)$ \hl{by the inverse $\Psi_{\vec{y}_I}(\vec{u})$} \;
  Calculate the stationary distribution of $I$ conditioned on ${\vec{y}}_{1:N+1}$, i.e.\ $\forall$ $i=1,...,N+1$, $p(I=i|{{\vec{y}}}_{1:N+1}) = $ $\pi({{\vec{y}}}_i)\kappa({{\vec{y}}}_{{i}}, {{\vec{y}}}_{\setminus{i}}) / \sum_j \pi({{\vec{y}}}_j)\kappa({{\vec{y}}}_{{j}}, {{\vec{y}}}_{\setminus{j}})$, which can be done in parallel\;
{ Compute $\tilde{\boldsymbol{\mu}}_{\ell+1}=\sum_{i} p(I=i| {{\vec{y}}}_{1:N+1})f({{\vec{y}}}_i)$\;}
{Set $\boldsymbol{\mu}_{\ell+1} = \boldsymbol{\mu}_{\ell} + \frac{1}{\ell+1}\left(\tilde{\boldsymbol{\mu}}_{\ell+1} - \boldsymbol{\mu}_{\ell}\right)$}\;
\hl{Set $v' = u_{\ell(Nd+1)} \in (0,1]$}\;
 \hl{If $v' \in (\gamma_{j -1}, \gamma_{j}]$, where $\gamma_{j} = \sum_{i=1}^j p(I=i| \vec{y}_{1:N+1})$ 
 for $j= 1,...,N+1$ and $\gamma_0:=0$, set $I=j$}\;
 }
\caption{Importance sampling MP-QMCMC \newline
We highlight all altered code compared to (pseudo-random) IS-MP-MCMC, Algorithm \ref{algorithm:importance_sampling_mp_mcmc}}
 \label{algorithm:importance_sampling_mp_qmcmc}
\end{algorithm}

\subsubsection{Asymptotic unbiasedness of IS-MP-QMCMC}

\begin{corollary}[Asymptotic unbiasedness of IS-MP-QMCMC]
\label{corollary:asymptptic_unbiasedness_is_mp_qmcmc}
Under the conditions of Theorem \ref{theorem:consistency_multi_prop_mcmc},
the IS-MP-QMCMC sequence of estimators $(\boldsymbol{\mu}_L)_{L \ge 1}$ from Algorithm
\ref{algorithm:importance_sampling_mp_qmcmc} is asymptotically unbiased.
\begin{proof}
Due to the consistency of the underlying MP-QMCMC chain we may argue analogously
to Lemma \ref{lemma:asymptotic_unbiasedness_is}. 
\end{proof}
\end{corollary}

\subsubsection{Empirical results: Bayesian linear regression}
\label{mp_qmcmc_empirical_results}

\footnote{For the Python code of this simulation, we refer to \url{https://github.com/baba-mpe/MP-Quasi-MCMC}}

Let us consider the standard linear regression problem, in which the 
conditional distribution of an observation $\vec{y}\in \mathbb{R}^n$, 
given a design matrix $X \in \mathbb{R}^{n\times d}$, is specified by
\begin{align*}
\vec{y} = X \vec{\beta} + \vec{\varepsilon},
\end{align*}
where $\vec{\beta} \in \mathbb{R}^d$, and 
$\vec{\varepsilon} \sim \mathcal{N}(\vec{0},\sigma^2\mathbf{I}_n)$ 
denotes the $n$-dimensional noise term. Each row of the matrix $X$ is 
a predictor vector. The resulting likelihood function is given by
\begin{align*}
\pi(\vec{y}|\beta,X,\sigma^2) \propto (\sigma^2)^{-n/2}
\exp\left( -\frac{1}{2\sigma^2}\left( \vec{y}-X\vec{\beta} \right)^T 
\left( \vec{y} - X \vec{\beta} \right) \right).
\end{align*}
In a Bayesian context, we are interested in the distribution of the 
weight vector $\vec{\beta}$ conditioned on the design matrix $X$, the 
observation $\vec{y}$, the noise variance $\sigma^2$ and some prior 
information on $\vec{\beta}$. More precisely, given the above we would 
like to estimate the expectation of $\vec{\beta}$. Our a priori 
knowledge about $\vec{\beta}$ is expressed by the prior distribution
\begin{align*}
\pi(\vec{\beta}|\sigma^2) \propto |\det(\Sigma_0)|^{-1} 
\exp\left( -\frac{1}{2} \vec{\beta}^T \Sigma_0^{-1} \vec{\beta}  \right),
\end{align*}
where $\Sigma_0 = \sigma^2/g (X^TX)^{-1}$ with $g=1/n$. Thus, 
$\pi(\vec{\beta}|\sigma^2)$ denotes Zellner's g-prior according to 
\cite{zellner1986prior}. To estimate 
$\mathbb{E}_{\pi}[\vec{\beta}|\vec{y},X,\sigma^2]$ under the resulting 
posterior distribution,
\begin{align*}
\pi(\vec{\beta}|\vec{y},X,\sigma^2) \propto 
\pi(\vec{\beta}|\sigma^2)\pi(\vec{y}|\vec{\beta},X,\sigma^2),
\end{align*}
the importance sampling MP-QMCMC, described in Algorithm 
\ref{algorithm:importance_sampling_mp_qmcmc}, is applied. 
The underlying 
data consisting of $X$ and $\vec{y}$ is simulated according to 
$X \sim \mathcal{N}( \vec{0}, \Sigma_X)$ and 
$\vec{y} = X^T \vec{\beta^*} + \vec{\varepsilon}$, 
where $\vec{\beta^*} = (1,...,1)^T \in \mathbb{R}^d$, 
and $\Sigma_X \in \mathbb{R}^{n\times n}$ has non-negligible entries 
off its diagonal. Referring to Table \ref{table:results_bayesian_linear_regression},
experiments are performed for dimensionalities
between $d=1$ and $d=500$.

The bias of the underlying MP-MCMC methods can hereby be computed
exactly as the posterior is available analytically and is given by 
the Gaussian $\pi(\vec{\beta}|\vec{y},X,\sigma^2)= \mathcal{N}(\vec{\mu},\Sigma)$, 
where
\begin{align*}
\vec{\mu} = \left( X^TX + \Sigma_0 \right)^{-1}
\left( X^TX\vec{\hat{\beta}} + \Sigma_0\vec{\mu}_0\right), \quad \text{and} \quad
\Sigma =  \sigma\left(X^TX +  \Sigma_0\right).
\end{align*}
Here, $\vec{\hat{\beta}} = \left( X^TX \right)^{-1}X^T\vec{y}$
denotes the ordinary least squares solution for $\vec{\beta}$,
where we made use of the Moore-Penrose pseudo-inverse for $X^TX$.

To generate proposals we make use of the SmMALA kernel defined in
\eqref{eq:smMALA_kernel}. Since samples from a current iteration are therefore
dependent on samples from the previous iteration, the consistency
of the resulting method using a CUD seed is not covered by Corollary
\ref{corollary:asymptptic_unbiasedness_is_mp_qmcmc}.
However, as we can see in Figure 
\ref{fig:variance} and from the results for the MSE reductions given in Table 
\ref{table:results_bayesian_linear_regression1}, convergence of the MSE seems 
to hold true. Furthermore, using the importance sampling formulation of MP-QMCMC, 
together with the CUD driving sequence, results in an improved variance convergence 
rate of close to $n^{-2}$, compared to $n^{-1}$ when using an IID seed.

\begin{figure}[h]
    \centering
    \begin{subfigure}[b]{0.6\textwidth}
        \includegraphics[width=\textwidth]{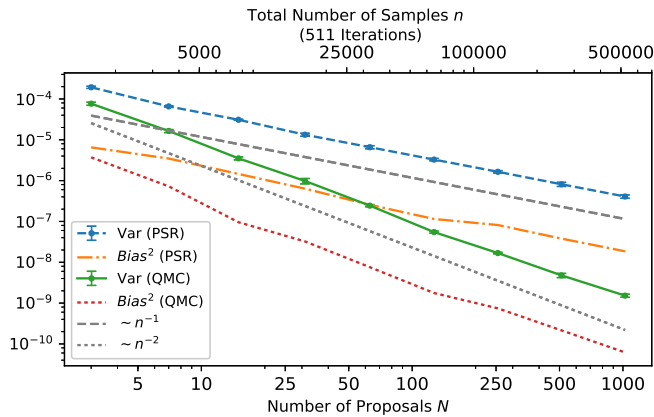}
    \end{subfigure} 
    \caption{\small{Empirical variance and squared bias of the sample mean in 
    $100$-dimensional Bayesian linear regression based on IS-MP-QMCMC (SmMALA) 
    using a pseudo-random (PSR) vs.\ CUD (QMC) seed, resp., for increasing proposal 
    numbers and sample sizes. The results are based on $25$ MCMC simulations, and
    the error bars correspond to twice a standard deviation}} 
    \label{fig:variance}
\end{figure}

\begin{table}[h]
\ra{1.1}
\centering
\caption{
Comparison of MSE convergence rates for IS-MP-MCMC (SmMALA)
using a pseudo-random (PSR) vs.\ a CUD (QMC) seed, resp., in Bayesian linear regression
for increasing dimensionality. Displayed
is also the associated reduction factor in MSE by the 
use of the deterministic
instead of the pseudo-random seed, resp.}
\centering
\resizebox{.55\textwidth}{!}{

\begin{tabular}{  @{} l  @{\hspace{3mm}} c @{\hspace{3mm}} c @{} *4c @{}}  \bottomrule
  \multirow{2}{*}{{Dimension}\hspace{1mm}} 
& \multicolumn{2}{c}{{MSE Rate}\hspace{-5mm}} & \multicolumn{4}{c}{{Reduction by using QMC}}\\
\cmidrule{2-3} \cmidrule{5-7}
& {\hspace{3mm}{PSR}\hspace{3mm}} & {\hspace{3mm}{QMC}\hspace{3mm}} &&
	{{$N=3$}} &{{$N=63$}} &{{$N=1023$}} \\ \midrule

	  1          	  & -1.06       & -1.90        &&     2.5  &  24.0 & 508.0 \\   

       2     	  & -1.09       & -1.97        &&     1.6   & 45.7 & 319.8\\ 

	    5        	  & -1.03        & -1.88        &&    1.9  &  35.2 & 234.1\\ 

	   10         	  & -1.03       & -1.89        &&   2.1   & 26.4 & 375.4\\ 

      25     	  & -1.04        & -1.86        &&    2.7   & 32.0 & 247.2\\ 

      50         	  & -1.04        & -1.89        &&    2.2   & 26.2 & 271.3 \\ 
      
      100 			& -1.04		& -1.88			&& 2.5		& 27.2 		&  269.2 \\ 
      
        250 			& -1.03		& -1.90			&& 1.8		& 31.5 		&  263.7 \\  

       500 			& -1.04		& -1.79			&& 2.7		& 15.7 		&  173.3 \\ 
       
 \bottomrule

\end{tabular}
\label{table:results_bayesian_linear_regression1}
}
\end{table}

\subsection{Adaptive IS-MP-QMCMC}
\label{subsec:adaptive_is_mp_qmcmc}

Finally, we introduce an importance sampling method that is based on MP-QMCMC 
proposals with adaptation of the proposal distribution. As a result we receive
Algorithm Algorithm \ref{algorithm:adaptive_importance_sampling_mp_qmcmc}, which 
states a direct extension of Algorithm \ref{algorithm:adaptive_importance_sampling_mp_mcmc} to the general CUD case.

\subsubsection{Algorithm description}
\label{subsubsec:algorithm_description_adaptive_IS_mpqmcmc}

In every MCMC iteration, the proposal distribution is updated. More precisely,
the mean and covariance of the proposal distribution, which is assumed to 
have a multivariate Gaussian distribution, is determined in each iteration by the average
of the weighted sum mean and covariance estimates, respectively, based on
previous iterations including the present one. The resulting mean estimate after the 
final iteration is the output of this algorithm.

 \begin{algorithm}[h]
\SetAlgoLined
\KwIn{Initialise starting point (proposal) $\vec{y}_1\in \Omega$, number of proposals $N$, auxiliary variable $I=1$, integrand $f$, initial mean estimate $\boldsymbol{\mu}_1 = \boldsymbol{\mu}_1(f)$,
initial covariance estimate $\Sigma_1$ \hl{and number of MCMC iterations $L$}\;} 
\hl{Generate a CUD sequence $u_1, ..., u_{LNd}\in (0,1)$}\;
\For{\textnormal{each MCMC iteration $\ell=1,...,$\hl{$L$} }}{
\hl{Set $\vec{u} = (u_{(\ell-1)Nd+1}, ..., u_{\ell Nd}) \in (0,1)^{Nd}$}, and 
sample ${\vec{y}}_{\setminus I}$ conditioned on $I$ {and $\Sigma_\ell$}, i.e., draw $N$ new points from $\kappa_{\Sigma_\ell}({\vec{y}}_I, \cdot) = p({\vec{y}}_{\setminus I}|{\vec{y}}_I, \Sigma_\ell)$ \hl{by the inverse $\Psi_{\vec{y}_I}(\vec{u}|\Sigma_\ell)$} \;
  Calculate the stationary distribution of $I$ conditioned on ${\vec{y}}_{1:N+1}$ and $\Sigma_\ell$, i.e.\ $\forall$ $i=1,...,N+1$, $p(I=i|{{\vec{y}}}_{1:N+1}, \Sigma_\ell) = $ $\pi({{\vec{y}}}_i)\kappa_{\Sigma_\ell}({{\vec{y}}}_{{i}}, {{\vec{y}}}_{\setminus{i}}) / \sum_j \pi({{\vec{y}}}_j)\kappa_{\Sigma_\ell}({{\vec{y}}}_{{j}}, {{\vec{y}}}_{\setminus{j}})$, which can be done in parallel\;
{ Compute $\tilde{\boldsymbol{\mu}}_{\ell+1}=\sum_{i} p(I=i| {{\vec{y}}}_{1:N+1}, \Sigma_\ell)f({{\vec{y}}}_i)$}\;
{Set $\boldsymbol{\mu}_{\ell+1} = \boldsymbol{\mu}_{\ell} + \frac{1}{\ell+1}\left(\tilde{\boldsymbol{\mu}}_{\ell+1} - \boldsymbol{\mu}_{\ell}\right)$}\;
\hl{Set $v' = u_{\ell(Nd+1)} \in (0,1]$}\;
\hl{If $v' \in (\gamma_{j -1}, \gamma_{j}]$, where $\gamma_{j} = \sum_{i=1}^j 
p(I=i| \vec{y}_{1:N+1}, \Sigma_\ell)$ for $j= 1,...,N+1$ and $\gamma_0:=0$, set $I=j$}\;
{Compute $\tilde{{\Sigma}}_{\ell+1} = \sum_{i}p(I=i| {\vec{y}}_{1:N+1}, {{\Sigma}}_\ell)[{\vec{y}}_i-\vec{\mu}_{\ell+1}][{\vec{y}}_i-\vec{\mu}_{\ell+1}]^T$}\;
{ Set ${\Sigma}_{\ell+1} = {\Sigma_\ell} + \frac{1}{\ell+1}(\tilde{{\Sigma}}_{\ell+1} - {\Sigma}_{\ell})$}\;
 }
\caption{Adaptive importance sampling MP-QMCMC \newline
All code altered compared to (pseudo-random) adaptive IS-MP-MCMC, Algorithm \ref{algorithm:adaptive_importance_sampling_mp_mcmc},
is highlighted}
 \label{algorithm:adaptive_importance_sampling_mp_qmcmc}
\end{algorithm}

\subsubsection{Asymptotic unbiasedness of adaptive IS-MP-QMCMC}

\begin{corollary}[Asymptotic unbiasedness of adaptive IS-MP-QMCMC]
Under the conditions of Theorem \ref{theorem:consistency_multi_prop_mcmc}, and
under any of the conditions stated in
Theorem \ref{thm:ergodicity_adap_mpmcmc_independent},
Theorem \ref{thm:ergodicity_adapt_mpmcmc_bounded} or
Theorem \ref{thm:ergodicity_adaptive_mpmcmc_positive}, 
the sequence of estimators 
$(\boldsymbol{\mu}_L)_{L \ge 1}$ from Algorithm
\ref{algorithm:adaptive_importance_sampling_mp_qmcmc} is asymptotically unbiased.
\begin{proof}
Consistency of the adaptive MP-QMCMC follows by the respective theorem used and
Theorem \ref{theorem:consistency_multi_prop_mcmc}.
Thus, we may argue analogously to the proof of Corollary
\ref{cor:asympt_unbias_adapt_ismpmcmc}.
\end{proof}
\end{corollary}

\subsection{Empirical Results: Simple Gaussian example}
\label{adaptive_is_mp_qmcmc_simple_Gaussian_example}

As a simple reference problem, we consider the estimation of the
posterior mean in a generic $1$-dimensional numerical example analogously
to \cite{owen2005quasi}, in which the posterior is just a standard Gaussian. 
For this problem, \cite{owen2005quasi} compared the Metropolis-Hastings
algorithm using a pseudo-random seed to using a QMC seed, respectively,
in terms of the resulting MSEs of the estimated posterior mean. 
The independent sampler is given by $\mathcal{N}(0,2.4^2)$ and the 
random walk sampler by $\mathcal{N}(x, 2.4^2)$, where $x\in \mathbb{R}$
denotes the last accepted sample.
Here, we additionally compare those algorithms with the IS-MP-QMCMC and its
adaptive version introduced in Section \ref{subsec:IS-MP-QMCMC} and Section 
\ref{subsec:adaptive_is_mp_qmcmc}, respectively.
In the case of the independent sampler, we apply adaptivity within
AIS-MP-QMCMC not only in the proposal covariance but also in the proposal mean,
according to the estimate in line 6 from Algorithm
\ref{algorithm:adaptive_importance_sampling_mp_qmcmc}.
For the random walk case, we do not make use of adaptivity in the mean,
as the proposal sampler mean is just a sampled point from the finite state
chain from one iteration.
The low-discrepancy sequence used in \cite{owen2005quasi} is based on a 
LCG, while the QMC construction used here is based on the LFSR introduced
in \cite{chen2012new}, see Section 
\ref{subsubsec:completely_uniformly_distributed_points}
for details. Nevertheless, we receive similar results in the reduction of 
MSE between standard Metropolis-Hastings and its QMC-driven counterpart
as \cite{owen2005quasi} did:
for the independent sampler, the MSE is reduced by a factor $\approx 7.0$,
and for the random walk sampler by a factor $\approx 2.3$, respectively. 
In the 
independent sampler case, the additional performance gain using the 
importance sampling approaches is significant: compared to standard
Metropolis-Hastings, the maximum MSE reduction is $\approx 112.2$, i.e.\
more than an order of magnitude (factor $\approx 16.1$) reduction compared 
to QMC-driven Metropolis-Hastings. Note that the performance of the 
multiple proposal methods increases with number of proposals used,
thereby making use of less number of iterations to result in the same total
number
of samples. In the case of random walk proposals, 
the additional gain is less substantial. Compared to standard 
Metropolis-Hastings we receive a maximum MSE reduction of $\approx 6.3$, i.e.\ 
a decrease of $\approx 2.7$ compared to QMC-driven Metropolis-Hastings.

\begin{table}[h]
\ra{1.3}
\centering
\caption{
Comparison of standard Metropolis-Hastings with different QMC driven
MCMC algorithms for independent and random walk proposals on a 1-dimensional
Gaussian numerical example with $65535$ samples. The numbers in the brackets correspond to three-times a standard deviation}
\centering
\resizebox{1.\textwidth}{!}{

\begin{tabular}{  @{} l  @{\hspace{3mm}} c @{\hspace{3mm}} c @{} *3c @{}}  \bottomrule
  \multirow{2}{*}{{Method}\hspace{1mm}} 
& \multicolumn{2}{c}{{Independence}\hspace{-5mm}} & \multicolumn{3}{c}{{Random Walk}}\\
\cmidrule{2-3} \cmidrule{5-6}
& {\hspace{8mm}{Mean}\hspace{8mm}} & {\hspace{8mm}{MSE}\hspace{8mm}} &&
	{{Mean}} &{{MSE}}  \\ \midrule

	  PSR Metropolis-Hastings          	  & $-1.64\times 10^{-4}$      
	  & $3.60 \times 10^{-5} (\pm 6.96 \times 10^{-6})$   
	       
	  &&     $-8.76\times 10^{-5}$  &  $6.76 \times 10^{-5} (\pm 6.56 \times 10^{-6})$   \\   

       QMC Metropolis-Hastings     	  & $1.96\times 10^{-4}$       
       & $5.17 \times 10^{-6} (\pm 1.61 \times 10^{-6})$        
       && $1.17 \times 10^{-4}$   & $2.88 \times 10^{-5} (\pm 6.68 \times 10^{-6})$ \\ 

	    IS-MP-QMCMC (4)        	  & $-1.97\times 10^{-4}$        
	    & $3.56 \times 10^{-6} (\pm 3.56 \times 10^{-7})$        
       && $-1.20 \times 10^{-3}$   & $2.74 \times 10^{-5} (\pm 7.70 \times 10^{-6})$ \\ 
      
	   Adapt.\ IS-MP-QMCMC (4)         	  & $-1.64\times 10^{-4}$       
	   & $4.31 \times 10^{-6} (\pm 2.94 \times 10^{-7})$        
		&&   $-1.44 \times 10^{-3}$   & $2.81\times 10^{-5} (\pm 8.11 \times 10^{-6})$  \\ 

      IS-MP-QMCMC (32)     	  & $-2.10\times 10^{-5}$        
      & $7.72 \times 10^{-7} (\pm 1.53 \times 10^{-7})$        
      && $-2.63 \times 10^{-4}$   & $1.07 \times 10^{-5} (\pm 1.52 \times 10^{-6})$ \\ 

      Adapt.\ IS-MP-QMCMC (32)         	  & $-1.15\times 10^{-4}$        
      & $7.83 \times 10^{-7} (\pm 1.60 \times 10^{-7})$     
        && $-5.62 \times 10^{-4}$		& $1.26 \times 10^{-5} (\pm 1.44 \times 10^{-6})$ 		 \\  

      IS-MP-QMCMC (256) 			& $9.44\times 10^{-5}$		
      & $5.32 \times 10^{-7} (\pm 1.29 \times 10^{-7})$			
	    && $-4.39 \times 10^{-4}$  &  $1.21 \times 10^{-5} (\pm 2.78 \times 10^{-6})$ \\ 

        Adapt.\ IS-MP-QMCMC (256) 			& $-1.29\times 10^{-5}$		
        & $3.21 \times 10^{-7} (\pm 6.17 \times 10^{-8})$			
	   &&  $-3.29 \times 10^{-4}$  & $1.07 \times 10^{-5} (\pm 1.35 \times 10^{-6})$ \\ 

       
 \bottomrule

\end{tabular}
\label{table:results_bayesian_linear_regression}
}
\end{table}

\subsection{Empirical Results: Bayesian logistic regression}
\label{adaptive_is_mp_qmcmc_empirical_results}

We now apply the proposed adaptive important sampling method as described
in Algorithm \ref{algorithm:adaptive_importance_sampling_mp_qmcmc}
to the Bayesian logistic regression from Section 
\ref{subsubsec:emp_results_adaptive_IS_mp_mcmc}. As proposal sampler,
we apply a Gaussian kernel that samples independently of previous samples
and adaptively updates its mean and covariance according to the weighted
estimates from Algorithm \ref{algorithm:adaptive_importance_sampling_mp_qmcmc},
respectively. Note that no gradient information about the posterior is needed
in this case. We compare the performance of the respective QMC algorithm to its pseudo-random version in terms of the empirical variance of their posterior mean estimates. As a reference, we perform the same experiments also for standard Metropolis-Hastings and SmMALA Metropolis-Hastings. Thereby, we employ the same total number of samples $n$ as the respective multiple proposal algorithms produce, i.e.\ $n=LN$, where $L$ denotes the number of iterations and $N$ the number of proposals used. This guarantees a fair comparison between the single and multiple proposal algorithms. The results for the datasets of varying dimension and increasing proposal numbers (total number of samples) are displayed in Figure 
\ref{fig:emprVar_blinreg_plus_mh_smmala}. Again, we observe
an improved convergence rate in the empirical variance of our estimator, which in many cases is close to $n^{-2}$. Further, Table \ref{table:results_bayesian_logistic_regression_emprVar} displays the associated reductions in empirical variance between the adaptive IS-MP-QMCMC and the other algorithms. Due to the improved convergence, the reductions increase with increasing proposal numbers and thus total numbers of proposals, leading to significant reductions in all models and compared to all algorithms for a large number of proposals. In some situations, a reduction of more than $4$ orders of magnitude compared to standard Metropolis-Hastings can be observed, and more than $2$ orders of magnitude compared to the SmMALA Metropolis-Hastings as well as compared to pseudo-random IS-MP-MCMC. For any model and throughout all number of proposals employed, the reduction compared to the two reference algorithms is significant.

\begin{figure}[H]
    \centering
    \begin{subfigure}[b]{0.45\textwidth}
        \includegraphics[width=\textwidth]{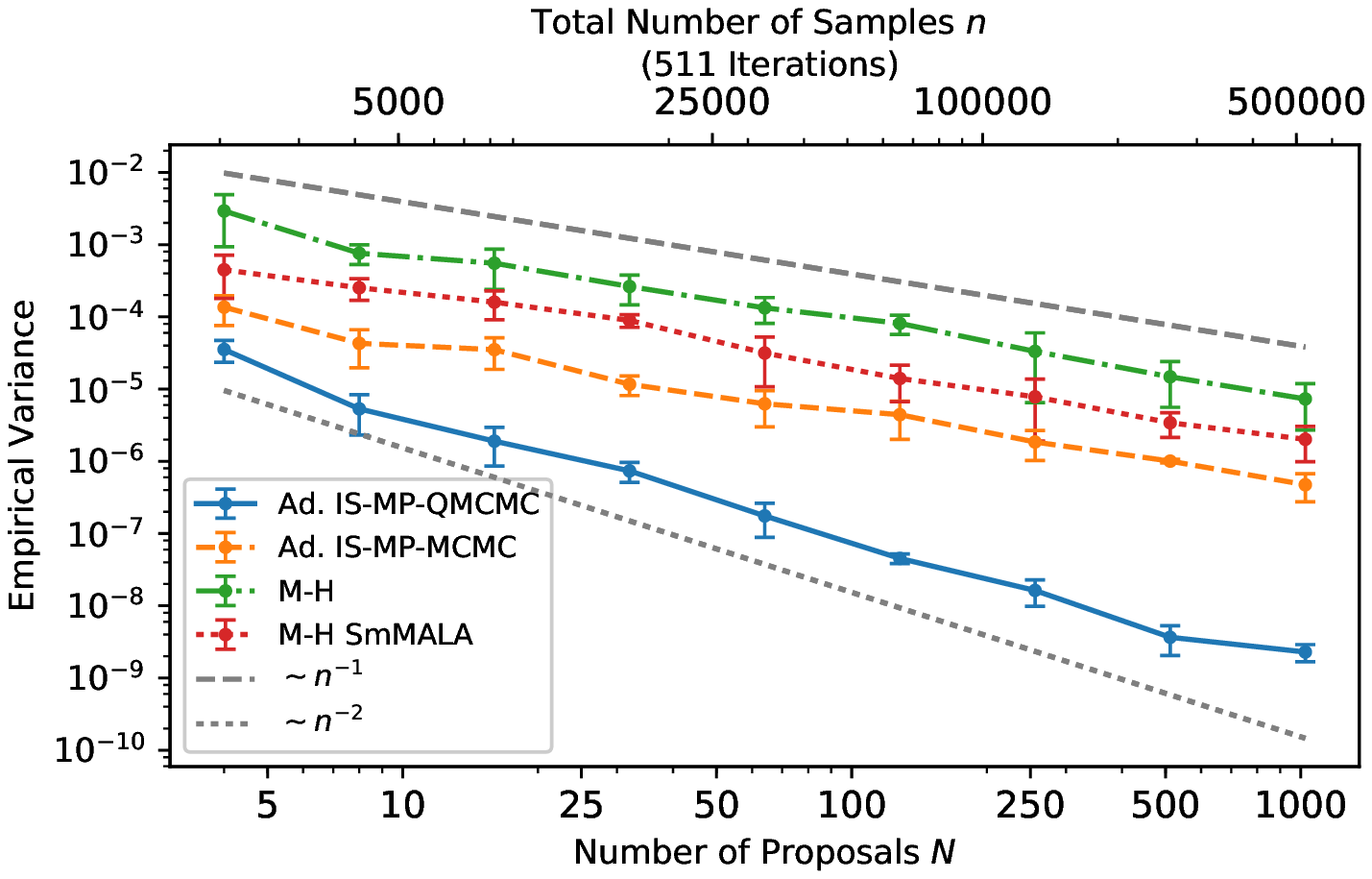}
        \subcaption{Ripley, $d=3$}
    \end{subfigure} 
        \begin{subfigure}[b]{0.45\textwidth}
        \includegraphics[width=\textwidth]{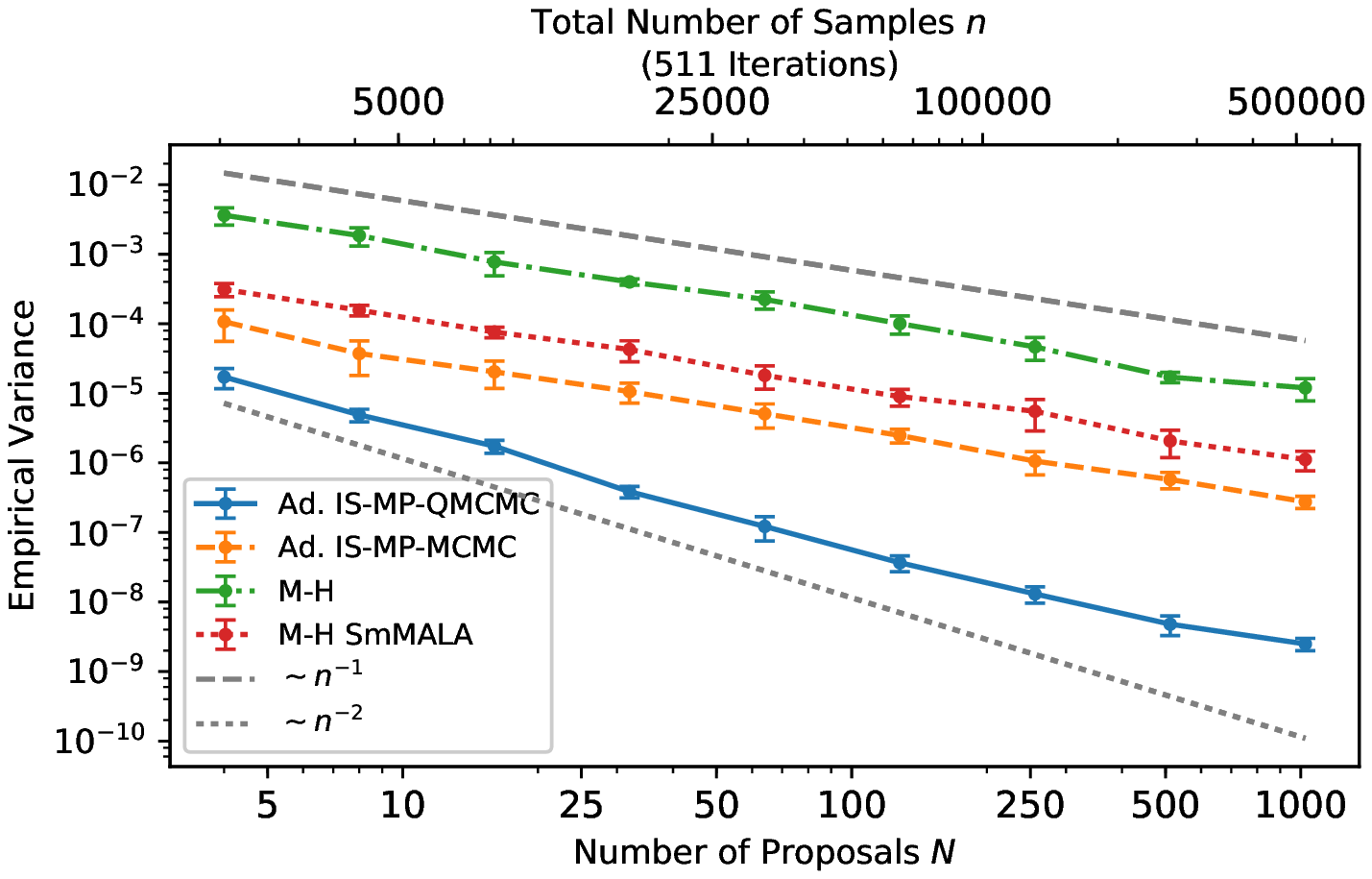}
        \subcaption{Pima, $d=8$}
    \end{subfigure} 
    \begin{subfigure}[b]{0.45\textwidth}
        \includegraphics[width=\textwidth]{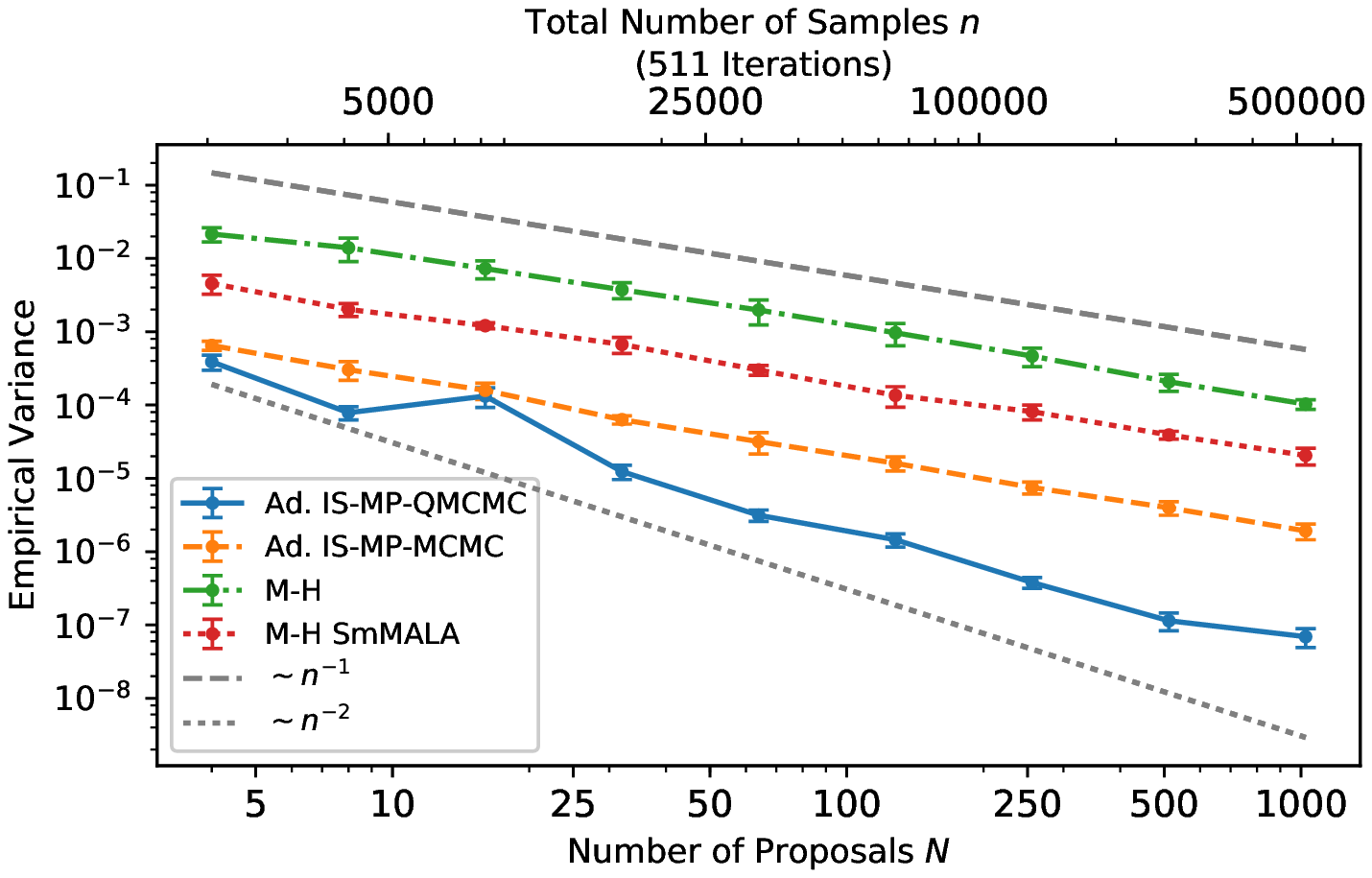}
        \subcaption{Heart, $d=14$}
    \end{subfigure} 
        \begin{subfigure}[b]{0.45\textwidth}
        \includegraphics[width=\textwidth]{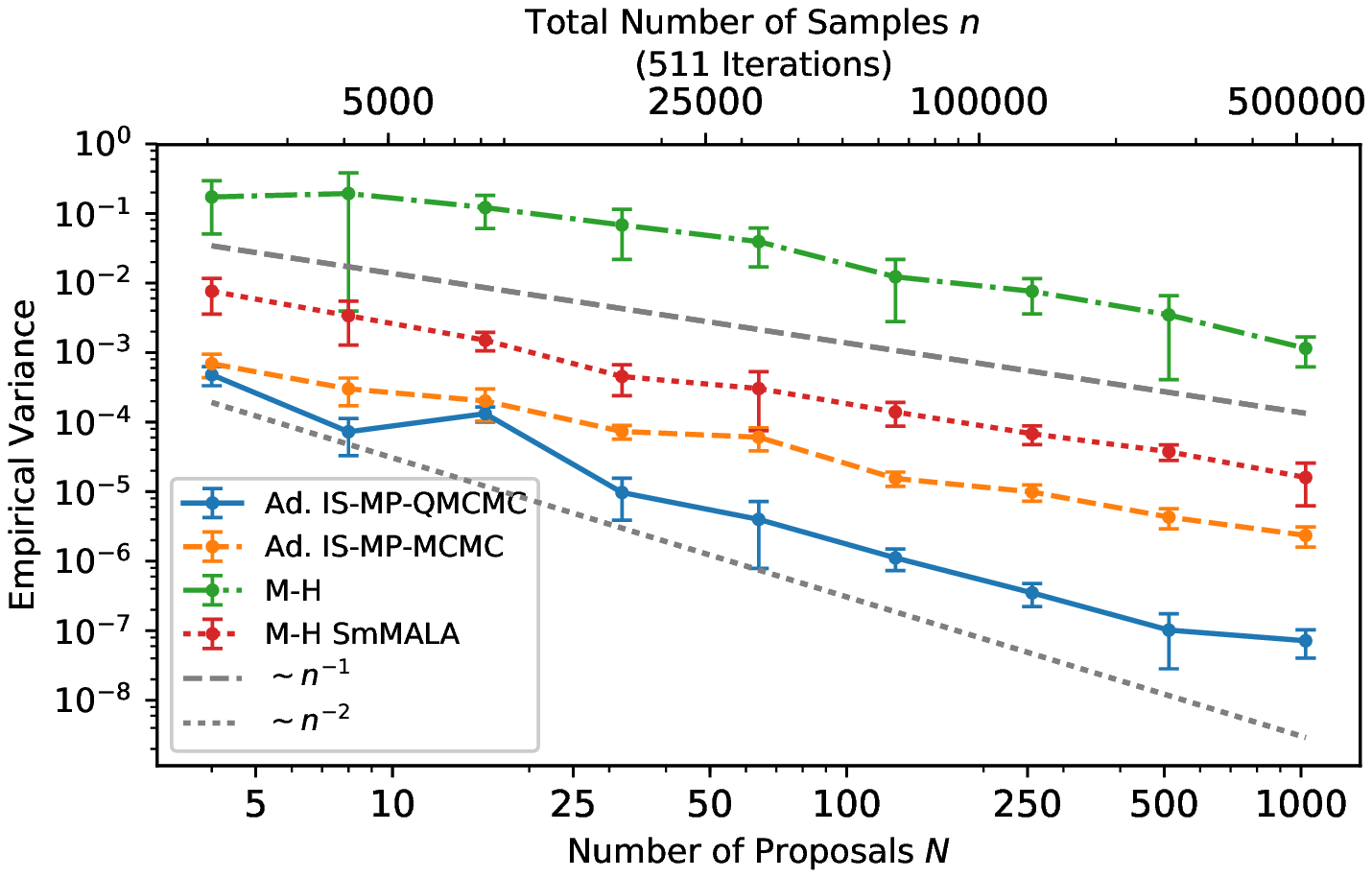}
        \subcaption{Australian, $d=15$}
    \end{subfigure}     
        \begin{subfigure}[b]{0.45\textwidth}
        \includegraphics[width=\textwidth]{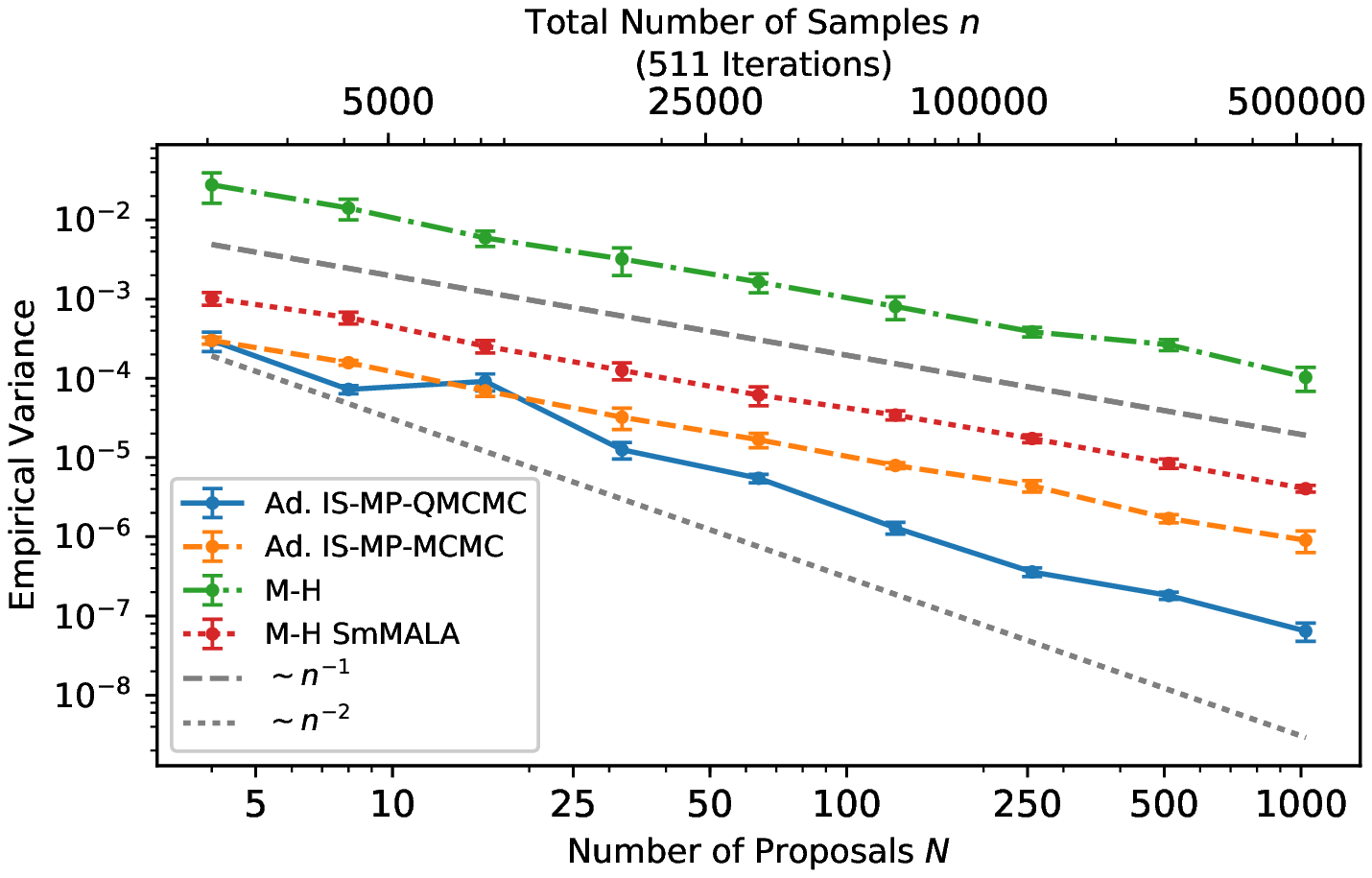}
        \subcaption{German, $d=25$}
    \end{subfigure}       
    \caption{\small{Empirical variance of adaptive IS-MP-QMCMC compared to adaptive IS-MP-MCMC,
SmMALA Metropolis-Hastings (M-H SmMALA) and standard Metropolis-Hastings (M-H), 
resp., for the Bayesian logistic regression problem from
\ref{subsubsec:emp_results_adaptive_IS_mp_mcmc}. Here, M-H SmMALA was tuned
to an approximately optimal acceptance rate of $50$-$60 \%$, and M-H to $20$-$25 \% $.
For the adaptive methods, a Burn-In of between 0-8192 samples, increasing with 
dimensionality $d$, was discarded. The results are based on $25$ MCMC simulations and
the errors bands correspond to three times a standard deviation}} 

    \label{fig:emprVar_blinreg_plus_mh_smmala}
\end{figure}

\begin{table}[h]
\ra{1.2}
\centering
\caption{
Ratio of empirical variances of adaptive IS-MP-QMCMC compared to adaptive IS-MP-MCMC,
SmMALA Metropolis-Hastings (M-H SmMALA) and standard Metropolis-Hastings (M-H), 
resp., for the Bayesian logistic regression problem from
\ref{subsubsec:emp_results_adaptive_IS_mp_mcmc}. Here, M-H SmMALA was tuned
to an approximately optimal acceptance rate of $50$-$60 \%$, and M-H to $20$-$25 \% $.
The results are based on $25$ MCMC simulations}
\centering
\resizebox{1.\textwidth}{!}{

\begin{tabular}{  @{} l  @{} c @{} c @{} *9c @{}}  \bottomrule
  \multirow{2}{*}{} 
& \multicolumn{1}{c}{{Ad.\ IS-MP-QMCMC}} & \multicolumn{10}{c}{{Ratio in empirical variance for $N=$}} \\
\cmidrule{2-2} \cmidrule{4-12}
 & {{VS}} &  & $4$& $8$& $16$ & $32$ & $64$ & $128$ & $256$ & $512$ & $1024$ \\
 \midrule

& {{Ad.\ IS-MP-MCMC}} &  &
	$3.9$ & $8.1$ & $18.5$ & $15.8$ & $35.7$ & $97.7$ & $113.5$ &
	$274.7$ & $207.1$ \\ 
Ripley & {{M-H SmMALA}} &  & $12.7$& $47.5$ & $84.0$ & $121.0$ &
	$179.9$ & $309.7$ & $479.4$ & $933.2$ & $880.7$ \\
& {{M-H}} &  & $82.8$ & $142.8$ & $290.8$ & $356.8$ & $759.3$ &
	$1794.6$ &$2040.6$ & $4040.5$ & $3191.7$\\	\midrule

& {{Ad.\ IS-MP-MCMC}}  &  & $6.2$& $7.6$ & $11.7$ & $27.5$ &
	$41.7$ & $67.6$ & $81.1$ & $120.2$ & $110.3$ \\
Pima & {{M-H SmMALA}}& \hspace{5mm} &
	$18.1$ & $32.0$ & $43.7$ & $110.8$ & $148.5$ & $244.7$ & $422.7$ &
	$430.8$ & $448.9$ \\ 
& {{M-H}} &  & $211.3$ & $378.5$ & $445.3$ & $1038.2$ & $1837.7$ &
	$2739.3$ &$3561.0$ & $3568.3$ & $4808.5$\\	\midrule  

 & {{Ad.\ IS-MP-MCMC}}  &  & $1.6$& $3.9$ & $1.2$ & $5.1$ &
	$10.1$ & $11.0$ & $19.8$ & $34.7	$ & $ 27.7$ \\
Heart & {{M-H SmMALA}}& \hspace{5mm} &
	$11.7$ & $25.6$ & $9.1$ & $54.2$ & $95.8$ & $93.2$ & $213.9$ &
	$342.0$ & $296.1$ \\ 
& {{M-H}} &  & $55.2$ & $177.7$ & $54.8$ & $301.7$ & $629.2$ &
	$665.3$ &$1225.1$ & $1815.6$ & $1486.8$\\	\midrule  

 & {{Ad.\ IS-MP-MCMC}}  &  & $1.4$& $4.1$ & $1.5$ & $7.5$ &
	$15.2$ & $13.9$ & $28.2$ & $42.1$ & $32.8$ \\
Australian & {{M-H SmMALA}}& \hspace{5mm} &
	$15.9$ & $46.7$ & $11.4$ & $46.6$ & $76.0$ & $125.1$ & $193.7$ &
	$367.5$ & $222.5$ \\ 
& {{M-H}} &  & $359.6$ & $2664.2$ & $917.4$ & $7017.4$ & $9842.3$ &
	$10985.6$ &$21625.8$ & $34043.2$ & $16010.5$\\	\midrule 

 & {{Ad.\ IS-MP-MCMC}}  &  & $1.0$& $2.2$ & $0.8$ & $2.6$ &
	$3.1$ & $6.1$ & $12.2$ & $9.4$ & $14.0$ \\
German & {{M-H SmMALA}}& \hspace{5mm} &
	$3.4$ & $8.1$ & $2.8$ & $10.0$ & $11.2$ & $26.5$ & $48.3$ &
	$46.5$ & $62.6$ \\ 
& {{M-H}} &  & $92.1$ & $195.0$ & $64.8$ & $255.7$ & $301.0$ &
	$622.9$ &$1076.5$ & $1465.6$ & $1595.2$\\	
		      
 \bottomrule

\end{tabular}
\label{table:results_bayesian_logistic_regression_emprVar}
}
\end{table}

\subsection{Empirical Results: Bayesian inference in non-linear differential equations}
\label{subsec:bayesian_inference_non_linear_odes}

An important category of inverse problems concern the study of uncertainty quantification in dynamical systems given by a set of ordinary differential equations (ODEs). Typically, such a system can be formalised by $M$ coupled ODEs and a model parameter $\vec{\theta} \in \mathbb{R}^d$, which describes the dynamics of the system's state $\vec{x}\in \mathbb{R}^M$ in terms of its time derivative by $\mathrm{d}\vec{x}/ \mathrm{d}t = \vec{f}(\vec
{x}, \vec{\theta}, t)$. Given state observations $\vec{y}(t)$ at $T$ distinct points in time, our aim is to infer about the underlying parameter $\vec{\theta}$ and, more specifically, about integral quantities $\int g(\vec{\theta}) \mathrm{d}\vec{\theta}$ for an integrable scalar-valued function $g$. An observation $\vec{y}(t)$ at time $t$ is usually subject to a measurement error, which can be modeled as $\vec{y}(t) = \vec{x}(t) + \bm{\varepsilon}(t)$, where $\bm{\varepsilon}(t)\in \mathbb{R}^M$ states a suitable multivariate noise variate at time $t$. Often, $\bm{\varepsilon}(t)$ is Gaussian with zero mean and standard deviation $\sigma_m$ in the $n$th component for $m=1,...,M$. Given $T$ distinct observations, the observed system can be summarised in matrix notation by $Y = X+E$, where $Y,X,E$ denote $T \times M$ matrices whose rows correspond to the observation process at the distinct $T$ points in time. To generate a sample $X$ one needs to solve the underlying set of ODEs given the model parameter $\vec{\theta}$ and an initial condition $\vec{x}_0$, i.e.\ $X = X(\vec{\theta}, \vec{x}_0)$. If $\pi(\vec{\theta})$ denotes the prior for $\vec{\theta}$, the posterior density of $\vec{\theta}|Y$ can then be expressed as
\begin{align}
\pi(\vec{\theta}|Y) \propto \pi(\vec{\theta}) 
\prod_{m} \mathcal{M}\left(Y_{1:T,m} | X(\vec{\theta}, \vec{x}_0)_{1:T,m}, \Sigma_n \right).
\end{align}
In the following, experiments based on the adaptive importance sampling QMCMC scheme for multiple proposals introduced in Section \ref{subsec:adaptive_is_mp_qmcmc} are performed for two different ODE models, namely the Lotka-Volterra and the FitzHugh-Nagumo model.

\subsubsection{Lotka-Volterra}
\label{subsubsec:lotvol}

The Lotka-Volterra equations are a set of two non-linear ODEs that describe the interaction between two species in a predator-prey relationship. Formally, this can be expressed as
\begin{align}
\frac{\mathrm{d}u}{\mathrm{d}t} &= \alpha u - \beta u v\\
\frac{\mathrm{d}v}{\mathrm{d}t} &= \gamma u v - \delta v,
\end{align}
where $u$ and $v$ represent the population of prey and predators, respectively, and $\alpha, \beta, \gamma, \delta >0$ determine the interaction between the two species given $\mathrm{d}u/\mathrm{d}t$ and $\mathrm{d}v/\mathrm{d}t$. 

We used 400 data points generated by the respective models between $t\in [0,8]$. The respective model parameters were chosen as $\alpha = 1.8, \beta=0.5, \gamma=2.5$ and $\delta = 1$, and the initial conditions as $u(0)=10$ and $v(0)=5$. To the model state outcomes was then added a Gaussian noise with standard deviation equal to $0.25$. In Figure \ref{fig:ode_data}, the underlying true state trajectories together with their noisy measurements are displayed.

\subsubsection{FitzHugh–Nagumo}
\label{subsubsec:fitznag}

The FitzHugh-Nagumo model is a set of two non-linear ODEs that describe the dynamics of an excitable system in terms of two states, namely a membrane voltage $u$ and a recovery variable $v$, defined by
\begin{align}
\frac{\mathrm{d}u}{\mathrm{d}t} &= 
\gamma\left( u - \frac{u^3}{3} + v \right)\\
\frac{\mathrm{d}v}{\mathrm{d}t} &= 
-\left( \frac{u-\alpha + \beta v}{\gamma} \right).
\end{align}
Here, $\alpha, \beta$ and $\gamma$ serve as scaling parameters and to determine the unstable equilibrium state value.

The underlying data consists of 200 data points produced by the FitzHugh-Nagumo model between $t\in [0,2]$ with model parameters $\alpha = 0.5, \beta=0.5, \gamma=1.5$ and initial conditions $u(0)=-1$ and $v(0)=1$. A Gaussian noise with standard deviation $1$ was then added to the model outcomes. Figure \ref{fig:ode_data} (b) shows the model outcomes and the associated noisy observations.

\begin{figure}[h]
    \centering
    \begin{subfigure}[b]{0.45\textwidth}
        \includegraphics[width=\textwidth]{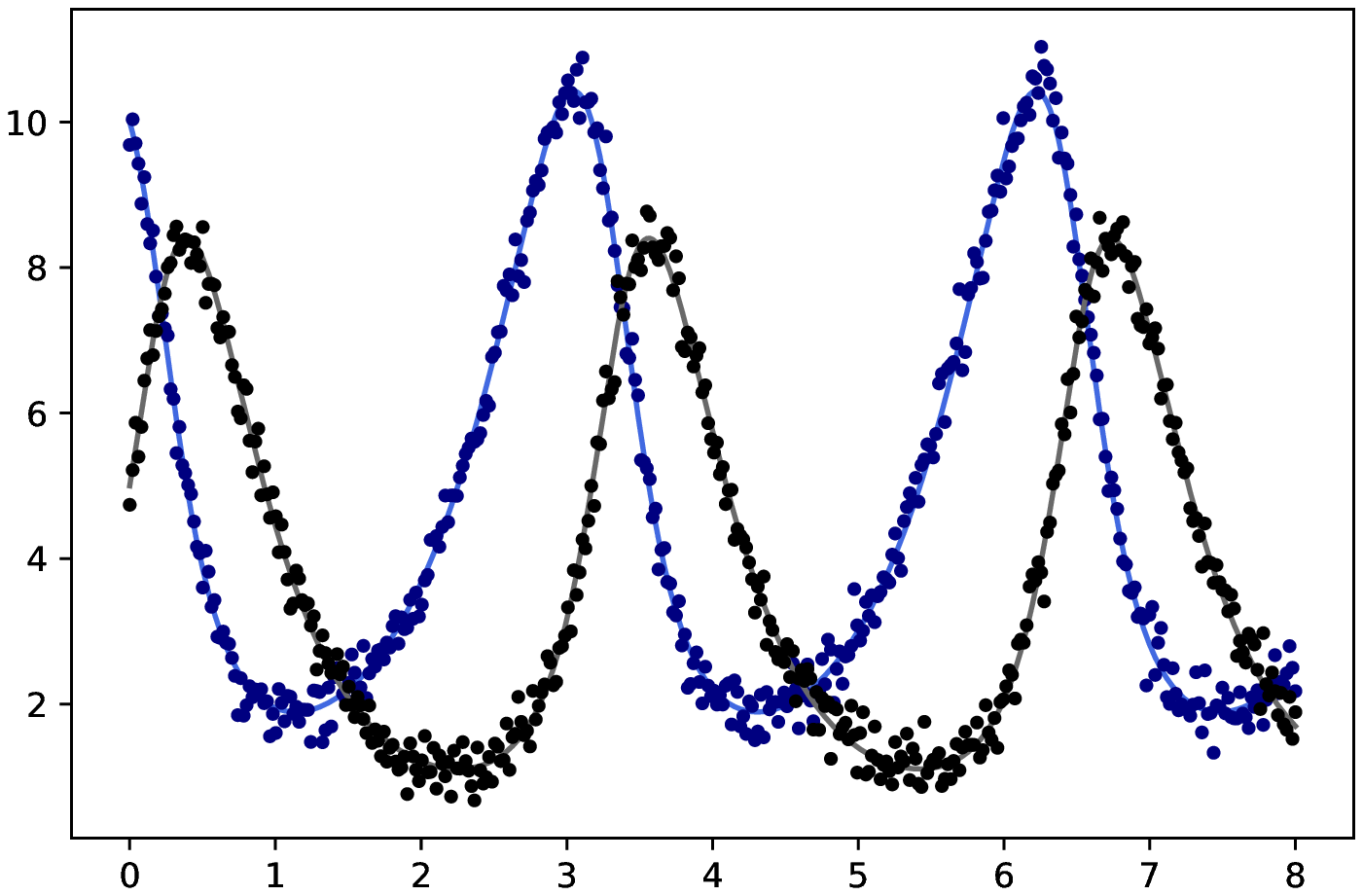}
        \subcaption{Lotka-Volterra}
    \end{subfigure}     
    \begin{subfigure}[b]{0.45\textwidth}
        \includegraphics[width=\textwidth]{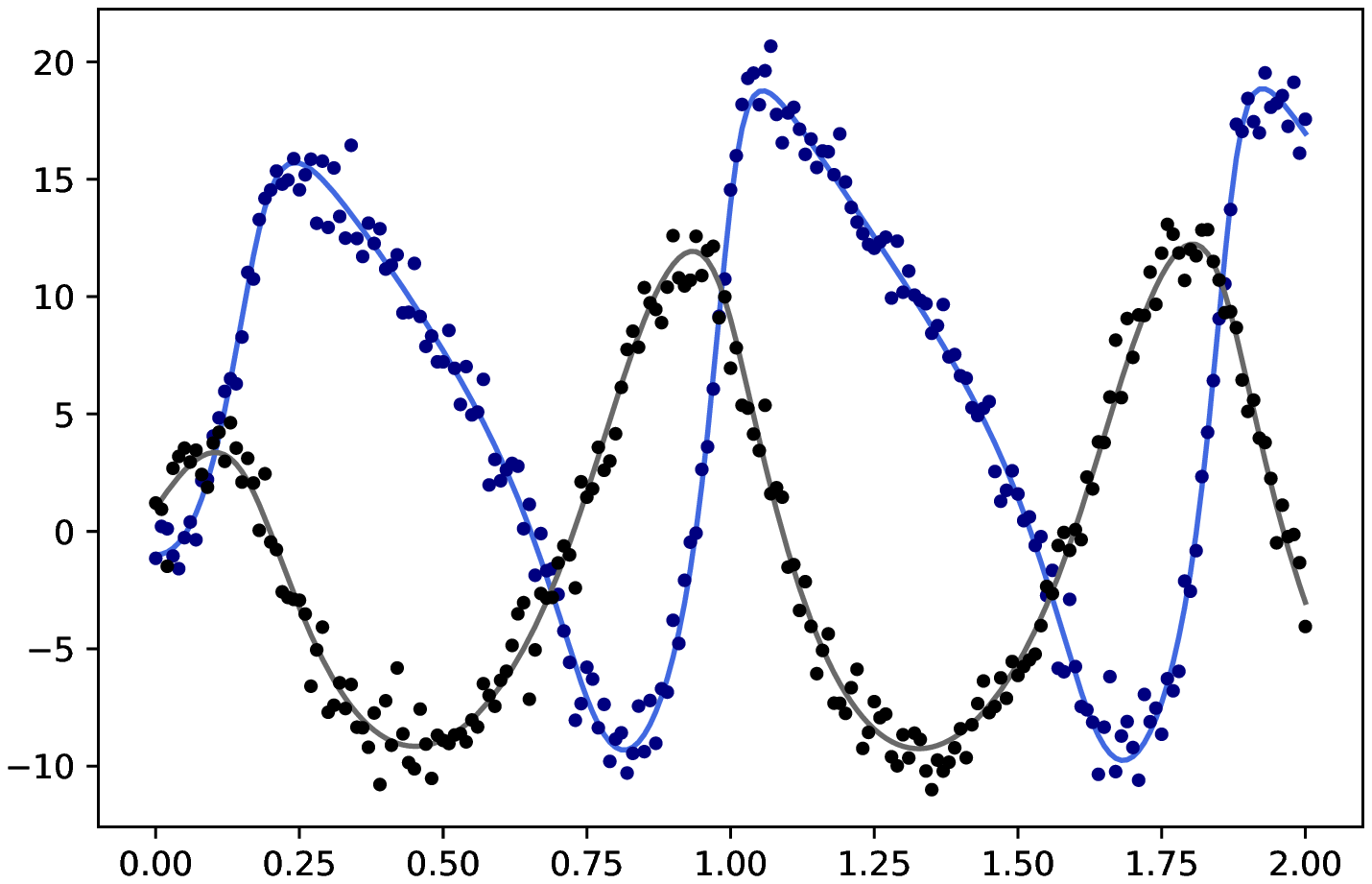}
        \subcaption{FitzHugh-Nagumo}
    \end{subfigure}       
    \caption{\small{Output for both states $u$ and $v$ in (a) the Lotka-Volterra model with parameters $\alpha = 1.8, \beta=0.5, \gamma=2.5$ and $\delta = 1$, and in (b) the FitzHugh-Nagumo model with parameters $\alpha=0.5, \beta=0.5$ and $\gamma=1.5$, respectively; the dots correspond to the respective noisy data}} 
    \label{fig:ode_data}
\end{figure}

\subsubsection{Numerical results}
 
The performance of using QMC numbers compared to using pseudo-random numbers as driving sequence is investigated numerically for the adaptive IS-MP-MCMC algorithm. As a reference, the respective simulations are also performed for a random-walk Metropolis-Hastings within Gibbs algorithm, which corresponds to performing a Metropolis-Hastings step within each directional component update in a Gibbs sampler. The standard Metropolis-Hastings algorithm suffers severely from low mixing for the considered problems, and therefore did not qualify as a performance reference. To satisfy fairness of comparison, the Metropolis-Hastings within Gibbs algorithm produces the same number $n$ of total samples as the multiple proposal algorithms, i.e.\ $n=LN$ with $L$ denoting the number of iterations and $N$ the number of proposals.  

Non-linear ODEs generally produce corresponding non-linear features in the posterior distribution, which can result in the emergence of multiple local maxima. It is therefore germane to ensure for an underlying MCMC method not to dwell in a wrong mode. However, to allow for the comparison of sampling efficiency measured by the empirical variance of estimates, we employ the respective MCMC methods initialised on the true mode. As initial proposal mean and covariance in the adaptive IS-MP-MCMC algorithms a rough posterior mean and covariance estimate was used. The former further served as initial value for the Metropolis-Hastings within Gibbs algorithm, whose steps sizes for individual components was chosen to meet an acceptance rate between $20$-$40 \%$.

The outcomes of the numerical experiments associated to the inference problems for the ODE models introduced above are shown in Figure \ref{fig:emprVar_ode_plus_mig} and Table \ref{table:results_ode_inference_emprVar}. As a prior, a Gamma distribution with shape parameter equal to $1$ and scale parameter equal to $3$, truncated at zero, was employed in both model problems. A clear improvement in the rate of convergence, being close to $n^{-2}$ in both the Lotka-Volterra and the FitzHugh-Nagumo case can be observed. In addition, we observe significant reductions in empirical variance for increasing numbers of proposals for adaptive IS-MP-QMCMC compared to its pseudo-random version and the reference Metropolis-Hastings within Gibbs algorithm. Compared to the latter, a maximal reduction of over $6$ orders of magnitude could be achieved.

\begin{figure}[h]
    \centering
    \begin{subfigure}[b]{0.45\textwidth}
        \includegraphics[width=\textwidth]{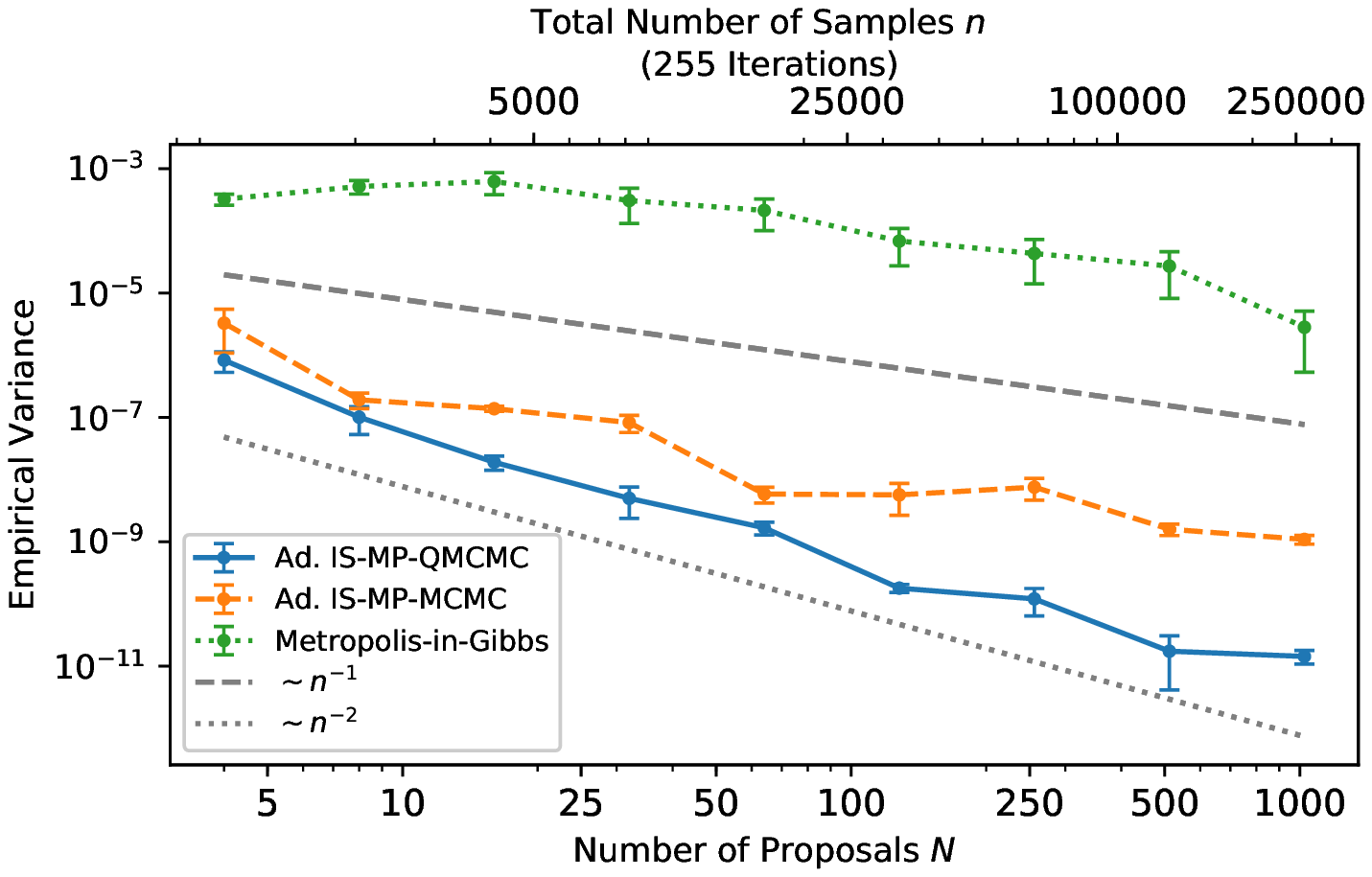}
        \subcaption{Lotka-Volterra}
    \end{subfigure} 
    \begin{subfigure}[b]{0.45\textwidth}
        \includegraphics[width=\textwidth]{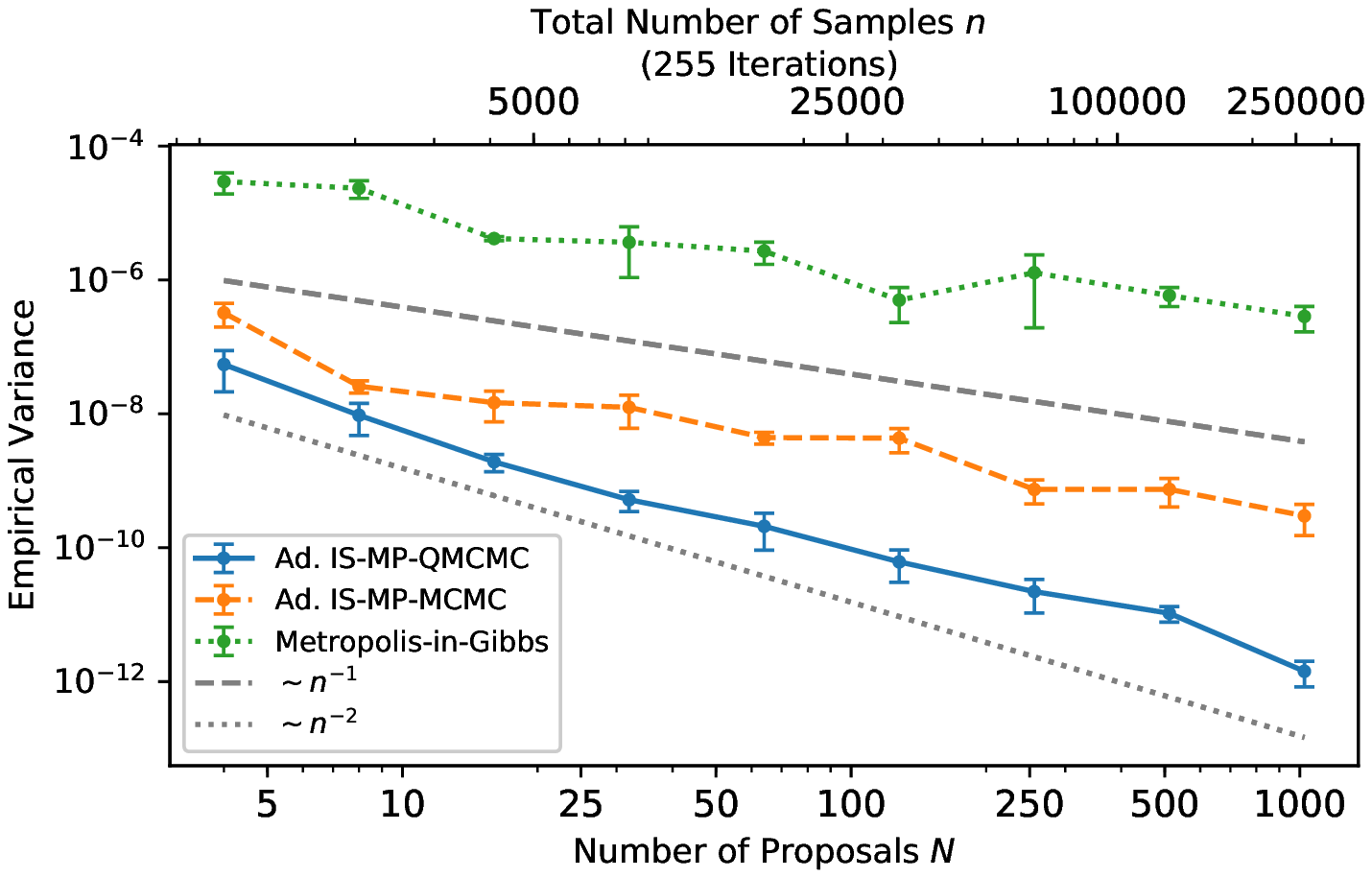}
        \subcaption{FitzHugh-Nagumo}
    \end{subfigure}     
    \caption{\small{Empirical variance of the posterior mean estimate associated to the (a) Lotka-Volterra and (b) FitzHugh-Nagumo model inference problem, 
using a pseudo-random (PSR) vs.\ CUD (QMC) seed, resp., for increasing proposal numbers and sample sizes}; also displayed are the results for the random-walk Metropolis-Hastings within Gibbs as a reference, tuned to an acceptance rate between $20$-$40 \%$. The results are based on $10$ MCMC simulations. The error band correspond to twice a standard deviation} 
    \label{fig:emprVar_ode_plus_mig}
\end{figure}

\begin{table}[h]
\ra{1.2}
\centering
\caption{
Ratio of empirical variances of adaptive IS-MP-QMCMC compared to adaptive IS-MP-MCMC and random-walk Metropolis-Hastings within Gibbs (M-H in Gibbs), 
resp., for the ODE model inference problems from
\ref{subsubsec:lotvol} and \ref{subsubsec:fitznag}, respectively. Here, Metropolis-Hasting within Gibbs was tuned to an acceptance rate between $20$-$40 \% $. The results are based on $10$ MCMC simulations}
\centering
\resizebox{1.\textwidth}{!}{

\begin{tabular}{  @{} l  @{} c @{} c @{} *9c @{}}  \bottomrule
  \multirow{2}{*}{} 
& \multicolumn{1}{c}{{Ad.\ IS-MP-QMCMC}} & \multicolumn{10}{c}{{Ratio in empirical variance for $N=$}} \\
\cmidrule{2-2} \cmidrule{4-12}
 & {{VS}} &  & $4$& $8$& $16$ & $32$ & $64$ & $128$ & $256$ & $512$ & $1024$ \\
 \midrule

\multirow{ 2}{*}{Lotka-Volterra} & {{Ad.\ IS-MP-MCMC}} & \hspace{5mm} &
	$3.9$ & $1.9$ & $7.3$ & $16.5$ & $3.5$ & $31.6$ & $62.7$ &
	$91.5$ & $76.2$ \\ 
 & {{M-H in Gibbs}} &  & $389.3$& $5129.5$ & $32677.7$ & $61495.1$ &
	$127050.1$ & $381491.6$ & $358768.2$ & $1559126.1$ & $196450.0$ \\\midrule
\multirow{ 2}{*}{FitzHugh-Nagumo} & {{Ad.\ IS-MP-MCMC}} &  &
	$5.9$ & $2.7$ & $7.6$ & $24.0$ & $21.1$ & $70.5$ & $33.5$ &
	$70.9$ & $209.6$ \\
	& {{M-H in Gibbs}} &  & $540.3$ & $2455.1$ & $2160.8$ & $7005.2$ &
	$12887.9$ & $8128.5$ & $57942.5$ & $55702.0$ & $199902.3$ \\
 \bottomrule

\end{tabular}
\label{table:results_ode_inference_emprVar}
}
\end{table}

\subsection{Intuition for increased convergence using importance sampling and QMC combined}

Vanilla MP-MCMC has, independently of the underlying driving sequence, an
acceptance mechanism that works similarly to classical Metropolis-Hastings:
some proposals are accepted while others are refused, according to some 
acceptance ratio. This procedure can be viewed as introducing a discontinuity
in the mapping between seed and actual samples. This can be explained more
precisely as follows. Every proposal of dimensionality $d$ is generated
based on a tupel of size $d$ from the underlying driving sequence. If the
driving sequence is QMC, then it is homogeneously distributed on the unit
interval. To refuse a proposal means to discard the underlying tupel from
the driving sequence that created it. However, discarding points means
figuratively speaking creating holes in the otherwise equidistributed 
point set, which damages its homogeneity. Since this homogeneity
is the core of performance gain brought by QMC we cannot expect 
vanilla MP-QMCMC to converge at a faster rate. In contrast to this, in
the extension to importance sampling every proposal is accepted and weighted 
such that the entire underlying seed is respected. Therefore the previously 
mentioned discontinuity is removed. Convergence rates close to $n^{-2}$ known 
from vanilla QMC are then possible as our numerical experiments illustrate.

\subsection{Pros and Cons of using CUD numbers in MCMC}

Whenever the regularity conditions formulated in Section 
\ref{subsection:consistency} are satisfied, consistency of the 
resulting MCMC algorithm is guaranteed. However, in applications
where these do not hold true, e.g.\ in the algorithm used in
\ref{mp_qmcmc_empirical_results}, numerical experiments suggest that
the resulting samples asymptotically still follow the target 
distribution. In none of the experiments that we performed the CUD
version of MP-MCMC along with its extensions (importance sampling,
MALA, nonreversible kernels, etc.) has performed significantly worse
than standard MCMC using a pseudo-random seed. However, in many
simulations the benefit of using CUD numbers as driving sequence
compared to IID numbers is substantial, see Sections \ref{mp_qmcmc_empirical_results}, 
\ref{adaptive_is_mp_qmcmc_simple_Gaussian_example},
\ref{adaptive_is_mp_qmcmc_empirical_results} and \ref{subsec:bayesian_inference_non_linear_odes}.

The downside of using CUD numbers is that one relies on a construction
that must first be developed, implemented and computed previous to 
running any MCMC experiments. The implementation might become expensive
when the number of required samples is large. However, once a finite
sequence is constructed, it can be reused for as many applications and
experiments as one requires. Further, many CUD constructions do not allow
a user-defined specific length of the sequence, but rather certain
unique lengths. For example, some CUD constructions correspond to
PRNGs with a short period $p$. In the case of this work, an LFSR construction
was used which has a period of $p=2^s-1$ for $s \in \mathbb{N}$. Thus,
sample sizes that differ by roughly a factor of two can be achieved,
which limits the general applicability in the case of user-defined sample 
sizes.

\section{Discussion and Conclusions}

There is a rich history of research in QMC, importance sampling and MCMC methods, however these have tended to take different paths and have developed into somewhat separate research communities.  This paper adds to the growing literature that ties together these complementary directions, and provides a demonstration that combining these ideas can result in computationally efficient and practically useful methodology, which we hope will prompt much more research into the intersection of these Monte Carlo techniques.

\subsection{Contributions}

We have significantly built upon a recent generalisation of Metropolis-Hastings, which allows for straight-forward parallelisation of MCMC by making multiple proposals at each iteration, and we have proposed numerous methodological extensions and proven some fundamental theoretical results. In particular, we investigated the use of non-reversible and optimised transition kernels within the proposal sub-sampling step of this method, and compared the relative performance of these approaches through a simulation study.  We then extended this basic algorithm to make use of adaptivity of the proposal kernel and importance sampling, which can be considered as the limiting case of sampling from the finite state Markov chain on the multiple proposals.  In addition, for the seven proposed algorithms, we have proven a variety of theoretical results, including limit theorems, asymptotic unbiasedness of the proposed importance sampling based estimators and ergodicity of the proposed adaptive samplers.

We then showed how this general framework offers a principled and effective way of incorporating CUD numbers from the quasi-Monte Carlo literature into Markov chain Monte Carlo.  In the case of using a driving sequence of CUD numbers, we prove consistency and asymptotic unbiasedness of our method under suitable regularity conditions.  Furthermore, we demonstrated that the use of importance sampling based estimators together with CUD numbers results in an MCMC method, whose mean square error empirically converges at a rate closer to $n^{-2}$ rather than the standard ${n^{-1}}$ of pseudo-randomly driven MCMC algorithms.  We argue that importance sampling removes the discontinuity induced by the acceptance threshold inherent in standard Metropolis-Hastings, thereby incorporating all points from the underlying homogeneously distributed CUD driving sequence. This leads to a smaller discrepancy, which is generally the basis of an increased performance when using QMC over pseudo-random numbers.

\subsection{Further research directions}

The work we have presented in this paper offers many interesting and potentially very useful theoretical, methodological and practical avenues for further research that tie together ideas from the MCMC and QMC literature.

We have provided strong numerical evidence of the increased convergence rates that are possible by incorporating CUD numbers within a highly parallelisable multiple proposal MCMC algorithm, and so a natural question is whether theoretical results on the convergence rate for IS-MP-QMCMC are possible?  While it appears to be a very challenging problem to tackle theoretically, some initial results on the convergence rate of CUD driven MCMC are given in \cite{chen2011consistencythesis}, which may potentially be extended.  The proofs we derive in this paper depend on the existence of a coupling region, which exists for an independent sampler, as well as asymptotically when incorporating adaptivity.  Empirical results suggest that this should also hold for dependent proposals, however a coupling region does not appear to be available for such an argument.  Further research could therefore investigate whether there exists a consistency proof that does not rely on the coupling region condition, perhaps based on the contraction condition in \cite{chen2011consistency} instead.

From a methodological perspective, there are many ways in which our work could be extended, for example investigating the use of variational approximations as proposal distributions, or indeed the use of optimal transport approximations for highly non-Gaussian target densities.  With the advent of RKHS and low-discrepancy sequences tailored to the underlying integrand, more challenging integration problems in very high-dimensional spaces can be more efficiently solved using QMC, circumventing the curse of dimensionality.  In light of this development, the use of QMC in MCMC becomes significantly more relevant, and this link is worthy of further investigation.  Furthermore, in this work we have used one particular construction of CUD numbers, although there is much research currently taking place in this area.  Could other constructions perhaps offer even greater efficiency gains within the proposed framework?

Finally, there are many practical avenues of the above suggestions for research, including investigation of this methodology across a wider range of statistical models, experimental comparisons of different CUD constructions, and eventually the development of robust implementations in software to allow a wider range of practitioners to benefit more easily from the increased convergence rates and parallelisation offered by these multiple proposal quasi-Markov chain Monte Carlo methods.

\bibliographystyle{alphamod}
\bibliography{sample}

\newpage

\appendix

\section{Transition kernel derivation, Section \ref{subsec:derivation_mpmcmc}}

Here, we derive the equations from Section 
\ref{subsubsec:transition_probabilities_product_space}
and Section 
\ref{subsubsec:trans_probs_sample_state_space} 
for the transition probabilities based on the formulation of the
Markov chain defined on the
product space of proposals and auxiliary variables per iteration, and on the
space of accepted variables per iteration, respectively, in detail. We thereby
make use of the same notation as in the respective sections above.

\subsection{Derivation in Section 
\ref{subsubsec:transition_probabilities_product_space}}
\label{appendix:transition_probabilities_product_space}

A more thorough derivation of equation \eqref{eq:last_eq_transition_kernel_mp_mcmc},
using the notation from Section
\ref{subsubsec:transition_probabilities_product_space} is,
\begin{align*}
\hat{P}(\tilde{z}, z) &= \hat{P}((\tilde{\vec{y}}_{1:N+1}, \tilde{I}_{1:M}=\tilde{i}_{1:M}), (\vec{y}_{1:N+1}, I_{1:M}=i_{1:M}))
\\
&= \kappa(\tilde{\vec{y}}_{\tilde{i}_M}, \vec{y}_{\setminus{i_M}}) p(I_{1:M}=i_{1:M}|\vec{y}_{1:N+1}, \tilde{I}_M=\tilde{i}_M)\\
&= \kappa(\vec{y}_{i_0}, \vec{y}_{\setminus{i_0}}) \prod_{m=1}^M p(I_m=i_m|\vec{y}_{1:N+1}, I_{m-1} = i_{m-1})\\
&= \kappa(\vec{y}_{i_0}, \vec{y}_{\setminus{i_0}}) \prod_{m=1}^M A(i_{m-1}, {i_m}),
\end{align*}
where we used that $\tilde{i}_M=i_0$ and $\tilde{\vec{y}}_{\tilde{i}_M} = \vec{y}_{i_0}$.
If the latter is not assumed, we need to add the term $\delta_{\tilde{\vec{y}}_{\tilde{i}_M}}(\vec{y}_{i_0})$
to the expression of the transition kernel. Equation
\eqref{eq:transition_kernel_set_mp_mcmc1} then follows in detail, as
\begin{align*}
\hat{P}(\tilde{z}, B) &= \hat{P}((\tilde{\vec{y}}_{1:N+1}, \tilde{I}_{1:M}=\tilde{i}_{1:M}), C_{1:{N+1}} \times D_{1:M})
\\
&= \int_{C_{1:N+1}} \kappa(\vec{y}_{i_0}, \vec{y}_{\setminus{i_0}}) \delta_{\tilde{\vec{y}}_{\tilde{i}_M}}(\vec{y}_{i_0})  \sum_{i_{1:M}\in D_{1:M}} 
\prod_{m=1}^M A(i_{m-1}, i_m)
\mathrm{d}\vec{y}_{1:N+1}\\
&= \chi_{C_{i_0}(\tilde{\vec{y}}_{\tilde{i}_M})}\int_{C_{\setminus{i_0}}}\kappa(\vec{y}_{i_0}=\tilde{\vec{y}}_{\tilde{i}_M},  \vec{y}_{\setminus{i_0}}) \sum_{i_{1:M}\in D_{1:M}} \prod_{m=1}^M A(i_{m-1}, i_m| \vec{y}_{i_0}=\tilde{\vec{y}}_{\tilde{i}_M})
\mathrm{d}\vec{y}_{\setminus{i_0}},
\end{align*}
where we used the same notation as in Section 
\ref{subsubsec:transition_probabilities_product_space}.

\subsection{Transition kernel derivation, Section 
\ref{subsubsec:trans_probs_sample_state_space}}
\label{appendix:transition_probabilities_accepted_samples}

In order to derive equation \eqref{eq:mp_mcmc_transition_kernel_general_case1a}
in more detail, we consider the relationship
between the transition kernel $\hat{P}$ on the product space of proposals and auxiliary
variables and the transition kernel $P$ on the space of accepted samples
more thoroughly. First, let us treat the case of $M=N=1$, which corresponds to standard
Metropolis-Hastings. 
Based on the current
sample $\tilde{\vec{x}}$, or in terms of $\hat{P}$ any state 
$(\tilde{\vec{y}}_{1:2}, \tilde{I}=\tilde{i})$ with
$\tilde{\vec{y}}_{\tilde{i}}=\tilde{\vec{x}}$, for the next accepted sample it holds
$\vec{x}\in B\in \mathcal{B}(\Omega)$, either
if $\tilde{\vec{x}}\in B$ and the additional proposal is refused or the additional
proposal is $\in B$ and is moreover accepted, i.e.\
\begin{align}
P(\tilde{\vec{x}}, B) &= \chi_B(\tilde{\vec{x}}) \int_{\mathbb{R}^d} \kappa(\vec{y}_{i_0}=\tilde{\vec{x}}, \vec{y}_{\setminus{i_0}})\left[ 1- A(i_0, {\setminus{i_0}}|\vec{y}_{i_0}=\tilde{\vec{x}}) \right]
\mathrm{d}\vec{y}_{\setminus{i_0}} \nonumber\\
&\quad 
+ \int_B \kappa(\vec{y}_{i_0}=\tilde{\vec{x}}, \vec{y}_{\setminus{i_0}})A(i_0, {\setminus{i_0}}|\vec{y}_{i_0}=\tilde{\vec{x}})\mathrm{d}\vec{y}_{\setminus{i_0}} 
\nonumber\\
&= \int_{B}\int_{\mathbb{R}^d}\delta_{\tilde{\vec{x}}}(\vec{y}_{i_0}) \kappa(\vec{y}_{i_0}, \vec{y}_{\setminus{i_0}}) \left[ 1- A(i_0, {\setminus{i_0}}) \right]
\mathrm{d}\vec{y}_{\setminus{i_0}}\mathrm{d}\vec{y}_{i_0}
\nonumber\\
&\quad+ \int_{\mathbb{R}^d}\int_{B}\delta_{\tilde{\vec{x}}}(\vec{y}_{i_0})\kappa(\vec{y}_{i_0}, \vec{y}_{\setminus{i_0}})
A({i_0}, {\setminus{i_0}})\mathrm{d}\vec{y}_{\setminus{i_0}}\mathrm{d}\vec{y}_{i_0} 
\nonumber\\
&= \hat{P}\left(\tilde{\vec{x}}, \left[B\times\mathbb{R}^d\times\{i_0\}\right] \ \cup\ \left[\mathbb{R}^d\times B\times\{\setminus i_0\}\right]\right).
\nonumber
\end{align}
The case $M=1$, and general $N\in\mathbb{N}$, can be treated in a similar fashion. 
In order to have $\vec{x}\in B\in \mathcal{B}(\Omega)$, 
either $\tilde{\vec{x}} \in B$ and all $N$ 
additional proposals are refused, or one of the additional proposals is $\in B$
and is moreover accepted, the latter case factorising into $N$ separate cases:
\begin{align*}
{P}(\tilde{\vec{x}}, B) &= \chi_B(\tilde{\vec{x}}) \int_{\Omega^N} \kappa(\vec{y}_{i_0}=\tilde{\vec{x}}, \vec{y}_{\setminus{i_0}})
\bigg[ 1- \sum_{i \neq i_0}A(i_0, i |\vec{y}_{i_0}=\tilde{\vec{x}}) \bigg]
\mathrm{d}\vec{y}_{\setminus{i_0}}\\
&\quad + \sum_{i=1}^N \int_{\Omega^{(i-1)d}\times B \times \Omega^{(N-i)d}}\kappa(\vec{y}_{i_0}=\tilde{\vec{x}}, \vec{y}_{\setminus{i_0}})
A(i_0, i |\vec{y}_{i_0}=\tilde{\vec{x}}) \mathrm{d}\vec{y}_{\setminus{i_0}}\\
&=\sum_{i=1}^{N+1}\int_{\Omega^{(i-1)d}\times B \times \Omega^{(N+1-i)d}}
\delta_{\tilde{\vec{x}}}(\vec{y}_{i_0})\kappa(\vec{y}_{i_0}, \vec{y}_{\setminus{i_0}})
A(i_0, i) \mathrm{d}\vec{y}_{1:N+1}\\
&= \hat{P}\bigg(\tilde{\vec{x}}, \bigcup_{i=1}^{N+1} \left[\Omega^{(i-1)d}\times B \times \Omega^{(N+1-i)d} \times \{i\}\right]\bigg),
\end{align*}
The general case of $M,N \in\mathbb{N}$, corresponding to equation
\eqref{eq:mp_mcmc_transition_kernel_general_case1a} now follows from
carefully considering what composition of sets allow for 
$\vec{x}_{1:M}\in B \in \Omega^M\subset \mathbb{R}^{Md}$.

\section{Proof of MP-MCMC limit theorems, Section \ref{subsection:limit_theorems}}

We now prove the law of large numbers and the central limit theorem for
MP-MCMC.

\subsection{Proof of Lemma \ref{lemma:lln_mp_mcmc}, LLN for MP-MCMC}
\label{appendix:proof_lemma_lln_mp_mcmc}
\begin{proof}
It was proven that the Markov chain on the product space on variables
$(\vec{y}_{1:N+1}, I)$ has an invariant distribution that is preserved
by the updating kernel due to detailed balance. According to
\ref{subsubsec:equivalence_mpmcmc}, the same chain
can be described on the accepted samples $\vec{x}_{1:M}\in \Omega^M$ 
in one iteration or on the product space of proposals and auxiliary 
variables $(\vec{y}_{1:N+1}, I_{1:M})\in \Omega^{N+1}\times \{1,...,N+1\}^M$ in one iteration. The
resulting stationary distribution on any of the respective spaces,
we denote for simplicity by $p$.
We apply Theorem 17.0.1 and Theorem 17.1.6 from \cite{meyn2012markov},
which state that for any scalar-valued, integrable function $F$ 
on $\mathbb{R}^{Md}$ it holds $\hat\mu_{F,n} \rightarrow \mu_F$ a.s., where
\begin{align*}
    \hat\mu_{F,n} &= \frac{1}{n}\sum_{\ell=1}^n F(\vec{x}_{1:M}^{(\ell)}), \text{ and},\\
    \mu_F &= \mathbb{E}_p\left[ F(\vec{x}_{1:M}) \right]. 
\end{align*}
Let us define
\begin{align*}
	F(\vec{x}_{1:M}) =  \frac{1}{M} \sum_{m=1}^{M} f(\vec{x}_{m}).
\end{align*}
Thus, $\hat\mu_{F,n} = \hat\mu_{n,M,N}$, as defined in
\eqref{eq:lln_mp_mcmc_inner_and_outer_iterations}.
Further, note that $\vec{x}_m^{(\ell)} = \vec{y}_{I_m^{(\ell)}}^{(\ell)}$, 
with the usual notation, for any $\ell=1,...,n \in \mathbb{N}$. Therefore, 
$F(\vec{x}_{1:M}) = 1/M \sum_{m} f(\vec{y}_{I_m})$, and,
\begin{align*}
    \mu_F&= \frac{1}{M}\sum_{m=1}^{M} \int_{\Omega^M} f(\vec{x}_m) p(\vec{x}_{1:M})
    \mathrm{d}\vec{x}_{1:M}\\
    &=\frac{1}{M}\sum_{m=1}^{M} \int_{\Omega} f(\vec{x}_m)p(\vec{x}_{m})
    \mathrm{d}\vec{x}_{m}
    \\
    &=\frac{1}{M}\sum_{m=1}^{M}\sum_{i_m=1}^{N+1} \int_{\Omega^{N+1}} f(\vec{y}_{i_m})p(y_{1:N+1},I_m=i_{m}) \mathrm{d}\vec{y}_{1:N+1}\\
  &= \frac{1}{M}\sum_{m=1}^{M}\sum_{i_m=1}^{N+1} \int_{\Omega^{N+1}} f(\vec{y}_{i_m}) \frac{1}{N+1}
  \pi(\vec{y}_{i_m}) 
  \kappa(\vec{y}_{i_m}, \vec{y}_{\setminus{i_m}})
  \mathrm{d}\vec{y}_{1:N+1}  \\   
    &= \frac{1}{N+1}\sum_{i=1}^{N+1} \int_{\Omega^{N+1}} f(\vec{y}_{i})\pi(\vec{y}_{i})\kappa(\vec{y}_{i}, \vec{y}_{\setminus{i}}) \mathrm{d}\vec{y}_{1:N+1}\\
    &= \frac{1}{N+1}\sum_{i=1}^{N+1}\int_{\Omega} f(\vec{y}_{i})\pi(\vec{y}_{i})\left[\int_{\Omega^N} \kappa(\vec{y}_{i}, \vec{y}_{\setminus{i}}) \mathrm{d}\vec{y}_{\setminus{i}}\right]\mathrm{d}\vec{y}_{i} \\
    &= \mu,
\end{align*}
where in the last line, we used
$\int \kappa(\vec{y}_{i}, \vec{y}_{\setminus{i}}) \mathrm{d}\vec{y}_{\setminus{i}}=1$
for any $i=1,...,N+1$.
This concludes the proof.
\end{proof}

\subsection{Proof of Lemma \ref{lemma:clt_mp_mcmc}, CLT for MP-MCMC}
\label{appendix:proof_lemma_ctl_mp_mcmc}

Before proving the CLT we specify a condition under which the Markov chain
on the accepted samples per iteration, defined by MP-MCMC, is \textit{well-behaved},
i.e.\ such that the asymptotic variance can be represented by the limit of 
the variances at iteration $n$ for $n \rightarrow \infty$, and is well-defined
and positive.
Referring to the proof of the CLT in \cite{meyn2012markov}, such a condition
can be formulated by $V$-uniform ergodicity of the underlying Markov chain defined
on the $M$ accepted samples per iteration: an ergodic Markov chain on 
$\vec{z}\in \Omega^M \subset \mathbb{R}^{Md}$ with limiting
distribution $p$ and transition function $P^n$ for $n=1,2,...$ is called 
$V$-uniformly ergodic, with a positive function $1\le V < \infty$, if
\begin{align}
    \| P^n - p\|_V \rightarrow 0 \quad \text{for} \quad n \rightarrow \infty,
    \label{eq:uniform_V_ergodicity}
\end{align}
where
\begin{align}
    \|P_1 - P_2\|_V := \sup_{\vec{z}\in \Omega^M}\frac{\| P_1(\vec{z},\cdot) - P_2(\vec{z},\cdot)\|_V}{V(\vec{z})},
    \label{eq:V_norm}.
\end{align}
In \eqref{eq:uniform_V_ergodicity}, we set $p(\vec{z},B) = p(B)$ for any
$B\in \mathcal{B}(\mathbb{R}^{Md}), \vec{z}\in \Omega^M$. Note that $p$ is
in this case understood as the probability measure associated with the 
stationary distribution of the Markov chain on the accepted samples. Moreover, the 
norm in \eqref{eq:V_norm} is defined by
\begin{align*}
    \| \nu \|_V = \sup_{U \le V}\left|\int_X U(\vec{z}) \nu(\mathrm{d}\vec{z})\right|.
\end{align*}
Let us define $\gamma_F^2$ by
\begin{align*}
    \gamma_F^2 &= \mathbb{E}_p \left[ \bar F^2
    \left(\vec{z}^{(1)}\right) \right]\\
    &\quad\quad+ 2\sum_{k=2}^\infty \mathbb{E}_p \left[ \bar{F}\left(\vec{z}^{(1)} \right)
    \bar F\left(\vec{z}^{(k)}\right)\right]
\end{align*}
where $\bar{F} = F - \int F(\vec{z}) p(\vec{z}) \mathrm{d}\vec{z}$, and $\vec{z}^{(n)}$
denotes the state of the chain in the $n$th iteration for 
$n\in\mathbb{N}$.
Under the assumption that $F^2\le V$, the constant $\gamma_F^2$
is well-defined, non-negative and finite according to Theorem 17.0.1 from
\cite{meyn2012markov}, if the underlying Markov chain is 
V-uniformly ergodic. Further, if $\gamma_F^2>0$, then the CLT holds true for that
chain, i.e., which is defined on the $M$ accepted samples per iteration. Note that we 
aim on deriving a CLT for individual accepted states of MP-MCMC, which 
follows now.
\begin{proof}[Proof of Lemma \ref{lemma:clt_mp_mcmc}, CLT for MP-MCMC]
Since detailed balance holds true, the joint distribution $p$ is the invariant
distribution of the chain on the product space, as defined by MP-MCMC.
Therefore, we can apply Theorem 17.0.1 and Theorem 17.3.6 from \cite{meyn2012markov},
which ensure the CLT to hold true on this chain. Thus,
\begin{align}
    \sqrt{n}\left( \hat\mu_{F,n} - \mu \right) 
    \xrightarrow{\mathcal{D}}
    \mathcal{N}\left( 0, \sigma_{F}^2 \right),
    \label{eq:ctl_mp_mcmc_outer_iterations}
\end{align}
where $\sigma_{F}^2$ denotes the asymptotic variance of the sequence
of samples $(F(\vec{x}_{1:M}^{(i)}))_{i \ge 1}$, which, in real problems, 
cannot be determined exactly but can be estimated by the same run (several
runs in the case where batch method and ESS estimate cannot be applied) of 
the Markov chain that produced the estimate $\hat\mu_{F,n}$. The variance 
$\sigma_{F,n}^2$ of the expression on the left hand side of 
\eqref{eq:ctl_mp_mcmc_outer_iterations} is given by
\begin{align*}
    \sigma_{F,n}^2 &= \frac{1}{n}\sum_{i=1}^n \Var \left( F (\vec{x}_{1:M}^{(i)}) \right)\\
    &\quad\quad+ \frac{2}{n} \sum_{1\le i<j\le n} \Cov\left( F (\vec{x}_{1:M}^{(i)}), F (\vec{x}_{1:M}^{(j)}) \right).
\end{align*}
According to Theorem 17.1.6 from \cite{meyn2012markov}, if the CLT holds
true for a particular initial distribution of the Markov chain, then it automatically
holds true for every initial distribution. Thus, as the asymptotic behaviour 
of the Markov chain does not depend on the
initial distribution, we may assume the target distribution as initial
distribution. It follows the stationarity of the sequence
$(F(\vec{x}_{1:M}^{(i)}))_{i \ge 1}$, which implies
\begin{align*}
    \gamma^{(0)} := \Var\left( (F(\vec{x}_{1:M}^{(i)}) \right)
\end{align*}
is the same for any $i=1,2,...$, and similarly,
\begin{align*}
    \gamma^{(k)} := \Cov\left( F(\vec{x}_{1:M}^{(i)}), F(\vec{x}_{1:M}^{(i+k)}) \right)
\end{align*}
depends only on the lag $k$ between two samples on the product space. We call $\gamma^{(k)}$ the lag-$k$ auto-covariance of the series
$(F(\vec{x}_{1:M}^{(i)}))_{i \ge 1}$. Due to
stationarity we have
\begin{align*}
    \sigma_{F,n}^2 &= \gamma^{(0)} + 2 \sum_{k=1}^{n-1}\frac{n-k}{n}\gamma^{(k)}\\
    &\quad\xrightarrow{n \rightarrow \infty} \gamma^{(0)} + 2 \sum_{k=1}^\infty \gamma^{(k)} = \sigma_{F}^2,
\end{align*}
and the limit is well-defined and positive. 
We now want to derive an expression for the asymptotic variance $\sigma^2$
in terms of the function of interest $f$. In order to do so note that
\begin{align*}
    M\gamma^{(0)} &= \frac{1}{M}\sum_{m=1}^{M}\Var\left( f(\vec{x}_m) \right) 
    + \frac{2}{M}\sum_{1\le \ell<m \le M}\Cov\left( f(\vec{x}_\ell), f(\vec{x}_m) \right)\\
    &= \zeta^{(0)} + \frac{2}{M}\sum_{1\le \ell< m\le M} \zeta^{(0)}_{\ell,m},
\end{align*}
where $\zeta^{(0)}=\Var(f(\vec{x}_m))$ is independent of $m$, and 
$\zeta_{\ell,m}^{(0)} = \Cov(f(\vec{x}_\ell),f(\vec{x}_{m}))$.
Here, we used the stationarity of the Markov chain on the product space.
Similarly, for any $k \ge 1$,
\begin{align*}
    M\gamma^{(k)} &= \frac{1}{M}\sum_{\ell,m=1}^{M}\Cov\left( f(\vec{x}_\ell^{(i)}, f(\vec{x}_m^{(i+k)}) \right)\\
    &= \frac{1}{M}\sum_{\ell,m=1}^{M} \zeta^{(k)}_{\ell,m},
\end{align*}
where $\zeta^{(k)}_{\ell,m} = \Cov( f(\vec{x}_\ell^{(i)}, f(\vec{x}_m^{(i+k)}))$ for 
any $i\ge 1$. Summarising, we have
\begin{align*}
    \sqrt{nM}\left( \hat\mu_{n,M,N}- \mu \right)
    \xrightarrow{\mathcal{D}} \mathcal{N}(0,\sigma^2) \quad \text{for} \quad n\rightarrow \infty,
\end{align*}
where the asymptotic variance can be expressed as
\begin{align*}
    \sigma^2 &= M \sigma_F^2\\
    &= M \left[ \gamma^{(0)} + 2 \sum_{k=1}^{\infty}\gamma^{(k)} \right]\\
    &= \zeta^{(0)} + 2\sum_{1\le \ell<m \le M} \zeta^{(0)}_{\ell,m} + \frac{2}{M}\sum_{k=1}^\infty \sum_{\ell,m=1}^{M} \zeta^{(k)}_{\ell,m}.
\end{align*}

\end{proof}

\section{Proofs for ergodicity of adaptive MP-MCMC}
\label{app:sec:ergodicity_adaptive_mpmcmc}

In the following, we show ergodicity of adaptive MP-MCMC under the conditions
given by Theorem \ref{thm:ergodicity_adap_mpmcmc_independent},
Theorem, \ref{thm:ergodicity_adapt_mpmcmc_bounded} and
Theorem \ref{thm:ergodicity_adaptive_mpmcmc_positive}.

\subsection{Ergodicity proof of adaptive MP-MCMC, 
Theorem \ref{thm:ergodicity_adap_mpmcmc_independent}}
\label{subsec:proof_ergodicity_adapt_mpmcmc_independent}

\begin{proof}
We shall use Theorem \ref{thm:theorem2roberts2007} to prove ergodicity.
Diminishing adaptation can be shown analogously to the proofs of Theorem
\ref{thm:ergodicity_adaptive_mpmcmc_positive} and Theorem 
\ref{thm:ergodicity_adapt_mpmcmc_bounded}.
We still need to prove the containment condition.
It was shown in \cite{roberts2007coupling} that simultaneous strongly 
aperiodic geometric ergodicity implies containment, which we will use in 
what follows.

\begin{definition}
\label{def:ssage}
A family of transition kernels $\{P_\gamma: \gamma \in \mathcal{Y} \}$ is simultaneous 
strongly aperiodic geometric ergodic (S.S.A.G.E.) if there is 
$C \in \mathcal{B}(\Omega^M)$, a function $V:\Omega^M\rightarrow[1,\infty)$ 
and $\delta>0, \lambda<1$, $b<\infty$ such that $\sup_{\vec{z}\in C}V(\vec{z}) < \infty$, 
and
\begin{enumerate}[label=\arabic*)]
\item $\forall \ \gamma\in \mathcal{Y}\ \exists$ a probability measure $\nu_\gamma$ on $C$ such that $P_\gamma(\vec{z},\cdot)\ge \delta \nu_\gamma(\cdot)$ for all $\vec{z}\in C$, and
\item $P_\gamma V(\vec{z}) \le \lambda V(\vec{z}) + b \mathbbm{1}_C(\vec{z})$ for all $\gamma\in \mathcal{Y}, \vec{z}\in\Omega^M$,
\end{enumerate}
where $P_\gamma V(\vec{z}) := \mathbb{E}[V(\vec{Z}_1)|\vec{Z}_0=\vec{z}]$.
\end{definition}

In order to show S.S.A.G.E., set $C=\Omega^M$ and $V \equiv 1$. Since the 
proposal distribution is independent of previous samples, we may set 
$\delta=1$ and $\nu_\gamma = P_\gamma$. Further, by setting $\lambda=1/2$ and $b=1$ we have
\begin{align*}
P_\gamma V(\vec{z}) =1 \le \lambda \cdot 1 + b = \lambda V(\vec{z}) + b \mathbbm{1}_C(\vec{z}),
\end{align*}
which implies S.S.A.G.E., and the proof is complete.
\end{proof}

\subsection{Ergodicity proof of adaptive MP-MCMC, Theorem \ref{thm:ergodicity_adapt_mpmcmc_bounded}}
\label{subsec:ergodicity_adapt_mpmcmc_bounded}
\begin{proof}
Without loss of generality we set the state space of accepted samples 
per iteration to the restriction $S\subset \mathbb{R}^{Md}$, 
i.e.\ $\Omega^M = S$, of states in $\mathbb{R}^{Md}$ that have positive 
probability with respect to the stationary probability $p$.
The support of $\pi$ is bounded by assumption. Further, it is closed
since the support of a continuous function is the closure of sets on which
it is non-zero. Hence, the support of $\pi$ is compact, and so is $S$.
The theorem is proven 
using Theorem \ref{thm:theorem2roberts2007}. Diminishing adaptation follows in 
the same way as in the proof of Theorem \ref{thm:ergodicity_adaptive_mpmcmc_positive},
where it is used that $\mathcal{Y}$ is bounded, which follows from the 
bounded support assumption.

It remains to show the containment condition, which follows via S.S.A.G.E., which we prove now. Let $C=S$. We need to find $\delta>0$ and $\nu_\gamma$ such that 
${P}_\gamma (\vec{z}, B) \ge \delta \nu_\gamma(B)$ 
for all $B\in \mathcal{B}(C)$ and for all $\vec{z}\in C, \gamma \in \mathcal{Y}$. 
Let
\begin{align*}
\tilde{C} = \bigcup_{m=1}^M C_m.
\end{align*}
Clearly, $\tilde{C}$ is compact. Since $\pi$ is positive on $C$ and assumed to be
continuous, and since $\mathcal{Y}$ is bounded, there is a $c_A>0$ such that
\begin{align}
c_A \le \frac{\pi({\vec{y}}_i)N_{{\gamma}} ({\vec{y}}_{i}, {\vec{y}}_{\setminus{i}}) }
{\sum_{j=1}^{N+1}\pi({\vec{y}}_j)N_{{\gamma}} ({\vec{y}}_{j}, {\vec{y}}_{\setminus{j}})}, 
\label{eq:boundness_finite_chain_transition_probability}
\end{align}
for all $\vec{y}_{1}, ..., \vec{y}_{N+1} \in \tilde{C}$ and all 
$\gamma \in \mathcal{Y}$. 
Similarly, there is a $c_N>0$ such that
\begin{align*}
c_N \le N_\gamma ({\vec{y}}_{i}, {\vec{y}}_{\setminus{i}}),
\end{align*}
for all $\vec{y}_{1}, ..., \vec{y}_{N+1} \in \tilde{C}$ and all 
$\gamma \in \mathcal{Y}$.
Without loss of generality, let $i_0=1$ and thus $\vec{x}_M=\vec{y}_{1}$. 
Further, let ${i}_1, ..., {i}_M\in \{2,...,N+1\}$ be chosen fixed.
With $S_{i_m}(B)$ from \eqref{eq:Sim_set} and $T_{i_m}(B)$, defined by
\begin{align}
T_{i_m}(B)= \{
\vec{y}_{\setminus{i_0}} \in \Omega^N: 
\vec{y}_{i_m}\in B_{m}, \text{ and }
\vec{y}_{\setminus{i_m}} \in\Omega^N
\},
\label{eq:notation_trans_sets_dim_adapt}
\end{align}
for $m=1,...,M$ and $i_m=1,...,N+1$, we have
\begin{align}
{P}(\vec{z}, B) &\ge P\Big(\vec{x}_M, \bigcap_{m=1}^M S_{i_m}(B)\Big)
\\
&= \int_{\bigcap_m {T_{i_m}(B)}}
N_{\gamma}(\vec{y}_{i_0}, \vec{y}_{\setminus{i_0}}) 
\prod_{m=1}^M A({i_{m-1}}, {{i_m}})
\mathrm{d}\vec{y}_{\setminus{i_0}}\nonumber\\
&\ge c_A^M \int_{\bigcap_m {T_{i_m}(B)}}
N_\gamma (\vec{y}_{i_0}, \vec{y}_{\setminus{i_0}})
\mathrm{d}\vec{y}_{\setminus{i_0}}\nonumber\\
& \ge c_A^M c_N \lambda\left(\bigcap_m T_{i_m}(B)\right)\nonumber\\
& = \delta \nu(B),\nonumber
\end{align}
where $\delta := c_A^M c_N>0$, $\nu(B):=\lambda(\bigcap_m T_{i_m}(B))$, 
and $\lambda(\cdot)$ denotes the Lebesgue measure. This concludes the
first condition in the S.S.A.G.E.\ definition \ref{def:ssage}.
The second condition follows immediately by
setting $\lambda=1/2$, $b=1$ and $V\equiv 1$, which concludes the proof.
\end{proof}

\subsection{Ergodicity proofs of adaptive MP-MCMC, Theorem \ref{thm:ergodicity_adaptive_mpmcmc_positive}}
\label{subsec:ergodicity_adapt_mpmcmc_positive}

\begin{proof}
In the following, we consider MP-MCMC as a single Markov chain over
the space $\Omega^M \subset \mathbb{R}^{Md}$ of accepted samples in each 
iteration, as introduced in Section
\ref{subsubsec:trans_probs_sample_state_space}.
Referring to Theorem \ref{thm:theorem2roberts2007}, there are 
two things that we need to show: diminishing adaptation and containment.\\


Diminishing adaptation: Let $B=B_{1:M} \in \mathcal{B}(\Omega^M)$ and 
$\vec{z} \in \Omega^M\subset \mathbb{R}^{Md}$ be arbitrary. 
For an arbitrary but fixed $i_0 \in \{1,...,N+1\}$, let 
$T_{i_m}(B)\in \mathcal{B}(\Omega^N)$ defined as in 
\eqref{eq:notation_trans_sets_dim_adapt} for $m=1,...,M$ and 
$i_m=1,...,N+1$. Further, let
$\tilde{B}\in \mathcal{B}(\Omega^N)$ be defined by
\begin{align*}
\tilde{B} = \bigcup_{i_1,...,i_M=1}^{N+1} \bigcap_{m=1}^M T_{i_m}(B).
\end{align*}
Using the formulas for the transition kernel
for the MP-MCMC in \eqref{eq:transition_kernel_set_mp_mcmc1}
and \eqref{eq:mp_mcmc_transition_kernel_general_case1a}, 
and the notation $\vec{z} = \vec{x}_{1:M}$, $\vec{y}_{i_0}= \vec{x}_M$, 
then yields
\begin{align}
&\quad\quad 
\| P_{\Gamma_{n+1}}(\vec{z},B) - P_{\Gamma_n}(\vec{z},B) \| 
\\
&\le \int_{\tilde{B}} 
\left[N_{\Sigma_{n+1}}(\vec{y}_{i_0}, \vec{y}_{\setminus{i_0}}') 
- N_{\Sigma_n}(\vec{y}_{i_0}, \vec{y}_{\setminus{i_0}}') \right]
\sum_{i_1,...,i_M=1}^{N+1}
\prod_{m=1}^M
A(i_{m-1}, i_m)\mathrm{d}\vec{y}_{\setminus{i_0}}
\label{eq:1st_inequality_proof_dim_adapt}
\\
&\le (N+1)^M \int_{\Omega^N}
\Big|N_{\Sigma_{n+1}}(\vec{y}_{i_0}, \vec{y}_{\setminus{i_0}}') 
- N_{\Sigma_n}(\vec{y}_{i_0}, \vec{y}_{\setminus{i_0}}') \Big|\mathrm{d}\vec{y}_{\setminus{i_0}}
\nonumber\\
&= (N+1)^M \int_{\Omega^N}
\Big| 
\prod_{\ell=1}^N N_{\Sigma_{n+1}}(\vec{y}_\ell) - \prod_{\ell=1}^N N_{\Sigma_{n}}(\vec{y}_\ell)
\Big|\mathrm{d}\vec{y}_{1:N}
\nonumber\\
&\le (N+1)^M \int_{\Omega^N} \int_0^1 \Big| \frac{\mathrm{d}}{\mathrm{d}s} 
\prod_{\ell=1}^N N_{\Sigma_n + s(\Sigma_{n+1} - \Sigma_n)}(\vec{y}_\ell) \Big| \mathrm{d}s\ \mathrm{d}y_{1:N} \label{eq:upper_bound_dim_adapt} \\
&= (\star ), \nonumber
\end{align}
where we used that $A(i_{m-1}, i_m)\le 1$ for any $m=1,...,M$ and
any $i_1,...,i_M=1,...,N+1$, and $\tilde{B}\subset \Omega^N$. Note that
the smaller-or-equal sign becomes an equal sign in
\eqref{eq:1st_inequality_proof_dim_adapt} if $\vec{x}_M\in B_m$ for any
$m=1,...,M$.
Further, setting $A_{n}(s)= \Sigma_n + s(\Sigma_{n+1}-\Sigma_n)$ leads to
\begin{align}
\prod_{\ell=1}^N N_{\Sigma_n + s(\Sigma_{n+1} - \Sigma_n)}(\vec{y}_\ell) = (2\pi)^{dN/2}\det\left(A_{n}(s)\right)^{-N/2}
\exp\left( - \frac{1}{2}\sum_{\ell=1}^N \vec{y}_\ell^T A_n(s)^{-1} \vec{y}_\ell \right).
\label{eq:product_equality_dim_adapt}
\end{align}
For the first of the two individual terms of derivatives of the 
product on the right hand side of \eqref{eq:product_equality_dim_adapt}, 
we have 
\begin{align}
&\quad\ \Big|\frac{\mathrm{d}}{\mathrm{d}s}\left[(2\pi)^{dN/2} \det\left( A_n(s) \right)^{-N/2}\right]\Big| 
\nonumber
\\
&= \Big|(2\pi)^{dN/2}\frac{N}{2}\det\left( A_n(s) \right)^{-N/2-1} \frac{\mathrm{d}}{\mathrm{d}s}\left[ \det( A_n(s)) \right]\Big|\nonumber\\
&= \Big|(2\pi)^{dN/2}\frac{N}{2}\det\left( A_n(s) \right)^{-N/2-1}
\det( A_n(s) )\operatorname{tr}\left( A_n(s) ^{-1}(\Sigma_{n+1}-\Sigma_n)\right)\Big|
\nonumber\\
&\le \const \left\| \Sigma_{n+1} - \Sigma_n \right\|, \label{eq:exp_term_estimate_dim_adapt2}
\end{align}
where in the third line we used Jacobi's formula, and in the last 
line we used $0<c_1\le \det(A_n(s))\le c_2 < \infty$ for any 
$n\in \mathbb{N}$, which is a consequence of $c_1I\le \gamma \le c_2I$ 
for any $\gamma \in \mathcal{Y}$, i.e.\ the boundedness of 
$\mathcal{Y}$. Moreover, we used
\begin{align}
\operatorname{tr}\left( A_n(s)^{-1}(\Sigma_{n+1}-\Sigma_n) \right) &= \langle A_n(s)^{-1}, \Sigma_{n+1}-\Sigma_n \rangle_F
\nonumber\\
&\le\|A_n(s)^{-1}\|_F \|\Sigma_{n+1}-\Sigma_n\|_F
\nonumber\\
&\le \const \|\Sigma_{n+1}-\Sigma_n\|,
\label{eq:trace_estimate_dim_adapt}
\end{align}
where $\langle \cdot, \cdot \rangle_F$ denotes the Frobenius inner 
product and $\| \cdot \|_F$ the associated Frobenius norm, for which
we made use of the Cauchy-Schwarz inequality. For the last estimate,
i.e.\ the boundedness of the Frobenius norm of $A_n(s)^{-1}$ we refer
to the calculations below. Note that we do not need 
to further define the norm used in \eqref{eq:exp_term_estimate_dim_adapt2} 
and \eqref{eq:trace_estimate_dim_adapt}
since all norms are equivalent over finite-dimensional linear spaces. 
For the second term of derivatives on the right hand side of 
\eqref{eq:product_equality_dim_adapt} we have
\begin{align*}
&\quad \ \Big| \frac{\mathrm{d}}{\mathrm{d}s}\left[ \exp\left( -\frac{1}{2}\sum_{\ell=1}^N \vec{y}_\ell^T A_n(s)^{-1} \vec{y}_\ell \right) \right] \Big|
\\
&= \Big|\exp\left( -\frac{1}{2}\sum_{\ell=1}^N \vec{y}_\ell^T A_n(s)^{-1} \vec{y}_\ell \right)\sum_{\ell=1}^N \vec{y}_\ell^T A_n(s)^{-1}(\Sigma_{n+1}-\Sigma_n)A_n(s)^{-1} \vec{y}_\ell\Big| 
\\
&\le \const \cdot \sum_{\ell=1}^N \vec{y}^T_\ell \vec{y}_\ell
\exp\left( -\frac{1}{2}\sum_{\ell=1}^N \vec{y}_\ell^T A_n(s)^{-1} \vec{y}_\ell \right)
\|\Sigma_{n+1}-\Sigma_n\|,
\end{align*}
where we in the last line we used some basic properties of the Schur
complement of submatrices in 
\begin{align*}
B_n(s)=
\begin{bmatrix}
A_n(s) & I \\
I & c_1^{-1}I
\end{bmatrix},
\end{align*}
referring to \cite{zhang2006schur}.
More precisely, since $c_1^{-1}I>0$ and $B_n(s) / c_1^{-1}I = A_n(s)-c_1I \ge 0$
for any $s \in [0,1]$, it follows that $B_n(s)\ge 0$. As $A_n(s)^{-1}>0$,
this is equivalent to $B_n(s)/ A_n(s) = c_1^{-1}I - A_n(s)^{-1} \ge 0$, the
latter being also used in equation \eqref{eq:trace_estimate_dim_adapt}.
Since $A_n(s)^{-1}$ is symmetric and positive definite, there is a unique 
symmetric square root of $A_n(s)^{-1}$, and therefore
\begin{align*}
\vec{y}^T (A_n(s)^{-1})^2 \vec{y} &= ([A_n(s)^{-1}]^{1/2} \vec{y})^T A_n(s)^{-1} ([A_n(s)^{-1}]^{1/2}\vec{y})\\
&\le c_1^{-1} ([A_n(s)^{-1}]^{1/2} \vec{y})^T([A_n(s)^{-1}]^{1/2} \vec{y})\\
&= c_1^{-1} \vec{y}^T A_n(s)^{-1} \vec{y}\\
&\le c_1^{-2} \vec{y}^T \vec{y} \quad\quad \forall\ \vec{y} \in \mathbb{R}^d.
\end{align*}
Finally, using Fubini for interchanging integration, the boundedness of moments
of the Normal distribution and the again the boundedness of $\mathcal{Y}$, 
we have
\begin{align*}
(\star ) \le \const  \|\Sigma_{n+1} - \Sigma_n\|\le \const \cdot \frac{1}{n} \rightarrow 0 \quad \text{ for } \quad n\rightarrow\infty,
\end{align*}
which proves diminishing adaptation.

\vspace{0.5cm}

Containment follows immediately by applying Theorem 21 from
\cite{craiu2015stability}, under the assumptions formulated in
\ref{assumption:bounded_jump_condition} and 
\ref{assumption:adaptation_within_compact}. This concludes the proof.

\end{proof}

\subsection{Asymptotic unbiasedness proof for IS-MP-MCMC, 
Lemma \ref{lemma:asymptotic_unbiasedness_is}}
\label{proof:lemma:asymptotic_unbiasedness_is}
\begin{proof}
Due to the ergodicity of MP-MCMC and since the asymptotic behaviour of the 
Markov chain is independent of its initial distribution
we may assume the stationary distribution $p$ on $(\vec{y}_{1:N+1},I)$ as initial
distribution. It follows the stationarity of the Markov chain, which implies,
\begin{align*}
\mathbb{E}_p\left[\boldsymbol{\mu}_L \right] &= \sum_{i=1}^{N+1}\mathbb{E}\left[ w_{i}f({\vec{y}}_i) \right]\\
&= \sum_{i=1}^{N+1}\int w_{i} f(\vec{y}_i) p(\vec{y}_{1:N+1}) \mathrm{d}\vec{y}_{1:N+1}
\\
&= \sum_{i=1}^{N+1}\int w_{i} f(\vec{y}_i) \sum_{j=1}^{N+1}
p(I=j) p(\vec{y}_{1:N+1}|I=j) \mathrm{d}\vec{y}_{1:N+1} 
\\
&= \sum_{i=1}^{N+1} \int w_{i} f({\vec{y}}_i)
\sum_{j=1}^{N+1}\frac{1}{N+1}  \pi({\vec{y}}_j) K({\vec{y}}_j, {\vec{y}}_{\setminus{j}})\mathrm{d}{\vec{y}}_{1:N+1}\\
&= \frac{1}{N+1}\sum_{i=1}^{N+1} \int \frac{\pi({\vec{y}}_i)K({\vec{y}}_i, \vec{y}_{\setminus{i})})}{\sum_{k=1}^{N+1}\pi(\vec{y}_k)K(\vec{y}_k, \vec{y}_{\setminus{k}})}f(\vec{y}_i) \sum_{j=1}^{N+1}\pi({\vec{y}}_j) K({\vec{y}}_j, {\vec{y}}_{\setminus{j}}) \mathrm{d}\vec{y}_{1:N+1} \\
&= \frac{1}{N+1}\sum_{i=1}^{N+1} \int f(\vec{y}_i) \pi(\vec{y}_i) \left(\int K(\vec{y}_i, \vec{y}_{\setminus{i}}) \mathrm{d}\vec{y}_{\setminus{i}}\right)\mathrm{d}\vec{y}_i\\
&=\frac{1}{N+1}\sum_{i=1}^{N+1}\int f(\vec{y}_i) \pi(\vec{y}_i)\mathrm{d}\vec{y}_i\\
&= \mathbb{E}_\pi \left[ f(\vec{y}) \right],
\end{align*}
where for the first and fourth equality we used stationarity, and in the penultimate line the kernel property. The statement follows now immediately by the ergodic theorem.
\end{proof}

\subsection{Asymptotic unbiasedness of the covariance estimate, Corollary
\ref{corollary:asymptotic_unbiasedness_is_covariance}}
\label{proof:corollary:asymptotic_unbiasedness_is_covariance}

\begin{proof}
Due to ergodicity, and since the asymptotic behaviour of the 
Markov chain is independent of its initial distribution we may set the 
stationary distribution $p$ on $(\vec{y}_{1:N+1},I)$ as initial
distribution. Stationarity of the chain follows, and thus, for any
 $j,k\in \{1,...,d\}$ we have
\begin{align*}
\mathbb{E}_p\left[ ({\Sigma}_L)_{j,k} \right] &= \frac{1}{L}\sum_{\ell=1}^L\mathbb{E}
\Bigg[ \sum_{i=1}^{N+1}w_{i}^{(\ell)}  
\bigg[
\left( (y_i^{(\ell)})_j - \mu_j \right)
-  
\Big( ({\mu}_{L})_j - \mu_j \Big) 
\bigg]
\\
&\quad\quad\quad\quad\quad
\cdot
\bigg[
\left( (y_i^{(\ell)})_k - \mu_k \right) 
-
\Big( ({\mu}_{L})_k - \mu_k \Big)
\bigg]
\Bigg]\\
&=\frac{1}{L}\sum_{\ell=1}^L \mathbb{E}\Bigg[\sum_{i=1}^{N+1}w_{i}^{(\ell)} 
\bigg\{
\Big((y_i^{(\ell)})_j- \mu_j  \Big)\Big((y_i^{(\ell)})_k - \mu_k  \Big)
\\ 
&\quad\quad
-\Big((y_i^{(\ell)})_j - \mu_j  \Big)\Big(({\mu}_{L})_k - \mu_k  \Big)
-\Big(({\mu}_{L})_j - \mu_j  \Big)\Big((y_i^{(\ell)})_k - \mu_k  \Big)
\\ 
&\quad\quad
+\Big(({\mu}_{L})_j - \mu_j  \Big) \Big(({\mu}_{L})_k - \mu_k \Big)
\bigg\}
\Bigg]
\\
&= \frac{1}{L}\sum_{\ell=1}^L\Bigg(
\mathbb{E}\Bigg[ \sum_{i=1}^{N+1} w_{i}^{(\ell)}\Big((y_i^{(\ell)})_j - \mu_j  \Big)\Big((y_i^{(\ell)})_k - \mu_k  \Big) \Bigg]
\\ 
&\quad\quad
-\mathbb{E}\bigg[ \Big(({\mu}_{L})_j - \mu_j  \Big) \Big(({\mu}_{L})_k- \mu_k \Big) \bigg]
\Bigg)
\\
&= \Cov_{\pi}\left(x_j,x_k\right) - \Cov_p\left(({\mu}_{L})_j, ({\mu}_{L})_k \right),
\end{align*}
where $x_i$ denotes the $i$th component of the random vector $\vec{x}\sim \pi$.
In the last line we applied Lemma \ref{lemma:asymptotic_unbiasedness_is}.
For $L\rightarrow \infty$, $\vec{\mu}_{L}$ converges to the constant
mean vector $\vec{\mu}$. Hence, for any $j,k\in \{1,...,d\}$,
\begin{align*}
\Cov\left(({\mu}_{L})_j, ({\mu}_{L})_k \right) \rightarrow 0 \quad \text{for } L\rightarrow \infty.
\end{align*}
Applying the ergodic theorem concludes the proof.
\end{proof}

\section{Proof of consistency of MP-QMCMC, Theorem \ref{theorem:consistency_multi_prop_mcmc}}
\label{app:sec:proof_consistency_mpqmcmc}

In the following, consistency of MP-QMCMC, as displayed in Algorithm 
\ref{algorithm:multiproposal_quasi_MH}, is proven. We assume the special case of
when the underlying
state space is one-dimensional, i.e.\ $d=1$. The general case can however be derived 
from this in a straightforward fashion. Before going to the proof of Theorem
\ref{theorem:consistency_multi_prop_mcmc}, we need a technical result 
on CUD points, which is similar to Lemma 6 in \cite{chen2011consistency}.

\begin{Lemma}[Auxiliary technicality]
\label{lemma:auxiliary_technicality_consistency_proof}
For $i\ge 1$ and $1\le j \le d$, let $v^i_j \in [0,1]$. For any $d,i,k \in \mathbb{N}$, let 
\begin{align*}
\vec{x}_i= (v_1^i, \ldots, v_{d}^i, \ldots,v_1^{i+k-1}, \ldots, v_{d}^{i+k-1})\in [0,1]^{dk}.
\end{align*}
If $v^i_j = v_{id+j}$, and $(v_i)_i$ is CUD, then $(\vec{x}_i)_{i}$ are uniformly distributed on $[0,1]^{dk}$ in the sense that 
\begin{align*}
\frac{1}{n} \sum_{i=1}^n \mathbbm{1}_{(\vec{a}, \vec{b}]} (\vec{x}_i) \rightarrow \operatorname{Vol}\left((\vec{a}, \vec{b}]\right),
\end{align*}
for any rectangular set $[\vec{a}, \vec{b}]\subset [0,1]^{dk}$.
\end{Lemma}

\begin{proof}
Let $\vec{c}\in [0,1]^{dk}$ and let $v:=\prod_{i=1}^{dk}c_i$ be the volume of $(\vec{0}, \vec{c}]$. For $r\in \mathbb{N}$ define $f_r: [0,1]^{d(k+r)} \rightarrow [0,\infty)$ by
\begin{align*}
f_r(\vec{w})&=f_r\left((w_1^{1}, \ldots, w_d^{1}, \ldots, w_1^{r+k}, \ldots, w_d^{r+k})\right) \\
&= \sum_{j=0}^{r-1} \mathbbm{1}_{(\vec{0}, \vec{c}]}\left((w_1^{j+1}, \ldots, w_d^{j+1}, \ldots, w_1^{j+k}, \ldots, w_d^{j+k})\right).
\end{align*}
The Riemann-integral of $f_r$ over $[0,1]^{d(k+r)}$ equals $rv$ as we sum up $r$ integral of $\mathbbm{1}_{(\vec 0, \vec c]}$ over the unit hypercube. Using $f_r$ on non-overlapping blocks of size $d(k+r)$ of $v^i_j$ yields
\begin{align}
\frac{1}{n} \sum_{i=1}^n \mathbbm{1}_{(\vec{0}, \vec{c}]} (\vec{x}_i) &= \frac{1}{n} \sum_{i=1}^n \mathbbm{1}_{(\vec{0}, \vec{c}]} \left((v_1^i, \ldots, v_{d}^i, \ldots,v_1^{i+k-1}, \ldots, v_{d}^{i+k-1})\right) 
\nonumber \\
\begin{split}
&\ge \frac{1}{n}\sum_{i=1}^{\lfloor n/r \rfloor} \sum_{j=0}^{r-1} \mathbbm{1}_{(\vec{0},\vec{c}]}\left((v_1^{(i-1)r + j +1}, \ldots, v_d^{(i-1)r + j +1}, \right. 
\end{split}
\label{eq:lower_estimate_auxlemma_proof}\\
&\quad \quad \quad\quad \quad \quad \quad \quad \left. \ldots, v_1^{(i-1)r + j +k}, \ldots, v_d^{(i-1)r + j +k}) \right)
\nonumber \\
&= \frac{1}{n} \sum_{i=1}^{\lfloor n/r \rfloor} f_r\left((v_1^{(i-1)(r+k)+1}, \ldots, v_d^{(i-1)(r+k)+1}, \right.
\nonumber \\
&\left. \quad \quad \quad\quad \quad \quad \quad \quad \ldots,v_1^{i(r+k)}, \ldots, v_d^{i(r+k)})  \right)
\nonumber \\
& \rightarrow v, \quad \text{as } n\rightarrow \infty, \nonumber
\end{align}
where in the second line we split up the sum from line one into $\lfloor r/n \rfloor$ segments of sums. For the convergence we used \eqref{eq:cud_convergence_non_overlapping}. Note that equality holds in \eqref{eq:lower_estimate_auxlemma_proof} if and only if $n/r \in \mathbb{N}$. From the calculation above we conclude
\begin{align*}
\liminf_{n\rightarrow \infty} \frac{1}{n}\sum_{i=1}^n \mathbbm{1}_{(\vec{0}, \vec{c}]}(\vec{x}_i) \ge v.
\end{align*}
and therefore
\begin{align}
\liminf_{n\rightarrow \infty} \frac{1}{n}\sum_{i=1}^n \mathbbm{1}_{(\vec{a}, \vec{b}]}(\vec{x}_i) \ge \operatorname{Vol}\left((\vec{a}, \vec{b}]\right),
\label{eq:liminf_bound_auxlemma_proof}
\end{align}
for any rectangular set $(\vec{a}, \vec{b}]\subset [0,1]^{dk}$. But this implies
\begin{align*}
\limsup_{n\rightarrow \infty} \frac{1}{n}\sum_{i=1}^n \mathbbm{1}_{(\vec{0}, \vec{c}]} (\vec{x}_i) &= 1 - \liminf_{n\rightarrow \infty} \frac{1}{n}\sum_{i=1}^n \mathbbm{1}_{(\vec{0}, \vec{c}]^{\operatorname{C}}}(\vec{x}_i) \\
&\le 1 - \operatorname{Vol}\left( (\vec{0}, \vec{c}]^{\operatorname{C}}\right)  \\
& = v,
\end{align*}
since $(\vec{0}, \vec{c}]^{\operatorname{C}}$ can be written as the union
\begin{align*}
(\vec{0}, \vec{c}]^{\operatorname{C}} = \bigcup_{i=1}^{dk} (\vec{a}_i, \vec{b}_i],
\end{align*}
of disjoint sets
\begin{align*}
(\vec{a}_i, \vec{b}_i] = [0,1]\times \ldots \times [0,1] \times [c_i, 1] \times (0, c_{i+1}] \times \ldots \times (0, c_{dk}].
\end{align*}
Hence, it also holds
\begin{align}
\limsup_{n\rightarrow \infty} \frac{1}{n}\sum_{i=1}^n \mathbbm{1}_{(\vec{a}, \vec{b}]} (\vec{x}_i) \le \operatorname{Vol}\left((\vec{a}, \vec{b}]\right).
\label{eq:limsup_bound_auxlemma_proof}
\end{align}
Combining equations \eqref{eq:liminf_bound_auxlemma_proof} and \eqref{eq:limsup_bound_auxlemma_proof} concludes the proof.
\end{proof}

\subsection{Proof of Theorem \ref{theorem:consistency_multi_prop_mcmc}}

\begin{proof}
Let $\varepsilon>0$. Let $m,n\in\mathbb{N}$ and for $i=1, \ldots, n$, let
\begin{align}
x^{i,m,0}_{1}, \ldots, x^{i,m,0}_{M}, \ldots, x^{i,m,m}_{1}, \ldots, x^{i,m,m}_{M} \in \Omega\subset \mathbb{R},
\end{align}
be the Rosenblatt-Chentsov transformation of $\vec{u}^i, \ldots, \vec{u}^{i+m}$. Let $f$ be a bounded and continuous function on $\mathbb{R}$. 
One can write
\begin{align*}
\int_\Omega f(x) \pi(x) \mathrm{d}x - \frac{1}{nN}\sum_{\substack{i=1,\ldots,n\\j=1,\ldots,M}} f(x^i_j) = \Sigma_1 + \Sigma_2 + \Sigma_3,
\end{align*}
where
\begin{align*}
\Sigma_1 &= \int_\Omega f(x) \pi(x) \mathrm{d}x - \frac{1}{nM}\sum_{i,j} f(x^{i,m,m}_{j}),\\
\Sigma_2 &= \frac{1}{nM}\sum_{i,j} \left(f(x^{i,m,m}_{j}) -f(x^{i+m}_j)\right),\\
\Sigma_3 &= \frac{1}{nM}\sum_{i,j} \left(f(x^{i+m}_j) -f(x^i_j)\right).
\end{align*}
Note that when driven by IID random numbers in $[0,1]$, MP-MCMC is assumed to be ergodic. It also satisfies the detailed balance condition. Hence, it samples from the stationary distribution $\pi$, that is, 
\begin{align*}
x^{i,m,m}_{j}\sim \pi \quad \forall j, \quad \text{ if }\quad (v^i_1, \ldots, v^i_{N+M}, \ldots v^{i+m}_1, \ldots, v^{i+m}_{N+M}) \sim \mathcal{U}[0,1]^{(m+1)(N+M)}.
\end{align*}
Lemma \ref{lemma:auxiliary_technicality_consistency_proof} with $d=N+M$ and $k=m+1$ implies
\begin{align}
\frac{1}{n}\sum_{i=1}^n \mathbbm{1}_{(\vec{a}, \vec{b}]}\left((v^i_1, \ldots, v^i_{N+M}, \ldots v^{i+m}_1, \ldots, v^{i+m}_{N+M})\right) \rightarrow \operatorname{Vol}\left( (\vec{a}, \vec{b}] \right), 
\label{eq:uniformity_proof_theorem_multi_mcqmc}
\end{align}
Using \eqref{eq:uniformity_proof_theorem_multi_mcqmc} and the MCMC regularity, it therefore holds
\begin{align*}
\frac{1}{nM}\sum_{i,j} f(x^{i,m,m}_{j}) \xrightarrow{ n \to \infty } \int_{\Omega} f(x) \pi(x) \mathrm{d}x.
\end{align*}
Hence, we have $\Sigma_1 \rightarrow 0$ as $n\rightarrow \infty$.\\
Considering $\Sigma_2$, note that the only non-zero terms in the sum arise when $x^{i,m,m}_{j}\neq x^{i+m}_j$. This case occurs whenever the coupling region $\mathcal{C}$ is avoided $m$ consecutive times, 
namely by $\vec{u}^{i+1}, \ldots, \vec{u}^{i+m}$. Then, 
\begin{align*}
(v^{i+1}_1, \ldots, v^{i+1}_{N+M}, \ldots, v^{i+m}_1, \ldots, v^{i+m}_{N+M})
\end{align*}
belongs to an area $A\subset [0,1]^{m(N+M)}$ of volume at most $\left(1-\operatorname{Vol}(\mathcal{C})\right)^{m}$. Thus,
\begin{align*}
\limsup_{n\rightarrow \infty} \Sigma_2 &\le \int_{A}\left[f(x^{i,m,m}_{j}(\vec{u})) -f(x^{i+m}_j(\vec{u}))\right] \mathrm{d}\vec{u}\\
&\le \left(1-\operatorname{Vol}(\mathcal{C})\right)^{m}\cdot \max_{x,x'\in \Omega} \left|f(x)-f(x') \right|\\
&< \varepsilon,
\end{align*}
for $m$ chosen sufficiently large and since the maximising term is bounded by assumption.\\
For $\Sigma_3$, it holds
\begin{align*}
\Sigma_3 \le \frac{1}{nM}\sum_{\substack{i=1,\ldots,n\\j=1,\ldots,M}} \left(f(x^{i+m}_j) -f(x^i_j)\right) \le 2 \frac{m}{n}\max_{x,x'\in \Omega} \left|f(x)-f(x') \right| \xrightarrow{ n \to \infty } 0.
\end{align*}
Combining the results for $\Sigma_1, \Sigma_2$ and $\Sigma_3$ yields
\begin{align*}
\left|\int_\Omega f(x) \pi(x) \mathrm{d}x - \lim_{n \rightarrow \infty}\frac{1}{nM}\sum_{\substack{i=1,\ldots,n\\j=1,\ldots,M}} f(x^i_j) \right|< \varepsilon,
\end{align*}
which concludes the proof since $\varepsilon>0$ was chosen arbitrarily.
\end{proof}

\end{document}